\documentclass[final,reqno]{elsarticle}
\let\today\relax
\makeatletter
\def\ps@pprintTitle{%
    \let\@oddhead\@empty
    \let\@evenhead\@empty
    \def\@oddfoot{\footnotesize\itshape
         {} \hfill\today}%
    \let\@evenfoot\@oddfoot
    }
\makeatother

\usepackage{mathrsfs}
\usepackage{amssymb,amsmath,amsfonts,latexsym}
\usepackage{amsthm}
\usepackage{graphicx,float,epsfig,color,fancyhdr}
\usepackage{framed}
\usepackage{todonotes}
\usepackage{comment}
\usepackage{multirow}
\usepackage{relsize}
\usepackage[utf8x]{inputenc}
\usepackage{scalerel,subfigure,booktabs}

\allowdisplaybreaks

\newtheorem{lemma}{Lemma}[section]
\newtheorem{remark}{Remark}[section]
\newtheorem{proposition}{Proposition}[section]

\newtheorem{theorem}{Theorem}[section]
\newtheorem{corollary}{Corollary}[section]
\newtheorem{problem}{Problem}

\def\CT{\mathcal{T}}
\def\E{K}
\def\G{\Gamma}

\def\LO{L^2(\O)}
\def\N{\mathbb{N}}
\def\O{\Omega}
\def\P{\mathbb{P}}
\def\R{\mathbb{R}}

\def\disp{\displaystyle}

\def\l{\lambda}
\def\PiK{\Pi_{K}^{\Delta}}

\def\sp{\mathop{\mathrm{sp}}\nolimits}

\def\Vh{V_h}

\def\HdsO{{H^{2+s}(\O)}}
\def\HdoK{{H^{2}(\E)}}
\journal{}
\date{\today}

\def\LO{L^2(\Omega)}

\def\Vh{V_h}

\def\N{\mathbb{N}}
\def\R{\mathbb{R}}

\newcommand{\dl}{~\textrm{d}l}

\newcommand{\dK}{~\textrm{d}K}

\DeclareMathOperator{\spec}{sp}

\begin{document}
\begin{frontmatter}

\title{A posteriori error estimates for a $C^1$ virtual element method applied to the thin plate vibration problem.}

\author[1]{Franco Dassi}
\ead{franco.dassi@unimib.it}
\address[1]{Dipartimento di Matematica e Applicazioni, Università degli studi di Milano Bicocca, Via Roberto Cozzi 55 - 20125 Milano, Italy.}
\author[2]{Andr\'es E. Rubiano}
\ead{andres.rubianomartinez@monash.edu}
\address[2]{School of Mathematics, Monash University, 9 Rainforest Walk, Melbourne, VIC 3800, Australia.}
\author[3]{Iv\'an Vel\'asquez}
\ead{ivan.velasquez@unimilitar.edu.co}
\address[3]{Departamento de Matem\'aticas, Universidad Militar Nueva Granada, Bogot\'a, Colombia.}

\begin{abstract} 
We propose and analyse residual-based a posteriori error estimates for the virtual element discretisation applied to the thin plate vibration problem in both two and three dimensions.  Our approach involves a conforming $C^1$ discrete formulation suitable for meshes composed of polygons and polyhedra. The reliability and efficiency of the error estimator are established through a dimension-independent proof. Finally, several numerical experiments are reported to demonstrate the optimal performance of the method in 2D and 3D.

\end{abstract}


\end{frontmatter}

%-----------------------------------------------------------------------
\setcounter{equation}{0}

\section{Introduction}\label{SEC:Introduction}
This paper focuses on the vibration problem of a thin plate modelled by the Kirchhoff-Love equations. In particular, given a bounded simply-connected Lipschitz domain $\Omega\subseteq \mathbb{R}^d$ ($d=2,3$)  with boundary $\Gamma:=\partial \Omega$. The deflection $u:\mathbb{R}^d\rightarrow \mathbb{R}$ and the vibration frequency $\omega>0$ satisfy the following eigenvaule problem, with $\lambda = \omega^2$.
\begin{align}\label{MPr}
  \begin{cases}
    \begin{aligned}
      \Delta^2 u              &= \lambda u,     && \text{in } \Omega, \\
      \mathcal{B}^{j}u   &= 0,             && \text{on } \Gamma,
    \end{aligned}
  \end{cases}
\end{align}
where, $\Delta^2 u := \Delta(\Delta u)$ denotes the fourth-order biharmonic operator. We also introduce the Hessian matrix, which consists of all second-order partial derivatives of $u$, denoted by $\nabla^2u$. 

On the other hand, the linear differential operator $\mathcal{B}^{j}$ ($j=1,2$) involves partial derivatives of the function $u$ (e.g. \cite[Section~2.3]{gazzola}).
Specifically, let $\boldsymbol{n}$ be the outward unit normal vector to $\Gamma$,  and $\partial_{\boldsymbol{n}}$ the normal derivative. We consider two types of homogeneous boundary conditions:
\begin{itemize}
\item {\bf Simply supported plate (SSP):}   
\begin{equation}\label{SSP}
\mathcal{B}^{j}u:=\Delta^{j-1}u=0\quad \mbox{on }\Gamma \quad \mbox{ for }j=1,2.
\end{equation}      
\item {\bf Clamped plate (CP):}   
\begin{equation}\label{CP}
\mathcal{B}^{j}u:=\partial_{\boldsymbol{n}}^{j-1}u=0\quad \mbox{on }\Gamma \quad \mbox{ for }j=1,2.
\end{equation}      
\end{itemize}

There is considerable interest in developing numerical methods for eigenvalue problems, among both practical applications \cite{BDMRS1995}, and theoretical analysis \cite{Boffi,BGG2012}. In particular, conforming finite element methods have been proposed to get the solution of this problem. However, due to the C$^1$ continuity requirement, it is well known that a high polynomial degree is required to achieve conformity in the discrete space \cite{ciarlet}, which leads to an important computational effort. To avoid this issue, we mention some alternatives applied to the thin plate vibration problem, such as: non-conforming methods \cite{BO,ciarlet,Ra}, discontinuous schemes \cite{ENGEL2002,GIANI2015}, and mixed formulations \cite{ciarlet1974mixed,monk1987,MoRo2009}.  

 The Virtual Element Method (VEM), first introduced in \cite{BBCMMR2013,BBMR2014}, has the flexibility to be applied to more general polygonal/polytopal meshes, and by construction, one can define high regular discrete spaces \cite{BM14}. Moreover, $C^1$-conforming VEM spaces necessitate only 3 degrees of freedom (DoFs) per vertex when $\Omega$ is a bidimensional domain, and 4 DoFs for the three-dimensional case, providing an efficient way to retain the accuracy of the method at low computational cost.  We provide a brief, non-exhaustive overview of VEM discretisations for equations involving the biharmonic operator \cite{ABSV2016,C1VEM_Polyhedral,BM13,MCS2021,MS2022}, and different eigenvalue problems \cite{gardini2,DV_camwa2022,gardini1,Meng2020,C0VEM2021,SteklovVEM2019,MRR2015,MV18,MVsiam2021,MRV2018}.

 Adaptive schemes driven by a posteriori error estimators allow optimal convergence recovery in cases such as singular solutions and complex geometries (e.g., non-convex domains and sharp corners), we refer to \cite{CARSTENSEN2024,feng2023,G2014,Li2018} for previous work on adaptive FEM for the biharmonic problem. The main advantage of VEMs is its natural handling of hanging nodes during adaptive mesh refinement algorithms, \cite{cangiani2017posteriori,dassi2025post,MORA2017,munar2024} provide some examples on adaptive VEMs.
 
 Our goal here is to derive a reliable and efficient a posteriori global error estimator $\eta$ for \eqref{MPr}, i.e. the following property holds for $\eta$: 
 $$C_1 (\text{error} + \text{oscillation terms}) \leq \eta \leq C_2 (\text{error}  + \text{oscillation terms} + \text{higher-order terms}).$$
 Typically, oscillation terms appear in virtual element methods due to the projection and stabilisation required to achieve the computability of virtual functions. On the other hand, high-order terms are negligible when $h\rightarrow 0$. 

\paragraph{Plan of the paper} The contents of this paper are organised as follows. The remainder of this section contains preliminary notational conventions. Section~\ref{Sec:ModPro_Stek} presents the continuous spectral formulation for the vibration problem and provides the spectral characterisation for the eigenvalues. The virtual element space, along with its spectral properties and error estimates, is provided in Section~\ref{SEC:DISCRETE}. Section~\ref{SEC:EST_A_POST} is devoted to deriving a reliable and efficient residual-type estimator. Finally, our theoretical results are illustrated via numerical examples in Section~\ref{sec:numerical-examples}.
 
\paragraph{Recurrent notation} Throughout this paper, we shall use standard notations for Sobolev spaces, norms and semi-norms as in \cite{adams2003sobolev}. The notation $a \lesssim b$ means that there exists a positive constant $C$ independent of the mesh parameter $h$, such that $a\leq C b$. We underline that the constant $C$ might stand for different values at its different occurrences.
In addition, for any linear bounded operator $\mathcal{S}:\,\mathcal{V}\to \mathcal{V}$, defined on a Hilbert space $\mathcal{V}$, we denote its spectrum by 
$\sp(\mathcal{S}):=\left\{z\in\mathbb{C}:\,\left(zI-\mathcal{S}\right)\mbox{ is not
	invertible}\right\}$ and by $\rho(\mathcal{S}):=\mathbb{C}\setminus\sp(\mathcal{\mathcal{S}})$ the resolvent
set of $\mathcal{S}$. Moreover, for any $z\in\rho(\mathcal{V})$,
$R_z(\mathcal{S}):=\left(zI-\mathcal{S}\right)^{-1}:\,\mathcal{V}\to \mathcal{\mathcal{V}}$ denotes the resolvent operator of $\mathcal{S}$ corresponding to $z$.

\section{Two continuous spectral formulations}
\label{Sec:ModPro_Stek}

In this section we introduce two continuous variational formulations associated with the vibration problems of a simply supported and clamped plate~\eqref{MPr}.

\begin{itemize}
    \item {\bf SSP:} Find $(\lambda,u)\in  \mathbb{R}\times \left(H^2(\Omega)\cap H_0^1(\Omega)\right)$, with $u\neq 0$ such that
\begin{equation}
    \int_{\Omega}\nabla^2 u: \nabla^2 v\ d\Omega=\lambda \int_{\Omega} u v\ d\Omega\qquad \forall v\in H^2(\Omega)\cap H_0^1(\Omega).\label{WForSSP}
\end{equation}
    \item {\bf CP:} Find $(\lambda,u)\in  \mathbb{R}\times H_0^2(\Omega)$, with $u\neq 0$ such that
\begin{equation}
    \int_{\Omega}\nabla^2 u: \nabla^2 v\ d\Omega=\lambda \int_{\Omega} u v\ d\Omega\qquad \forall v\in H_0^2(\Omega).\label{WForCP}
\end{equation}
\end{itemize}
For notational convenience, we introduce the following bilinear forms:
\begin{align}
	&a: H^2(\Omega)\times H^2(\Omega)  \to \mathbb{R},\qquad a(u,v):=\int_{\Omega}\nabla^2 u: \nabla^2 v\ d\Omega\,,\label{def_aD}\\
	&b:L^2(\Omega)\times L^2(\Omega)\to \mathbb{R},\qquad \quad b(u,v):=\int_{\Omega}u\,v\ d\Omega\,.\label{def_b0}
\end{align}
The following result is essential to characterize the spectrum of eigenvalue problems {\bf SSP} and {\bf CP}.
\begin{lemma}\label{Bound_Elli}
There exist positive constants $C$ and $\alpha_{\Omega}$ such that
\begin{align}
|a(u,v)|&\leq C||u||_{2,\Omega}||v||_{2,\Omega} 
&\forall u,v&\in H^2(\Omega),\label{bound_a}
\\
|b(u,v)|&\leq C||u||_{0,\Omega}||v||_{0,\Omega}
&\forall u,v&\in L^2(\Omega),\label{bound_b0}\\
a(v,v) &\geq \alpha_{\Omega}  ||v||_{2,\Omega}^2
&\forall v&\in  H^2(\Omega)\cap H_0^1(\Omega),\label{ellip_aSSP}\\
a(v,v) &\geq \alpha_{\Omega}  ||v||_{2,\Omega}^2
&\forall v&\in  H_0^2(\Omega)\label{ellip_aCP}.
\end{align}
\end{lemma}

\begin{proof}
Standard arguments in Sobolev spaces imply the estimates in \eqref{bound_a}   and~\eqref{bound_b0}. In addition, the proof of~\eqref{ellip_aSSP} and \eqref{ellip_aCP} are obtained from the Poincar\'e generalised inequality, and the fact that $\Vert \nabla^2 v\Vert_{0,\O}$ is equivalent to the usual norm in $H_0^2(\Omega)$, respectively.
\end{proof}

Since the bilinear form $a(\cdot,\cdot)$ is elliptic in both $H^2(\Omega)\cap H_0^1(\Omega)$ and $H_0^2(\Omega)$ spaces, then the analysis for the  problems~\eqref{WForSSP} and \eqref{WForCP} can be developed using same arguments. Therefore, from now on, we write $V$ to refer to both spaces. In what follows, we are interested in approximating the eigenpair of the following spectral problem.
\begin{problem}\label{WMPr} 
Find $(\lambda,u)\in \mathbb{R}\times V$, with $u\neq 0$ such that
\begin{equation*}
    a(u,v)=\lambda b(u,v)\qquad \forall v\in V.
\end{equation*}
\end{problem}
\subsection{The continuous source problem}\label{sec:contSolOp}
To deal with the a priori and a posteriori error analysis for the spectral  Problem~\ref{WMPr}, we define the following solution operators 
\begin{gather*}
T:V \rightarrow  V,
\quad  f \longmapsto w,\\
T^0:L^2(\Omega) \rightarrow  V\subseteq L^2(\Omega),
\quad f^0 \longmapsto w^0,
\end{gather*}
where $w:=Tf$ and $w^0:=T^0f$ are the two  unique solutions associated with their respective source problems
\begin{align}
a( w ,v)=b(f,v),\qquad&\forall v\in V,    \label{ContSourGPr}\\
a( w^0 ,v)=b(f^0,v),\qquad&\forall v\in V, \label{ContSourGPr0}
\end{align}
From Lemma~\ref{Bound_Elli} and the Lax-Milgram theorem, it is easy to check that the linear operators $T$ and $T^0$ are well defined and bounded. Moreover, the following results are standard in the literature of spectral theory of operators in Banach spaces (e.g. \cite{BO}).
\begin{proposition}
Given $(\lambda,u)\in\R\times V$, the following statements are equivalent:
\begin{itemize}
\item $(\lambda,u)\in\R\times V$ is a eigenpair of Problem~\ref{WMPr} with $u\neq 0$ and $\lambda\neq 0$;
\item $(\mu, u)\in\R\times V$ is a eigenpair of $T$, where $u\neq 0$ and $\mu:=1/\lambda.$
\end{itemize}
\end{proposition}
\begin{proposition}
The operators $T$ and $T^0$ are self-adjoint with respect to the inner products $a(\cdot,\cdot)$ in $V$ and $b(\cdot,\cdot)$ in $L^2(\Omega)$, respectively.
%\aer{[THE $\forall$ ON THE RIGHT CONFUSED ME, I HOPE THAT THIS CHANGE DO NOT MESS UP THE PROOF.]}
\label{prop:adjCont}
\end{proposition}

\begin{proof} Given $v,f\in V$, and $f^0,g^0\in L^2(\Omega)$, we readily see that
\begin{gather*}
a(T f,v)
=b (f,v)
=b (v,f)
=a(T v,f)
=a(f,T v), \\
b(T^0f^0,g^0)
	=b(g^0,T^0f^0)
	=a(T^0g^0,T^0f^0)
	=a(T^0f^0,T^0g^0)
	=b(f^0,T^0g^0).
\end{gather*}
\end{proof}

The following lemma gives additional regularity results for the solutions of source problem~\eqref{ContSourGPr} and, consequently, for the eigenfunctions of the operator $T$. For simplicity, we only show the results involving the operator $T$, recalling that similar arguments hold also for $T^0$.

\begin{lemma}\label{lem:uZero}
There exist $s\in(1/2,1]$ and a positive constant $C_{\Omega}$ depending on $\Omega$ such that: 
for all $f,f^0\in L^2(\Omega)$, the source solutions $ w $ and $w^0$ in the  equations~\eqref{ContSourGPr} and \eqref{ContSourGPr0}, respectively, 
satisfy $ w,w^0 \in H^{2+s}(\Omega)$
and 
\begin{align}
|| w ||_{2+s,\Omega} &\leq C_{\Omega} ||f||_{0,\Omega},\nonumber\\[1ex]
|| w^0 ||_{2+s,\Omega}& \leq C_{\Omega} ||f^0||_{0,\Omega}.\nonumber
\end{align}
\end{lemma}
\begin{proof}
Such an estimate follows from classical regularity results for the biharmonic problem; we refer the reader to \cite{BRM2AS80,G,savare98} for further details.
\end{proof}

\begin{lemma}
Let $(\lambda,u)$ an eigenpair  of the vibration of the eigenvalue Problem~\ref{WMPr}, then there exist $s>1/2$ and a positive  $C_{\Omega}^{\lambda}$ depending on $\Omega$ and $\lambda$ such that $u\in H^{2+s}(\Omega)$ and
$$
\left\|u\right\|_{2+s,\Omega}
\le C_{\Omega}^{\lambda}\left\|u\right\|_{2,\Omega}\,.
$$  
\label{lem:uStar}
\end{lemma}

\begin{proof}
The proof follows from Lemma~\ref{lem:uZero} with $f = \lambda u $.
\end{proof}

\begin{proposition}
The operator $T$ is compact. 
\end{proposition}

\begin{proof}
The compact inclusion $H^{2+s}(\Omega) \hookrightarrow V$, together with Lemma~\ref{lem:uZero}, imply the result.
\end{proof}

As consequence, we have the following spectral characterisation for the eigenvalues of the solution operators $T$.

\begin{theorem}
The spectrum of the operators $T$ and $T^0$ satisfy
$$
\spec(T)=\{0 \}\cup \{ \lambda_k\}_{k\in \mathbb{N}}\,=\spec(T^0), 
$$
where $\{ \lambda_k\}_{k\in \mathbb{N}}$ are sequences of real positive eigenvalues 
that converges to 0. 
Moreover, the multiplicity of each eigenvalue is finite.
\end{theorem}

%-----------------------------------------------------------------------
\setcounter{equation}{0}
\section{The virtual spectral approximation}
\label{SEC:DISCRETE}

Let $\left\{\Omega_h\right\}_h$ be a sequence of decompositions of $\O$
into polygons $\E$. Let $h_\E$ denote the diameter of the element $\E$
and  
$h:=\max_{\E\in\Omega_h}h_\E$. Similar to \cite{BBCMMR2013}, we  assume that there exists a positive constant $C_{\Omega}$ such that
\begin{itemize}
	\item[{\bf A1}:] the ratio between the shortest edge
	and the diameter $h_\E$ of $\E$ is larger than $C_{\Omega}$.
	\item[{\bf A2}:] $\E\in\Omega_h$ is star-shaped with
	respect to every point of a  ball
	of radius $C_{\Omega}h_\E$.
\end{itemize}

In what follows,
we denote by $N_K$ the number of vertices of $K$,
$l$ denotes a generic edge of $\Omega_h $
and for all $l\in \partial K$, we define a unit normal vector $\boldsymbol{n}_K^l$ 
that points outside of $\E$. Moreover, for any polygon $K\in \Omega_h$ we denote by $\mathcal{E}_K$ the set of edges of $K$ and $\mathcal{E}:=\cup_{K\in \Omega_h}\mathcal{E}_K$. We also introduce the sets $\mathcal{E}_{\G}:=\{l\in \mathcal{E}:l\subset \G \}$, and  $\mathcal{E}_{\O}:= \mathcal{E}\backslash \mathcal{E}_{\G}$. Finally, given $\mathcal{O}\subseteq \mathbb{R}^d$,
we denote the polynomial space in $d$-variables of degree lower or equal to $k$ as $\mathbb{P}_{k}(\mathcal{O})$.

Next, for all polygon $\E\in \Omega_h$  we consider the local virtual space
\begin{align*}
\widetilde{V}_K
:=\biggl\{v_h\in \HdoK : &\Delta^2 v_h\in\P_{2}(\E),\\
& v_h|_{\partial\E}\in C^0(\partial\E),\,
v_h|_l\in\P_3(l),\forall l\in\partial\E,\\
& \nabla v_h|_{\partial\E}\in C^0(\partial\E)^2,\,
\partial_{\boldsymbol{n}_K^l} v_h|_l\in\P_1(l), \forall l\in\partial\E\biggr\}.
\end{align*}
From here, the enhanced local virtual space is given by:
\begin{align*}
  V_h^K      
:=\left\{v_h\in\widetilde{V}_K : \int_{\E}(\PiK v_h)p_2\ dK=\int_{\E}v_hp_2\ dK,\,\forall p_2\in\P_{2}(\E)\right\},
\end{align*}
where, the polynomial projection operator $\Pi_K^\Delta:\widetilde{V}_K\rightarrow \mathbb{P}_2(K)\subseteq \widetilde{V}_K$ is defined as follows:
\begin{align*}
  \begin{cases}
    \begin{aligned}
      \int_K \nabla^2 (\Pi_{K}^{\Delta} v_h -v_h)\cdot \nabla^2 p_2\dK &= 0, && \forall p_2\in \mathbb{P}_2(K),\\
      \int_{\partial K} (\Pi_{K}^{\Delta} v_h -v_h)p_1 \dl &=0, && \forall p_1\in \mathbb{P}_1(K).
    \end{aligned}
  \end{cases}
\end{align*}
We are ready to present the global virtual space: for every
decomposition $\Omega_h$ of $\O$ into simple polygons $\E$, we define
$$
 V_h  :=\left\{v_h \in  V  :\ v_h|_{\E}\in  V_h^K      \right\}.
$$
It is well known (see \cite{BM13}) that any virtual function $v_h\in V_h^{\E}$ and  $\Pi_{\E}^{\Delta}v_h$ 
are uniquely determined by the following degrees of freedom:
\begin{align*}
	&{\bf D_1}: \text{Evaluation of $v_h$ at the vertexes of $\E$}, \\
	&{\bf D_2}: \text{Evaluation of $\nabla v_h$ at the vertexes of $\E$.}
\end{align*}
Notice that this definition implies that only 3 DoFs are required for this space. In order to write the discrete spectral problem we define 
\begin{align*}
a_h(u_h,v_h)
:=\sum_{\E\in\Omega_h}a_{h,\E}(u_h,v_h),
\qquad %u_h,v_h\in V_h  ,\\
b_h(u_h,v_h)
:=\sum_{\E\in\Omega_h}b_{h,\E}(u_h,v_h),
\qquad \forall u_h,v_h\in V_h  ,
\end{align*}
where $a_{h,\E}(\cdot,\cdot)$ and $b_{h,\E}(\cdot,\cdot)$
are the local discrete bilinear forms on
$  V_h^K      \times  V_h^K      $  given by
\begin{align}
\begin{aligned}
a_{h,\E}(u_h,v_h)
&:=a_{\E}\big(\PiK u_h,\PiK v_h\big)
+s_{\E}^{\Delta}\big(u_h-\PiK u_h,v_h-\PiK v_h\big),
&&\forall u_h,v_h\in  V_h^K      ,\\
b_{h,\E}(u_h,v_h)\label{locforma2}
&:=b_{\E}\big(\Pi_K^0 u_h,\Pi_K^0 v_h\big) +s_{\E}^0\big(u_h-\Pi_K^0 u_h,v_h-\Pi_K^0 v_h\big),
&&\forall u_h,v_h\in  V_h^K, 
\end{aligned}
\end{align}
with $\Pi_K^0:L^2(K) \to \mathbb{P}_2(K)$ is the standard $L^2-$projection and  $s_{\E}^{\Delta}$ and $s_{\E}^0$ are any symmetric positive bilinear forms satisfying  
\begin{align*}
\begin{aligned}
\alpha_{*}a_{\E}(v_h,v_h)&\leq s_{\E}^{\Delta}(v_h,v_h)\leq  \alpha^{*}a_{\E}(v_h,v_h), 
&&\forall v_h\in V_h^\E\cap\ker(\Pi_{\E}^{\Delta}),\\
\beta_{*}b_{\E}(v_h,v_h)&\leq s_{\E}^{0}(v_h,v_h)\leq \beta^{*}b_{\E}(v_h,v_h),
&&\forall v_h\in V_h^\E\cap\ker(\Pi_{\E}^{0}),
\end{aligned}    
\end{align*}
with $\alpha_{*},\alpha^{*},\beta_{*}$ and $\beta^{*}$ positive constants independent of $h_K$.
\begin{remark}
    The polynomial projections $\PiK v_h$ and $\Pi_K^0 v_h$  coincide in the lowest order case, for all $v_h\in V_h^K$(see \cite{MRV2018}).
\end{remark}
\begin{lemma}\label{consandstability}
    The following relations hold:
\begin{itemize}
\item \textit{Consistency}: for all $h > 0$ and for all $\E\in\Omega_h$, we have that
\begin{align*}
\begin{aligned}
a_{h,\E}(p_2,v_h)
&=a_{\E}(p_2,v_h),
&&\forall p_2\in\P_2(\E),
\quad\forall v_h\in  V_h^K,\\ %\label{consis-a}\\ 
b_{h,\E}(p_2,v_h)
&=b_{\E}(p_2,v_h),
&&\forall p_2\in\P_2(\E),
\quad\forall v_h\in  V_h^K.      %\label{consis-b}. 
\end{aligned}
\end{align*}
\item \textit{Stability and
boundedness}: The following estimates hold
\begin{align*}
\begin{aligned}
\min\{1, \alpha_{*}\}a_{\E}(v_h,v_h)&\leq a_{\E,h}(v_h,v_h)\leq \max\{1,\alpha^{*}\}a_{\E}(v_h,v_h), 
&&\forall v_h\in V_h(\E)\cap\ker(\Pi_{\E}^{\Delta}),\\
\min\{1, \beta_{*}\}b_{\E}(v_h,v_h)&\leq b_{\E,h}(v_h,v_h)\leq \max\{1,\beta^{*}\}b_{\E}(v_h,v_h),
&&\forall v_h\in V_h(\E)\cap\ker(\Pi_{\E}^{0}).
\end{aligned}    
\end{align*}
\end{itemize}
\end{lemma}

\begin{lemma}\label{Disc_BoundAndEllip}
The following estimates hold
\begin{align*}
\begin{aligned}
a_{h}(v_{h},v_{h})& \geq  \alpha_{\Omega}\min\{1,\alpha_{*} \}  ||v_{h}||_{2,\Omega}^2,
&&\forall v_{h}\in  V_{h},\\	
|a_{h}(u_{h},v_{h})|& \lesssim \max\{1,\alpha^{*}\} ||u_{h}||_{2,\Omega}||v_{h}||_{2,\Omega},
&&\forall u_{h},v_{h}\in V_{h},\\
|b(u_{h},v_{h})|& \lesssim \max\{1,\beta^{*}\}||u_{h}||_{0,\Omega}||v_{h}||_{0,\Omega},
&&\forall u_{h},v_{h}\in V_{h}.
\end{aligned}
\end{align*}
\label{Bound_Elli_h}
\end{lemma} 
%-------------------------------------------
\subsection{Discrete spectral formulation}

The virtual element discretization of problem~\eqref{MPr} can be read as follows.
\begin{problem}\label{DiscVP}
Find $(\lambda_h,u_h)\in \mathbb{R}\times V_h$, with $u_h\neq 0$ such that
\begin{equation*}
    a(u_h,v_h)=\lambda_h b_h(u_h,v_h)\qquad \forall v_h\in V_h.
\end{equation*}
\end{problem} 
\noindent Following Section~\ref{sec:contSolOp}, we introduce the discrete solution operator
\begin{equation*}
T_h:V_h \rightarrow  V_h\subseteq  V, \quad
 f_h \longmapsto w_h,
\end{equation*}
where $w_h:=T_hf_h$ is the unique solution of the  source problem
\begin{equation}\label{DiscreteSourceProblem}
a_{h} (w_{h},v_{h})=b_{h}(f_h,v_{h}) \qquad\forall v_{h}\in V_h\,.
\end{equation}
Since $f_h\in V_h$, 
the term $b_{h}( f_h,v_{h})$ is  computable from the degrees of freedom ${\bf D_1}$ and ${\bf D_2}$ for all $f_h\in V_h$.
Moreover, such discrete solution operators ``inherit'' some good properties of their continuous counterparts.

\begin{proposition}
The operator $T_h$ is well defined and uniformly bounded.
\end{proposition}

\begin{proof}
The proof follows from  Lemma~\ref{Disc_BoundAndEllip} and the Lax-Milgram theorem.
\end{proof}

\begin{proposition}
Given $(\lambda_h,u_h)\in\R\times V_h$, the following statements are equivalent:
\begin{itemize}
\item $(\lambda_h,u_h)\in\R\times V_h$ is a eigenpair of Problem~\ref{DiscVP} with $u_h\neq 0$ and $\lambda_h\neq 0$;
\item $(\mu_h, u_h)\in\R\times V_h$ is a eigenpair of $T_h$, where $u_h\neq 0$ and $\mu_h:=1/\lambda_h.$
\end{itemize}
\end{proposition}

\begin{proof}
Since the proof is standard in spectral theory, it is omitted; we refer the reader to~\cite{SunZhou2016} for more details.
\end{proof}

\begin{proposition}
$T_h$ is a self-adjoint operator with respect to the inner product $a_h(\cdot,\cdot)$. 
\end{proposition}

\begin{proof}
The arguments used in Proposition~\ref{prop:adjCont} imply the result.
\end{proof}

Finally, the spectral characterization for operator $T_h$ is summarised in the following result.

\begin{theorem}
Let $m_h$ be the dimension of the discrete space $\Vh$, 
then the spectrum of $T_h$ consists of $m_h$ real and positive eigenvalues repeated according to their multiplicities. 
\end{theorem}

\subsection{A priori error analysis}
This section focuses on presenting a measured error estimation result in the $L^2$-norm for the eigenfunctions of the vibration problem of a thin plate. This result will be important for the a posteriori error analysis in the following section.

We start  by recalling some well-known interpolation theorems commonly used in the literature on virtual element methods.

\begin{proposition}\label{app1}
If the assumption {\bf A2} is satisfied, then there exists a constant
$C>0$, such that for every $v\in H^{2+s}(\E)$ with $s\in(1/2,1]$,
there exists $v_{\pi}\in\P_{2}(\E)$ such that
$$\vert v-v_{\pi}\vert_{\ell,\E}\le Ch_{\E}^{2+s-\ell}\vert v\vert_{2+s,\E}, \quad \ell=0,1,2.$$
\end{proposition}
\begin{proposition}\label{app2}
Assume {\textbf{A1}--\textbf{A2}} are satisfied,
let $v\in\HdsO$ with $s\in(1/2,1]$. Then, there exist $v_{I}\in V_h  $
and $C>0$ such that
$$\Vert v-v_{I}\Vert_{2,\O}\le Ch^s\vert v\vert_{2+s,\O}.$$
\end{proposition}
\noindent In addition, we introduce the broken $H^{2}$-seminorm defined as
$$|v|_{2,h}^{2}:=\sum_{\E\in\Omega_h}|v|_{2,\E}^{2},$$
which is well defined for every $v\in L^{2}(\O)$ such that
$v|_{\E}\in H^{2}(\E)$ for all polygon $\E\in \CT_{h}$.

\begin{remark}
The classical spectral theory for compact operators developed in \cite{BO} is essential to prove the convergence in $V$-norm of $T_h \to T$. However, since stabilisation terms $s^0(\cdot,\cdot)$ needs the dofs ${\bf D_1}$ and ${\bf D_2}$ on the right hand side of \eqref{DiscreteSourceProblem} 
the operator $T_h$ is not well defined for all  $f\in V$.  To overcome this issue, we consider the orthogonal projection 
$\mathcal{P}_h:V\to V_h$ given by $f\mapsto \mathcal{P}_hf$ as the unique solution of the following system of equations
\begin{align*}
b(\mathcal{P}_hf-f,v_h)=0,\quad \forall v_h\in V_h.
\end{align*}
Moreover, for all $f\in V$ the following holds
\begin{align*}
|| f- \mathcal{P}_hf||_{0,\Omega}=\inf\limits_{v_h\in V_h} ||f-v_h ||_{0,\Omega}, \quad \forall v_h\in V_h.
\end{align*}
\end{remark}

Next, we introduce the projector $\widehat{T}_h:=T_h\mathcal{P}_h: V\to V_h\subseteq V$. Then, the spectrum of $\widehat{T}_h$ satisfies $\sp(\widehat{T}_h)=\sp(T_h) \cup \{0\}$ and the eigenfunctions of operators $\widehat{T}_h$ and $T$ are the same (e.g. \cite{MR2019}). Now, the following result establishes the convergence of  $\widehat{T}_h\to T$ and $\widehat{T}_h\to T^0$ when $h\to 0$.

\begin{lemma}\label{Lema4.11}
There exist a positive constant independent of the parameter $h$ and $s\in (1/2,1]$ such that
\begin{subequations}
\begin{alignat}{2}
|| ( T-\widehat{T}_h)f||_{2,\Omega}&\leq Ch^s || f||_{2,\O}, && \quad \forall f\in V,\label{PropP1_Hat_Th}\\
|| ( T^0-\widehat{T}_h)f||_{0,\O}&\leq Ch^s || f||_{0,\O}, && \quad \forall f\in \LO\label{PropP1_Hat_T0}.
\end{alignat}
\end{subequations}
\end{lemma}
\begin{proof}
The proof for \eqref{PropP1_Hat_Th} was established in \cite{MRV2018} (see also \cite[Proposition~5]{DV_camwa2022}). Now, let us prove \eqref{PropP1_Hat_T0}. Given $f\in\LO$, consider $T^0f=w^0\in   V  $ and $\widehat{T}_hf=T_h\mathcal{P}_hf=w_h\in  V_h  $ as the unique solutions of the source problems $a(w^0,v)=b(f,v) \ \forall v\in   V,$ (cf. \eqref{ContSourGPr0}) and $a_h(w_h,v_h)=b_h(\mathcal{P}_hf,v_h)\ \forall v_h\in  V_h $ (cf. \eqref{DiscreteSourceProblem}), respectively. Then, by using the same steps as those applied in \cite[Theorem~5.1]{DV_camwa2022} we obtain 
\begin{align}
|w^0-w_h|_{2,\O}^2&\leq b_h(\mathcal{P}_hf,v_h)-b(f,v)-\sum\limits_{K\in \Omega_h} \Big\{ a_{h,K}(w_I^0-w_{\pi}^0,v_h) +a_K(w_{\pi}^0-w^0,v_h) \Big\}\label{conv_L2_1}.
\end{align}
Note that
\begin{align}
b_h(\mathcal{P}_hf,v_h)-b(f,v)&=\sum\limits_{K\in \Omega_h} \Big\{b_{h,K}(\mathcal{P}_hf,v_h- (v_h)_{\pi} )-b_K(f,v_h)+b_{h,K}(\mathcal{P}_hf, (v_h)_{\pi} ) \Big\}\nonumber\\
&=\sum\limits_{K\in \Omega_h} \Big\{b_{h,K}(\mathcal{P}_hf,v_h- (v_h)_{\pi} )-b_K(\mathcal{P}_hf,v_h)+b_{h,K}(\mathcal{P}_hf, (v_h)_{\pi} ) \Big\}\nonumber\\
&=\sum\limits_{K\in \Omega_h} \Big\{b_{h,K}(\mathcal{P}_hf,v_h- (v_h)_{\pi} )-b_K(\mathcal{P}_hf,v_h- (v_h)_{\pi} ) \Big\}\nonumber\\
&\leq C\Big(\sum\limits_{K\in \Omega_h} ||\mathcal{P}_hf ||_{0,K}^2\Big)^{1/2}\Big(||v_h- (v_h)_{\pi}  ||_{0,K}^2\Big)^{1/2}\nonumber\\
&\leq C\Big(\sum\limits_{K\in \Omega_h} ||\mathcal{P}_hf ||_{0,K}^2\Big)^{1/2}h^2||v_h ||_{0,\O}\nonumber\\
&\leq Ch^2\Big(\sum\limits_{K\in \Omega_h} ||f ||_{0,K}^2\Big)^{1/2}||v_h ||_{0,\O}\nonumber\\
&\leq Ch^2||f ||_{0,\Omega}||v_h ||_{0,\Omega}\label{conv_L2_2}.
\end{align}
Moreover, \cite[Theorem~5.1]{DV_camwa2022} leads to
\begin{align}
&\sum\limits_{K\in \Omega_h} \Big\{ a_{h,K}(w_I^0-w_{\pi}^0,v_h) +a_K(w_\pi^0-w^0,v_h) \Big\}\leq Ch^s||f||_{0,\O}||v_h ||_{2,\O}.\label{conv_L2_3}
\end{align}
Thus, by including the estimates  \eqref{conv_L2_2} and \eqref{conv_L2_3} in right-hand side of \eqref{conv_L2_1} we get
\begin{align}\label{eq_conv_h2}
|| (\widehat{T}_h- T^0)f||_{0,\O}=||w^0-w_h ||_{0,\O}\leq||w^0-w_h ||_{2,\O}\leq |w^0-w_h|_{2,\O}\leq Ch^s|| f||_{0,\O}.
\end{align}
\end{proof}
\subsection{Spectral convergence}
We recall the definitions of the spectral projections $E$ and $E_h$, associated with pairs $(T,\mu)$,  and  $(T_h,\mu_h)$, respectively, which are defined as follows:
\begin{align*}
E:=E(\mu)=(2\pi i)^{-1}\int_{\mathcal{C}} (z-T)^{-1}dz,
\quad 
E_h:=E_h(\mu_h)=(2\pi i)^{-1}\int_{\mathcal{C}} (z-T_h)^{-1}dz.
\end{align*}
$E$ and $E_h$ are projections onto the space of generalized eigenvectors $R(E)$ and $R(E_h)$, respectively, where $R$ denotes the range. Moreover, we also consider the definition of the \textit{gap} $\hat{\delta}$ between the closed subspaces $\mathcal{M}$ and $\mathcal{N}$ of $L^2(\Omega)$ (see \cite[Chapter~II, Section~6]{BO}):
$$
\widehat{\delta}(\mathcal{M},\mathcal{N}) :=\max\left\{\delta(\mathcal{M},\mathcal{N}),\delta(\mathcal{N},\mathcal{M})\right\}, \text{ with }
\delta(\mathcal{M},\mathcal{N})
:=\sup_{\substack{\mathbf{x}\in\mathcal{M}, \left\|\mathbf{x}\right\|_{ 0,\Omega}=1}}
 \mathrm{dist}(\mathbf{x},\mathcal{N}).
 $$

As a consequence of these definitions, we can follow the same arguments as those applied to \cite[Section~5]{DV_camwa2022} to obtain the following results.

\begin{theorem}
    There exists a positive constant independent of the parameter $h$ such that
    \begin{subequations}
        \begin{align}
            \widehat{\delta}(R(E),R(E_h)) \leq C\gamma_h \nonumber\\
            |\mu-\mu_h^{(j)}|\leq C \gamma_h\nonumber,
        \end{align}
    \end{subequations}
    where 
 $\gamma_h:=\sup\limits_{\substack{f\in R(E)\\||f||_{2,\Omega}= 1}}\left\|(T -\widehat{T}_{h})f\right\|_{2,\Omega}$. 
\end{theorem}

\begin{theorem}
There exist $s> 1/2$  such that
\begin{equation}
\left\|(T -\widehat{T}_{h})f\right\|_{2,\Omega}
\lesssim  h^{\min\{s,1\} }\left\|f\right\|_{2,\Omega}\qquad \forall f\in R(E)\,,
\label{bou_gamma_h}    
\end{equation}
and, as a consequence,
\begin{equation}
\gamma_{h}  \lesssim  h^{\min\{s,1\} }. 
\label{bound1r}
\end{equation}
\label{gapr}
\end{theorem}

\begin{theorem}
Given a space $R(E)\subset H^{2+s}(\Omega)$, ${s}\in(1/2,1]$, 
there exists a positive constant $h_0$ independent of $h$, 
such that for all $ h<h_0$ we have 
\begin{subequations}\label{teo:mainConv}
\begin{align}
\left|\l-\l_h^{(j)}\right|&\lesssim  h^{2s},\qquad j=1,\ldots,m\,.\label{double_conv_1}\\[1ex]
\left|\l-\l_h^{(j)}\right|&\lesssim \Big( ||u-u_h||_{2,\Omega}^2 + ||u-\Pi_K^0 u_h||_{0,\Omega}^2 + |u- \PiK u_h|_{2,h}^2 \Big)\label{double_conv_2}.
\end{align}
\end{subequations}
\end{theorem}
\begin{proof}
Estimate \eqref{double_conv_1} was established in \cite[Theorem~4.5]{MRV2018}. The proof for \eqref{double_conv_2} can be follows repeating the same arguments as those used in \cite[Theorem~4.5]{MRV2018}, but in this case considering  $\PiK u_h$ instead of $w_{\pi}$ and using the fact  $\Pi_K^0v_h=\PiK v_h$ for all $v_h\in V_h$ .
\end{proof}

The following lemma establishes an $L^2$-norm error estimate for the source problem associated with the operator $\widehat{T}_h$, which will be useful in a posteriori error analysis.

\begin{lemma}\label{Lem410}
There exists a positive constant $C$ independent of parameter $h$ such that for all $f\in R(E)$, if $w=Tf$ and $w_h=\widehat{T}_hf=T_h\mathcal{P}_hf$, then
\begin{align*}
||w-w_h ||_{0,\Omega}
\leq C h^s\Big\{|w-w_h|_{2,\Omega}+|w-\PiK w_h|_{2,h} +||w-\Pi_K^0 u_h ||_{0,\Omega} \Big\}.
\end{align*} 
\end{lemma}
\begin{proof}
Let $f\in R(E)$, if $w=Tf$ and $w_h=\widehat{T}_hf=T_h\mathcal{P}_hf$. Consider $v\in  V  $ the unique solution of 
\begin{align}\label{eq_Aub_NIS}
a(v,z)=b(w-w_h,z)\qquad \forall z\in   V  .
\end{align}
From the definition of $T$ we have that $T(w-w_h)=v$, and using Lemma~\ref{lem:uZero} results $v\in  H^{2+s}(\O)$ satisfying 
\begin{align}\label{reg_of_v}
||v||_{2+s,\O}\leq C||w-w_h ||_{0,\O},
\end{align}
for some $s\in (1,1/2]$.

On the other hand we consider $v_I\in  V_h  $ like in the Proposition~\ref{app2}. Thus, testing the equation \eqref{eq_Aub_NIS} with $z\equiv w-w_h\in V$ we have
\begin{align}
&||w-w_h||_{0,\O}^2=b(w-w_h,w-w_h)=a(v,w-w_h)=a(w-w_h,v-v_I)+a(w-w_h,v_I)\nonumber\\
&\leq |w-w_h|_{2,\O}|v-v_I|_{2,\O}+ a(w-w_h,v_I)\nonumber\\
&\leq Ch^s|w-w_h|_{2,\O}|v|_{2+s,\O} + \Big\{a_h(w_h,v_I)-a(w_h,v_I)+b(f,v_I)-b_h(\mathcal{P}_hf,v_I)\Big\}\nonumber\\
&\leq Ch^s|w-w_h|_{2,\O}||w-w_h||_{0,\O} + \Big\{a_h(w_h,v_I)-a(w_h,v_I)+b(f,v_I)-b_h(\mathcal{P}_hf,v_I)\Big\}.\label{conv_norL2_1}
\end{align}
Next, note that using the properties of consistence and stability for $a_h(\cdot,\cdot)$, Proposition~\ref{app2}, \eqref{reg_of_v}, and Proposition~\ref{lem:uZero}, we obtain
\begin{align}
&a_h(w_h,v_I)-a(w_h,v_I)=\sum\limits_{K\in \Omega_h} \Big\{ a_{h,K}(w-\PiK w_h,v_I-v_{\pi})+a(w_h-\PiK w_h,v_I-v_\pi)\Big\}\nonumber\\
&=\Big\{\sum\limits_{K\in \Omega_h}|w_h-\PiK w_h|_{2,K}^2 \Big\}^{1/2} \Big\{\sum\limits_{K\in \Omega_h}|v_I - v_\pi|_{2,K}^2 \Big\}^{1/2}\nonumber\\
&\leq \Big\{\sum\limits_{K\in \Omega_h}|w-w_h|_{2,K}^2+|w-\PiK w_h|_{2,K}^2 \Big\}^{1/2} \Big\{\sum\limits_{K\in \Omega_h}|v-v_I|_{2,K}^2+|v - v_\pi|_{2,K}^2 \Big\}^{1/2}\nonumber\\
&\leq Ch^s\Big\{\sum\limits_{K\in \Omega_h}|w-w_h|_{2,K}^2+|w-\PiK w_h|_{2,K}^2 \Big\}^{1/2}|v|_{2+s,\O}\nonumber\\
&\leq Ch^s\Big\{\sum\limits_{K\in \Omega_h}|w-w_h|_{2,K}^2+|w-\PiK w_h|_{2,K}^2\Big\}^{1/2}||w-w_h||_{0,\O}\nonumber\\
&\leq Ch^s\Big\{|w-w_h|_{2,\O}+|w-\PiK w_h|_{2,h}\Big\}||w-w_h||_{0,\O}.\label{conv_norL2_2}
\end{align}
Besides, using the fact that $f\in R(E)$, $w=Tf=\mu f$, the definition of $\mathcal{P}_h$, the properties of consistence and stability of $b(\cdot,\cdot)$, \eqref{reg_of_v}, and  Lemma~\ref{lem:uZero} we deduce
\begin{align}
&b(f,v_I)-b_h(\mathcal{P}_hf,v_I)=b(\mathcal{P}_hf,v_I)-b_h(\mathcal{P}_hf,v_I)=\sum\limits_{K\in \Omega_h} \Big\{b_K(f,v_I) -b_{h,K}(\mathcal{P}_hf,v_I) \Big\}\notag\\
&=\sum\limits_{K\in \Omega_h} \Big\{b_K(\mathcal{P}_hf-\mu^{-1}\PiK w_h,v_I-v_{\pi}) -b_{h,K}(\mathcal{P}_hf-\mu^{-1}\PiK w_h,v_I-v_{\pi}) \Big\} \notag\\
&\leq\Big\{\sum\limits_{K\in \Omega_h}||\mathcal{P}_hf-\mu^{-1}\PiK w_h ||_{0,K}^2 \Big\}^{1/2}\Big\{\sum\limits_{K\in \Omega_h}(||v-v_I ||_{0,K}^2 + || v-v_{\pi}||_{0,K}^2 )\Big\}^{1/2}\notag\\
&\leq \Big\{\sum\limits_{K\in \Omega_h}||\mathcal{P}_hf-\mu^{-1}\PiK w_h ||_{0,K}^2 \Big\}^{1/2}\Big\{\sum\limits_{K\in \Omega_h}(h^4||v||_{2+s,K}^2 + h^{2(2+s)}|| v||_{2+s,K}^2 )\Big\}^{1/2}\notag\\
&\leq Ch^2\Big\{\sum\limits_{K\in \Omega_h}||\mathcal{P}_hf-\mu^{-1}\PiK w_h ||_{0,K}^2 \Big\}^{1/2}|v|_{2+s,\O}\notag\\
&\leq Ch^2\mu^{-1}\Big\{\sum\limits_{K\in \Omega_h}||\mathcal{P}_hw-\PiK w_h ||_{0,K}^2 \Big\}^{1/2}||w-w_h||_{0,\O}\notag\\
&\leq Ch^2\mu^{-1}\Big\{\sum\limits_{K\in \Omega_h}(||w-\mathcal{P}_hw||_{0,K}^2+||w-\PiK w_h ||_{0,K}^2) \Big\}^{1/2}||w-w_h||_{0,\O}\notag\\
&\leq Ch^2\Big\{||w-\mathcal{P}_hw||_{0,\O}+||w-\PiK w_h ||_{0,\Omega} \Big\}||w-w_h||_{0,\O}\notag\\
&\leq Ch^2\Big\{||w-w_h||_{0,\O}+||w-\PiK w_h ||_{0,\Omega} \Big\}||w-w_h||_{0,\O}.\label{conv_norL2_3}
\end{align}

Therefore, inserting the estimates \eqref{conv_norL2_2} and \eqref{conv_norL2_3} in~\eqref{conv_norL2_1}, and multiplying by  $||w-w_h||_{0,\Omega}^{-1}$ we have
\begin{align*}
||w-w_h ||_{0,\O}&\leq C h^s\Big\{|w-w_h|_{2,\O}+|w-\PiK w_h|_{2,h}\Big\}+Ch^2\Big\{||w-w_h||_{0,\O}+||w-\PiK u_h ||_{0,\Omega} \Big\}\nonumber\\
&\leq C h^s\Big\{|w-w_h|_{2,\Omega}+|w-\PiK w_h|_{2,h} +||w-\PiK u_h ||_{0,\Omega} \Big\}.
\end{align*}
\end{proof}

\subsection{Error estimates on $L^2$-norm for the eigenfunctions}

In this section, we present an error estimation result for eigenfunctions in the $L^2$-norm that will be used to demonstrate reliability in a posteriori error estimation. To this end, we first introduce the proof of an auxiliary lemma.
\begin{lemma}\label{Lema4.12} Let $(\mu_h^{(j)},u_h)$ be an eigenpair of $\widehat{T}_h$  with $j=1,2,...,m$ and $||u_h ||_{0,\O}=1$. Then, there exist an eigenfunction $u\in \LO$ of $T$ associated to $\mu$, and a positive constant $C$ such that
\begin{align}
||u-u_h ||_{0,\O}\leq  Ch^s\Big\{|w-w_h|_{2,\O}+||w-\Pi_K^0 w_h ||_{0,\Omega}+|w-\PiK  w_h|_{2,h}\Big\} \label{aux_pre_th}
\end{align}  
\end{lemma}
\begin{proof}
From \eqref{PropP1_Hat_T0} and the classical theory for compact operators presented in \cite{BO} we know that $sp(\widehat{T}_h)$ converges to $sp(T^0)$. Moreover, from the relation between the eigenfunctions of $T$ and $T_h$ with those of $T^0$ and $\widehat{T}_h$, respectively, we obtain that $u_h\in R(E_h)$ and there exists $u\in R(E)$ such that \begin{align*}
||u-u_h ||_{0,\O}\leq C \sup\limits_{\substack{f^0\in R(E^0),||f^0 ||_{0,\O}=1}}||(T^0-\widehat{T}_h)f^0 ||_{0,\O}.
\end{align*} 
In addition, we have from Lemma~\ref{Lem410} that for each $f^0\in R(E^0)$ satisfying $f^0=f$ implies
\begin{align*}
||(T^0-\widehat{T}_h)f^0 ||_{0,\O}=||(T-\widehat{T}_h)f ||_{0,\O}\leq C h^s\Big\{|w-w_h|_{2,\Omega}+|w-\PiK w_h|_{2,h} +||w-\PiK u_h ||_{0,\Omega} \Big\}.
\end{align*}
Putting together about inequalities finish the proof.
\end{proof}

Now, we prove the $L^2$-norm error estimate for the eigenfunctions of the virtual spectral discretization analysed in this work.

\begin{theorem}\label{Teor4.9}
There exist $s\in(1/2,1]$ and a positive constant $C$ independent of parameter $h$ such that
\begin{align*}
||u-u_h ||_{0,\O}\leq Ch^s\Big\{|u-u_h|_{2,\O}+||u-\Pi_K^0 u_h ||_{0,\Omega}+|u-\PiK  u_h|_{2,h}\Big\}.
\end{align*}
%with $s$ like in Lemma~\ref{lem:uZero}.
\end{theorem}
\begin{proof}
The proof proceeds by bounding each term on the right-hand side of  \eqref{aux_pre_th} in  Lemma~\ref{Lema4.12}. We first consider $w\in V$ and $w_h\in V_h$ as the solutions of problems
\begin{align}
a(w,v)=b(u,v), \quad
a_h(w_h,v_h)=b_h(\mathcal{P}_hu,v_h), \quad \forall v\in V,\, \forall v_h \in V_h \label{aux_disc_prob_apost},
\end{align}
respectively. Since $b(u,v)=\frac{1}{\lambda}a(u,v)$ (cf. Problem~\ref{WMPr}) we get $w=\frac{1}{\lambda}
u$ with $\lambda\neq 0$ . Thus, 
\begin{align}
|w-w_h|_{2,\O}&\leq \lambda^{-1}|u-u_h|_{2,\O}+|\lambda^{-1}-\lambda_h^{-1}||u_h|_{2,\O}+|\lambda^{-1}u_h-w_h|.\label{aux_1ap} 
\end{align}
Note that, from estimate~\ref{double_conv_2} we obtain
\begin{align}
|\lambda^{-1}-\lambda_h^{-1}|=|\lambda\lambda_{h}^{-1}||\lambda-\lambda_h|\leq C\Big( ||u-u_h||_{2,\Omega}^2 + ||u-\Pi_K^0 u_h||_{0,\Omega}^2 + |u- \PiK u_h|_{2,h}^2 \Big)\label{est_chu}.
\end{align}

On the other hand, we know that $a_h(\lambda^{-1}u_h,v_h)=b_h(u_h,v_h)\ \forall v_h \in V_h $ (cf. Problem~\ref{DiscVP}). Hence, subtracting this equality from \eqref{aux_disc_prob_apost} we obtain 
\begin{align*}
a_h(w_h-\lambda^{-1}u_h,v_h)=b_h(\mathcal{P}_hu-u_h,v_h) \ \forall v_h \in V_h .
\end{align*}
Therefore, using the fact that $a_h(\cdot,\cdot)$ is $V_h-$elliptic, we can deduce
\begin{align}
|w_h-\lambda^{-1}u_h|_{2,\O}^2&=a_h(w_h-\lambda^{-1}u_h,w_h-\lambda^{-1}u_h)=Cb_h(\mathcal{P}_hu-u_h,w_h-\lambda^{-1}u_h)\nonumber\\
&\leq C || \mathcal{P}_hu-u_h||_{0,\O}||w_h-\lambda^{-1}u_h ||_{0,\O}\leq C || \mathcal{P}_h(u-u_h)||_{0,\O}||w_h-\lambda^{-1}u_h ||_{0,\O}\nonumber\\
&\leq C ||u-u_h||_{0,\O} ||w_h-\lambda^{-1}u_h ||_{2,\O}\leq C ||u-u_h||_{2,\O} |w_h-\lambda^{-1}u_h |_{2,\O}\nonumber\\
&\leq C |u-u_h|_{2,\O} |w_h-\lambda^{-1}u_h |_{2,\O},\nonumber
\end{align}
which implies that
\begin{align}
|w_h-\lambda^{-1}u_h|_{2,\O}&\leq C |u-u_h|_{2,\O} .\label{cua_chu}
\end{align}
Then, replacing \eqref{est_chu} and \eqref{cua_chu} in the right-hand side of \eqref{aux_1ap} lead to
\begin{align}
|w-w_h|_{2,\O}\leq C\Big( ||u-u_h||_{2,\Omega} + ||u-\Pi_K^0 u_h||_{0,\Omega} + |u- \PiK u_h|_{2,h}\Big)\label{1stToBound}.
\end{align}

To bound the second term on the right-hand side of \eqref{aux_pre_th} we apply the triangle inequality and we have
\begin{align}
||w-\Pi_K^0 w_h||_{0,\Omega}&\leq |w-w_h|_{2,\O}+||w_h-\Pi_K^0  w_h||_{0,\Omega}\label{like_c.vs_t}.
\end{align}
Next, for the second term in \eqref{like_c.vs_t}, we obtain
\begin{align}
||w_h-\Pi_K^0  w_h||_{0,\Omega}&\leq ||w_h-\lambda_h^{-1}u_h||_{0,\O}+\lambda_h^{-1}||u_h-\Pi_K^0 u_h||_{0,\Omega}+||\Pi_K^0(\lambda_h^{-1}u_h-w_h)||_{0,\Omega}\nonumber\\
&\leq 2|w_h-\lambda_h^{-1}u_h|_{2,\O}+\lambda_h^{-1}||u-u_h||_{0,\O}+\lambda_h^{-1}||u-\Pi_K^0 u_h||_{0,\Omega}\nonumber\\
&\leq C\Big\{||u-u_h||_{0,\O}+||u-\Pi_K^0 u_h||_{0,\Omega}\Big\}\label{car_fel}.
\end{align}
Thus, replacing  \eqref{1stToBound} and \eqref{car_fel} in \eqref{like_c.vs_t} imply that
\begin{align}
||w-\Pi_K^0  w_h||_{0,\Omega}\leq C\Big\{|u-u_h|_{2,\O}+||u-\Pi_K^0  u_h||_{0,\Omega}+|u-\PiK  u_h|_{2,h} \Big\}\label{2ndTermToBound}.
\end{align}

The third term on the right-hand side of \eqref{aux_pre_th} is bounded by following the same arguments as those applied to obtain \eqref{2ndTermToBound}. Therefore,
\begin{align}
|w-\PiK w_h|_{2,h}&\leq |w-w_h|_{2,\Omega}+|w_h-\PiK   w_h|_{2,h}=|w-w_h|_{2,\Omega}+\sum\limits_{K\in \Omega_h}|w_h-\PiK   w_h|_{2,K}\nonumber\\
&\leq C\Big\{|u-u_h|_{2,\Omega}+ ||u-u_h||_{2,\O}+|u-\PiK u_h|_{2,h}\Big\}\label{3TermToBound}
\end{align}
Finally, substituting  \eqref{1stToBound}, \eqref{2ndTermToBound} and \eqref{3TermToBound} on the right-hand side of \eqref{aux_pre_th} concludes the proof.
\end{proof}

%-----------------------------------------------------------------------
\setcounter{equation}{0}
\section{A posteriori error analysis}
\label{SEC:EST_A_POST}
This section is dedicated to constructing a residual-type estimator that depends only on quantities available from the VEM solution, and to establishing its reliability and efficiency. To achieve this, we first present some definitions and local estimates.

\begin{lemma}\label{lema4.2gatica2010analysis}
	Let $l,m\in \N\cup\{0\}$ such that $l\leq m$. Then, for any polygon $K$, there exists $C>0$, depending only on $k,l,m$ and the shape of $K$, such that
	\begin{equation}\label{eq4.4barrios2006residual}
	|q|_{m,K}\leq c h_K^{n-m}|q|_{n,K}\qquad \forall q\in \mathbb{P}_k(K).
	\end{equation}
\end{lemma}
To construct a suitable residual-based error estimator, we precise the following lemma.
\begin{lemma}\label{eq-residual}
	For any $v\in  V $, the following result is obtained
	 \begin{align*}
	 a(e,v)	 &= \lambda b(u,v)-\lambda_h b( u_h,v) + \sum\limits_{K\in \Omega_h} \Big\{ a_K(\PiK u_h-u_h,v)+ \lambda_h  b_{K}( u_h-\PiK u_h,v) \Big\}\\
     &+ \sum\limits_{K\in \Omega_h}   \int_{K}\lambda_h\PiK u_hv  - \sum\limits_{K\in \Omega_h}\sum_{l\in \mathcal{E}_K\cap \mathcal{E}_{\O}}\int_l [\![(\nabla^2\PiK u_h )\boldsymbol{n}_K^l]\!]\cdot \nabla v.
	 \end{align*}
\end{lemma}
\begin{proof}
Given any $v\in  V $ and using the fact that $(\lambda,u)$ solves the problem~\eqref{WMPr} and integrating by parts, one has
\begin{align*}
&a(e,v)%=a(u-u_h,v)=a(u,v)-a(u_h,v)
=\lambda b(u,v)-\lambda_h b(u_h,v) + \lambda_h b(u_h,v)-a(u_h,v)\nonumber\\
&= \lambda b(u,v)-\lambda_h b(u_h,v)  + \sum\limits_{K\in \Omega_h} \Big\{a_{K} (\PiK u_h-u_h,v) + \lambda_h  b_{K}(u_h,v) -  a_{K}(\PiK u_h,v)\Big\}\nonumber\\
&= \lambda b(u,v)-\lambda_h b(u_h,v)  + \sum\limits_{K\in \Omega_h} \Big\{a_{K} (\PiK u_h-u_h,v) + \lambda_h  b_{K}(u_h-\PiK u_h,v)\nonumber\\
&\quad + \lambda_h  b_{K}(\PiK u_h,v)\Big\} - \sum\limits_{K\in \Omega_h} \Bigg\{ \int_{\partial K} \Big[\nabla^2(\PiK u_h)\boldsymbol{n}_{K}\Big]\cdot \nabla v - \int_{\partial K}v \Big[ \mathrm{div} \nabla^2 \PiK u_h\Big]\cdot \boldsymbol{n}_K \Bigg\} \nonumber\\
&= \lambda b(u,v)-\lambda_h b( u_h,v) + \sum\limits_{K\in \Omega_h} \Big\{ a_K(\PiK u_h-u_h,v)+ \lambda_h  b_{K}( u_h-\PiK u_h,v) \Big\}\\
&\quad + \sum\limits_{K\in \Omega_h}   \int_{K}\lambda_h\PiK u_hv  - \sum\limits_{K\in \Omega_h}\sum_{l\in \mathcal{E}_K\cap \mathcal{E}_{\O}}\int_l [\![(\nabla^2\PiK u_h )\boldsymbol{n}_K^l]\!]\cdot \nabla v.\nonumber
\end{align*}
\end{proof}
We define the following local error, volume, jump, and stabilisation estimators for all $K\in \Omega_h$.
\begin{gather*}
\eta_K^2 := \Xi_K^2 + \sum_{l\in \mathcal{E}_K\cap \mathcal{E}_{\O}}  \mathcal{J}_{l}^2 + S_K^2,\\
\Xi_{K}^2 := h_K^4 ||\lambda_h\PiK u_h ||_{0,K}^2, \quad  \mathcal{J}_{l}^2 := h_{l}
||[\![(\nabla^2\PiK u_h )\boldsymbol{n}_K^{l}]\!] ||_{0,l}^2,\\
S_{K}^2
:= s_K^{\Delta}(u_h-\PiK u_h,u_h-\PiK u_h) + s_{K}^0(u_h-\Pi_K^0 u_h,u_h-\Pi_K^0 u_h)
\end{gather*}
In addition, we note that $\mathcal{J}_l$ is computable only on the basis of the output values of the operators in ${\bf D_2}$. Finally, our respective global error, volume, jump, and stabilisation estimators are defined as follows:
\begin{align}\label{global_estimators}
\eta^2:=\sum\limits_{K\in \Omega_h} \eta_K^2, \quad \Xi^2 := \sum\limits_{K\in \Omega_h} \Xi_K^2, \quad \mathcal{J}^2 := \sum\limits_{K\in \Omega_h} \sum_{l\in \mathcal{E}_K\cap \mathcal{E}_{\O}} \mathcal{J}_{l}^2, \quad \text{and} \quad S^2 := \sum\limits_{K\in \Omega_h} S_K^2.
\end{align}

\subsection{Reliability}\label{Reliablity}
The following result is important to establish the reliability of the a posteriori error estimator.
\begin{theorem}\label{Th5.6}
	There exists a positive constant $C$ independent of parameter $h$ such that
	\begin{equation}
		|u-u_h|_{2,\Omega}\leq C \Big\{ \eta + \frac{\lambda+\lambda_h}{2}||u-u_h||_{0,\O} \Big\}
	\end{equation}
\end{theorem}
\begin{proof}
Let $e=u-u_h\in  V_h  \subseteq  V $, then there exists $e_I\in  V_h$ (cf. Proposition~\ref{app2}) such that
\begin{equation}\label{est_pol_error}
		|e-e_I |_{t,\O}\leq Ch^{s-t}|e |_{s,\O},\quad s=2,3, \quad t=0,1,...,s	
	\end{equation}
From Lemma~\ref{eq-residual}, we have that
\begin{align}
&|u-u_h|_{2,\O}^2%=a(u-u_h,e)\nonumber\\
=a(e,e-e_I)+a(u,e_I)-a_h(u_h,e_I)+a_h(u_h,e_I)-a(u_h,e_I)\nonumber\\
&= \lambda b(u,e-e_I)-\lambda_h b( u_h,e-e_I) +a(u,e_I)-\lambda_h b_h(u_h,e_I)+a_h(u_h,e_I)-a(u_h,e_I) \nonumber \\
&\quad + \sum\limits_{K\in \Omega_h} \int_{K}\lambda_h\PiK u_h(e-e_I) +\sum\limits_{K\in \Omega_h}\sum_{l\in \mathcal{E}_K\cap \mathcal{E}_{\O}}\int_l - [\![(\nabla^2\PiK u_h )\boldsymbol{n}_K^l]\!] \cdot \nabla (e-e_I)\nonumber\\
&\quad +\sum\limits_{K\in \Omega_h} a_K(\PiK u_h-u_h,e-e_I)+ \lambda_h\sum\limits_{K\in \Omega_h} b_{K}( u_h-\PiK u_h,e-e_I)\nonumber\\
%&=\lambda b(u,e)-\lambda b(u,e_I)-\lambda_h b(u_h,e)+\lambda_h b(u_h,e_I)+a(u,e_I)-\lambda_h b_h(u_h,e_I)+a_h(u_h,e_I)-a(u_h,e_I)\nonumber\\
%&\quad +\sum\limits_{j=1}^{4}E_j \nonumber\\
&=\Big\{ \lambda b(u,e)-\lambda_h b(u_h,e)\Big\}+\lambda_h\Big\{b(u_h,e_I)-b_h(u_h,e_I) \Big\}+\Big\{ a_h(u_h,e_I)-a(u_h,e_I)\Big\}+\sum\limits_{j=1}^{4}E_j,\label{err_ap}
\end{align}
where the terms $E_j$ ($j=1,\dots,4$) are given by: 	
    \begin{align*}
        \begin{aligned}
			\disp E_1&:=\sum\limits_{K\in \Omega_h} a_K(\PiK u_h-u_h,e-e_I),\quad \disp &&E_2:=\lambda_h\sum\limits_{K\in \Omega_h} b_{K}( u_h-\PiK u_h,e-e_I),\\
			\disp E_3&:=\sum\limits_{K\in \Omega_h} \int_{K}\lambda_h\PiK u_h(e-e_I),\quad  \disp &&E_4:=\sum\limits_{K\in \Omega_h}\sum_{l\in \mathcal{E}_K\cap \mathcal{E}_{\O}}\int_l - [\![(\nabla^2\PiK u_h )\boldsymbol{n}_K^l]\!] \cdot \nabla (e-e_I).
        \end{aligned}
\end{align*}
	
Now, let us find upper bounds on each term in \eqref{err_ap}. In fact, if $(\lambda,u)\neq (0,0)$ is an eigenpair of $T$ with $\lambda$ a simple eigenvalue, we select $|| w||_{0,\O}=1$ so that for each mesh $\Omega_h$, the solution $(\lambda_h,u_h)$ of the Problem~\ref{DiscVP} satisfy $||u_h ||_{0,\O}=1$ and $|\lambda-\lambda_h|,|u-u_h |_{2,\O}\to0$. Thus, following \cite[Lemma~3.4]{MRR2015}, we have   
\begin{align}
\lambda b(u,e)-\lambda_h b_h( u_h,e)&= \lambda ||u ||_{0,\O}^2 +\lambda_h ||u_h ||_{0,\O}^2-\lambda_h b(u,u_h)\nonumber\\
&=C\frac{\lambda + \lambda_h}{2}|| e||_{0,\O}^2\leq C\frac{\lambda + \lambda_h}{2}|| e||_{0,\O}|e|_{2,\O}.\label{err_ap2}
\end{align}

On the other hand, using that $b(\cdot,\cdot) ,b_{h,K}(\cdot,\cdot)$ are scalar products, Lemma~\ref{consandstability}, and \eqref{est_pol_error}, provide the following bound 
	\begin{align}
		\lambda_h\Big\{b(u_h,e_I)-b_h(u_h,e_I) \Big\}&=\lambda_h\sum\limits_{K\in \Omega_h} \Big\{b_K(u_h,e_I)-b_{h,K}(u_h,e_I) \Big\}\nonumber\\
	&\leq |\lambda_h| \sum\limits_{K\in \Omega_h}\Bigg\{ b_K(u_h-\PiK u_h,u_h-\PiK u_h)^{1/2}b_K(e_I,e_I)^{1/2} \nonumber\\
    &\quad + \, b_{h,K}(u_h-\PiK u_h,u_h-\PiK u_h)^{1/2}b_{h,K}(e_I,e_I)^{1/2}\Bigg\}\nonumber\\
	&\leq  C\Big\{\sum\limits_{K\in \Omega_h} b_{h,K}(u_h-\PiK u_h,u_h-\PiK u_h)\Big\}^{1/2}|| e_I||_{0,\O}\nonumber\\
        &\leq C\Big\{\sum\limits_{K\in \Omega_h} b_{h,K}(u_h-\PiK u_h,u_h-\PiK u_h)\Big\}^{1/2}| e|_{2,\O}.\label{err_ap3}
	\end{align}
Following the same arguments, we obtain
	\begin{align}
		a_h(u_h,e_I)-a(u_h,e_I)\leq C\Big\{\sum\limits_{K\in \Omega_h} a_K(u_h-\PiK u_h,u_h-\PiK u_h)\Big\}^{1/2}| e|_{2,\O}.\label{err_ap4}
	\end{align}
	
Next, \eqref{est_pol_error} and Lemma~\ref{consandstability} imply that 
	\begin{align}
		E_1&\leq \sum\limits_{K\in \Omega_h} a_K(u_h-\PiK u_h,u_h-\PiK u_h)^{1/2}|e|_{2,K}\nonumber\\
		&= \sum\limits_{K\in \Omega_h} s_K^{\Delta}(u_h-\PiK u_h,u_h-\PiK u_h)^{1/2}|e|_{2,K} \nonumber\\
        &\leq  C\Big(\sum\limits_{K\in \Omega_h} S_{K}^2\Big)^{1/2} |e|_{2,\O}.\label{err_ap5}
	\end{align}
Similarly,  \eqref{est_pol_error} and Lemma~\ref{consandstability} lead to
	\begin{align}
	E_2&=\sum\limits_{K\in \Omega_h} b_K(\PiK u_h-u_h,e-e_I)\leq  C\Big(\sum\limits_{K\in \Omega_h} S_{K}^2\Big)^{1/2} |e|_{2,\O}.\label{err_ap6}
	\end{align}
To bound $E_3$, we apply the Cauchy-Schwarz inequality and \eqref{est_pol_error} to obtain 
	\begin{align}
	E_3&\leq  \sum\limits_{K\in \Omega_h} ||\lambda_h\PiK u_h ||_{0,K}||e-e_I||_{0,K}\nonumber\\
		&\leq C\Big( \sum\limits_{K\in \Omega_h} h_K^4||\lambda_h\PiK u_h ||_{0,K}^2 \Big)^{1/2}||e||_{2,\O}\nonumber\\
        &\leq C\Big( \sum\limits_{K\in \Omega_h} h_K^4||\lambda_h\PiK u_h ||_{0,K}^2 \Big)^{1/2}|e|_{2,\O}\nonumber\\
        &=C\Big( \sum\limits_{K\in \Omega_h} \Xi_K^2 \Big)^{1/2}|e|_{2,\O}\label{err_ap7}
	\end{align}
For $E_4$, the trace inequality for $\nabla (e-e_I)\in H^1(K)$ and \eqref{est_pol_error} provide that
	\begin{align}
		E_4&\leq \sum\limits_{K\in \Omega_h}\sum_{l\in \mathcal{E}_K\cap \mathcal{E}_{\O}} \Big|\Big|[\![(\nabla^2\PiK u_h )\boldsymbol{n}_K^l]\!]  \Big|\Big|_{0,l} || \nabla (e-e_I)||_{0,K}\nonumber\\
	&\leq\sum\limits_{K\in \Omega_h} \sum_{l\in \mathcal{E}_K\cap \mathcal{E}_{\O}} \Big|\Big|[\![(\nabla^2\PiK u_h )\boldsymbol{n}_K^l]\!]  \Big|\Big|_{0,l} \Big\{h_K^{-1/2}|| \nabla(e-e_I)||_{0,K} ||+h_K^{1/2}|\nabla(e-e_I)|_{1,K}\Big\}\nonumber\\
	%	&\leq \sum\limits_{K\in \Omega_h}\sum_{l\in \mathcal{E}_K\cap \mathcal{E}_{\O}} \Big|\Big|[\![(\nabla^2\PiK u_h )\boldsymbol{n}_K^l]\!]  \Big|\Big|_{0,l} \Big\{h_K^{-1/2}h_K|e|_{2,K} ||+h_K^{1/2}|e|_{2,K}\Big\}\nonumber\\
		&\leq \sum\limits_{K\in \Omega_h}\sum_{l\in \mathcal{E}_K\cap \mathcal{E}_{\O}} h_K^{1/2}\Big|\Big|[\![(\nabla^2\PiK u_h )\boldsymbol{n}_K^l]\!]  \Big|\Big|_{0,l}||\nabla e||_{1,K}\nonumber\\
        &\leq C \Bigg\{\sum\limits_{K\in \Omega_h}\sum_{l\in \mathcal{E}_K\cap \mathcal{E}_{\O}} h_K\Big|\Big|[\![(\nabla^2\PiK u_h )\boldsymbol{n}_K^l]\!]  \Big|\Big|_{0,l}^{2}\Bigg\}^{1/2}  ||\nabla e||_{1,\Omega}\nonumber\\
        &= \Bigg\{\sum\limits_{K\in \Omega_h}\sum_{l\in \mathcal{E}_K\cap \mathcal{E}_{\O}} \mathcal{J}_{l}^2\Bigg\}||e||_{2,\Omega} \leq  \mathcal{J}^2||e||_{2,\Omega}\leq  C_{P}\mathcal{J}^2|e|_{2,\Omega},\label{err_ap8}
	\end{align}
where, in the last step we used the Poincaré inequality.

Finally, the proof finishes by replacing \eqref{err_ap2}-\eqref{err_ap8} in \eqref{err_ap} and using the Cauchy-Schwarz inequality. 
\end{proof}

The following result is a consequence of Theorem~\ref{Th5.6}.
\begin{corollary}\label{coroll_MG}
There exists a positive constant $C$ independent of the parameter $h$ such that
\begin{align*}
||u-\Pi_K^0 u_h ||_{0,\Omega}+|u-\Pi_K^{\Delta} u_h|_{2,h}\leq C\Big\{\eta +\frac{\lambda+\lambda_h}{2}||u-u_h ||_{0,\O} \Big\}.
\end{align*}
As a consequence, it holds
\begin{align*}
|u-u_h|_{2,\O}+||u-\Pi_K^0 u_h ||_{0,\Omega}+|u-\Pi_K^{\Delta} u_h|_{2,h}\leq C\Big\{\eta +\frac{\lambda+\lambda_h}{2}||u-u_h ||_{0,\O} \Big\}.
\end{align*}
\end{corollary}
\begin{proof}
Given any $K\in \Omega_h$, the following bound holds 
\begin{align*}
||u-\Pi_K^0 u_h||_{0,K}\leq C \Big\{||u- u_h ||_{0,K}+ ||u_h-\Pi_K^0 u_h||_{0,K} \Big\}\leq C \Big\{||u- u_h ||_{2,K}+ ||u_h-\Pi_K^0 u_h||_{0,K} \Big\}.
\end{align*}
Next, by adding $|u-\PiK u_h|_{2,K}$ in the above inequality and summing for all $K\in \Omega_h$,  we have
\begin{align*}
||u-\Pi_K^0 u_h ||_{0,\O}^2 + |u-\PiK u_h|_{2,h}^2
&\leq C \Bigg( ||u-u_h ||_{2,\Omega}^2 + \sum\limits_{K\in \Omega_h} \Big\{||u- u_h ||_{2,K}^2+ ||u_h-\Pi_K^0 u_h||_{0,K}^2 \Big\} \Bigg)\\
&\leq C_P \Bigg( |u-u_h |_{2,\Omega}^2 + \sum\limits_{K\in \Omega_h} \Big\{||u- u_h ||_{2,K}^2+ ||u_h-\Pi_K^0 u_h||_{0,K}^2 \Big\} \Bigg).
\end{align*}
Finally, the proof is obtained from Lemma~\ref{consandstability}, the definition of the stabilisation estimator $S_K$, and Theorem~\ref{Th5.6}.
\end{proof}

Theorem~\ref{teo:mainConv} (cf. \eqref{double_conv_2}) and Corollary~\ref{coroll_MG} provide the following result. 
\begin{corollary}\label{coro_l-l_h}
	There exists a positive constant $C$ independent of the parameter $h$ such that 
	\begin{align*}
		|\lambda - \lambda_h|\leq C\Bigg\{\eta +\frac{\lambda+\lambda_h}{2}||u-u_h ||_{0,\O} \Bigg\}^2.
	\end{align*}
\end{corollary}
Notice that the terms on the right-hand side in Corollaries~\ref{coroll_MG} and \ref{coro_l-l_h} depend of $||u-u_h ||_{0,\Omega}$, which are not computable on the basis provided by the output values of the operators in ${\bf D_1}$ and ${\bf D_2}$. Therefore, we establish a result showing that the term $||u - u_h||_{0,\Omega}$ is asymptotically negligible.
\begin{theorem}\label{th:reliability}
	There exists a positive constant $C$ independent of parameter $h$ such that
	\begin{enumerate}
		\item[(i)] \quad $|u-u_h|_{2,\O}+||u-\Pi_K^0  u_h ||_{0,\Omega}+|u-\PiK u_h|_{2,h}\leq C\eta$.
		\item[(ii)] \quad $|\lambda - \lambda_h|\leq C\eta^2$.
	\end{enumerate}
\end{theorem}
\begin{proof}
	\begin{enumerate}
		\item[(i)] Using the Theorem~\ref{Teor4.9} and Corollary~\ref{coroll_MG} we deduce
		\begin{align*}
			&|u-u_h|_{2,\O}+||u-\Pi_K^0  u_h ||_{0,\Omega}+|u-\PiK u_h|_{2,h}\\
			&\leq C\Bigg\{\eta + h^s\Big\{|u-u_h|_{2,\O}+||u-\Pi_K^0  u_h ||_{0,\Omega}+|u-\PiK u_h|_{2,h}\Big\}\Bigg\}.
		\end{align*}
		Therefore, results clear that there exists $h_0>0$ such that $\forall h<h_0$ the item (i) is true.
	\item[(ii)] Again from Theorem~\ref{Teor4.9} and item (i) we have the existence of $h_0>0$ such that for all $h<h_0$ implies
		\begin{align}
			||u-u_h ||_{0,\O}\leq Ch^s\eta, \label{u-u_hleqeta}
		\end{align} 
		and applying \eqref{u-u_hleqeta} in  Corollary~\ref{coro_l-l_h} we have the result. 
	\end{enumerate}
\end{proof}

%%%--------------------------------------------------    Efficiency

\subsection{Efficiency}\label{efficiency}
This section proves that our error estimator $\eta$ is efficient up to oscillation terms given by the polynomial projections $\Pi_K^0$ and $\Pi_K^\Delta$. We start by bounding the local volume estimator $\Xi_K$.
\begin{lemma}\label{R_keffic}
	There exists a positive constant independent of $h_K$ such that
	\begin{align}
		\Xi_K\leq C\Big\{|u-u_h|_{2,K}+S_K+h_K^2||\lambda_h u_h- \lambda u||_{0,K}\Big\}
	\end{align}
	\begin{proof}
		Given any $K\in\Omega_h$ let define $\psi_K$ a polynomial function such that $0\leq \psi_K\leq 1$ and  $ \partial_{n_K}\psi_K,\psi|_{\partial K}=0$. Let us define $v^{K}:=\psi_K\lambda_h\PiK u_h \in H_0^2(K)$ and we extend $v^K$ by zero in all $\O$ such that $v^K\in   V  $ (see for instance \cite[Section~5.3]{chen2022posteriori}). Therefore, from Lemma~\ref{eq-residual} we deduce
		\begin{align}
			&a_K(e,\psi_K \lambda_h\PiK u_h)=\lambda b_K(u,\psi_K\lambda_h\PiK u_h)-\lambda_h b_K(u_h,\psi_K\lambda_h\PiK u_h) \nonumber\\ 
            &\quad - \, a_K(u_h-\PiK u_h,\psi_K\lambda_h \PiK u_h) + \lambda_h  b_K(u_h-\PiK u_h,\psi_K \lambda_h\PiK u_h)+\lambda_h || \PiK u_h||_{0,K}^2.
		\end{align}
		Since $\PiK u_h\in \mathbb{P}_2(K)$,  Lemma~\ref{lema4.2gatica2010analysis} implies that
		\begin{align}
			&C^{-1}||\lambda_h \PiK u_h ||_{0,K}^2\leq \int_K \psi_K (\lambda_h \PiK u_h)^2 \nonumber\\
			&=a_K(e,\psi_K \lambda_h\PiK u_h)-\Big\{\lambda b_K(u,\psi_K\lambda_h\PiK u_h)-\lambda_h b_K(u_h,\psi_K\lambda_h\PiK u_h) \Big\}\nonumber\\
			&\quad + \, a_K(u_h-\PiK u_h,\psi_K\lambda_h \PiK u_h) -\lambda_h  b_K(u_h-\PiK u_h,\psi_K \lambda_h\PiK u_h)\nonumber\\
			&\leq C\Big\{ |e|_{2,K}|\psi_K\lambda_h\PiK u_h|_{2,K}+|b_K(\lambda_hu_h-\lambda u,\psi_K \lambda_h \PiK u_h)| + |u_h-\PiK u_h|_{2,K}|\psi_K \lambda_h \PiK u_h|_{2,K}\nonumber\\
			&\quad +\, \lambda_h ||u_h-\Pi_K^0 u_h ||_{0,K}||\psi_K \lambda_h \PiK u_h ||_{0,K}\Big\}\nonumber\\
			&\leq C\Big\{h^{-2}|e|_{2,K}+||\lambda_hu_h-\lambda u ||_{0,K}+h_K^{-2} |u_h-\PiK u_h|_{2,K}+||u_h-\Pi_K^0 u_h ||_{0,K}  \Big\}|| \psi_K\lambda_h\PiK u_h||_{0,K},\nonumber
		\end{align}
where we have used Lemma~\ref{lema4.2gatica2010analysis}. Thus, using the fact that $0\leq  \psi_K\leq 1$, multiplying by the factor $Ch_K^2||\lambda_h \PiK u_h ||_{0,K}^{-1}$, Lemma~\ref{consandstability} and definition of the stabilisation term $S_K$ the result is  obtained.
	\end{proof}
\end{lemma}

Next, we will show an upper bound for the local stabilisation term $S_K$ and the local jump estimator $\mathcal{J}_l$.
\begin{lemma}\label{theta12K_effic}
	There exists a positive constant $C$ independent of the parameter $h_K$ such that
	\begin{align*}
		S_K\leq C\Big\{|u-u_h|_{2,K}+|u-\PiK  u_h|_{2,K}+||u-\Pi_K^0  u_h||_{0,K} \Big\}.
	\end{align*}
\end{lemma}
\begin{proof} 
The definition of the stabilisation term and the Cauchy-Schwarz inequality imply that
\begin{align*}
S_K&=    a_{h,K}(u_h-\PiK u_h,u_h-\PiK u_h)^{1/2} + b_{h,K}(u_h-\PiK u_h,u_h-\PiK u_h)^{1/2}\\
&\leq C \Big\{a_{K}(u_h-\PiK u_h,u_h-\PiK u_h)^{1/2} + b_{K}(u_h-\PiK u_h,u_h-\PiK u_h)^{1/2}\Big\}\\
&\leq C \Big\{|u_h-\PiK u_h|_{2,K}+ ||u-\Pi_K^0u_h||_{0,K}  \Big\}
\\
&\leq C \Big\{ |u-u_h|_{2,K} + |u-\PiK u_h|_{2,K}+ ||u-\Pi_K^0u_h||_{0,K}  \Big\}
\end{align*}
\end{proof}

\begin{lemma}\label{jump_eff}
	There exists a positive constant $C$ independent of the parameter $h_K$ such that
	\begin{align}
	\mathcal{J}_l\leq C\sum\limits_{K'\in \omega(l)}\Big\{|u-u_h |_{2,K'}+ h_{K'}^2 ||\lambda u-\lambda_h u_h ||_{0,K'}+S_{K'}\Big\},\quad \forall l\in \mathcal{E}_K\cap\mathcal{E}_{\O}, 
	\end{align}
	where $\omega(l):=\Big\{K'\in \Omega_h:l\in \partial {K'} \}$.
\end{lemma}
\begin{proof}
Let $l\in \mathcal{E}_K\cap\mathcal{E}_{\Omega}$ shared by two polygons $K^+$ and $K^-$. Denoting by $\mathcal{T}_K^+$ and $\mathcal{T}_K^-$ the triangulations for $K^+$ and $K^-$ allowed by ${\bf A2}$ and by $T^+$ and $T^-$ the triangles in $\mathcal{T}_K^+$ and $\mathcal{T}_K^-$, respectively, having to $l$ as a common side. We define $\psi_l$ as the corresponding edge bubble function, and $v:=\psi_l[\![(\nabla^2\PiK u_h )\boldsymbol{n}_K^{l}]\!]|_l$. Since $v$ is a polynomial function, $v$ can be extended to $T^+\cup T^-$. For simplicity, we denote this extension by $v$. Then, we consider $\hat{v}:=v\varphi_l\in H_0^2(T^+\cup T^-)\subset V$, such that 
$$
\partial_{\boldsymbol{n}^l}\hat{v}=\psi_l[\![(\nabla^2\PiK u_h )\boldsymbol{n}_K^{l}]\!]|_l,\quad  \hat{v}|_l=0,
\, \mbox{and}\,
||\hat{v}|_{0,T^\dagger}\leq Ch_{K^\dagger}^{3/2}|| [\![(\nabla^2\PiK u_h )\boldsymbol{n}_K^{l}]\!]||_{0,l},
\quad  \forall \dagger  \in \{+,- \},
$$
where $\varphi_l$ is the function of line $l$ (see \cite[Section~5.3]{chen2022posteriori} for further details). From the Lemma~\ref{eq-residual} we have
\begin{align*}
&C ||[\![(\nabla^2\PiK u_h )\boldsymbol{n}_K^{l}]\!]||_{0,l}^2 \leq \int_l [\![(\nabla^2\PiK u_h )\boldsymbol{n}_K^{l}]\!]\psi_l\nonumber\\
&=\sum\limits_{K=T^+,T^-}\Bigg\{ a_K(-e,\hat{v})+ \Big\{\lambda b_K(u,\hat{v})-\lambda_h b_K(u_h,\hat{v}) \Big \} + \sum\limits_{K'\in \omega_l}a_K(\PiK u_h-u_h,\hat{v}) \nonumber\\ 
&\quad + \lambda_h    b_K(u_h-\Pi_K^0 u_h,\hat{v})
 +\int_K \lambda_h \PiK u_h \hat{v} \Bigg\} \nonumber\\
&\leq C \sum\limits_{K=T^+,T^-}\Bigg\{ |e|_{2,K}|\hat{v}|_{2,K}+ ||\lambda u -\lambda_h u_h ||_{0,K}||\hat{v} ||_{0,K} + |u_h-\PiK u_h|_{2,K}|\hat{v}|_{2,K}\nonumber \\
&\quad +||u_h-\Pi_K^0 u_h ||_{0,K}|| \hat{v}||_{0,K}
+ ||\lambda_h \PiK u_h ||_{0,K}|| \hat{v}||_{0,K}\Bigg\}\nonumber\\
&\leq C \sum\limits_{K=T^+,T^-}\Bigg\{\Big( h_K^{-2}|e|_{2,K} + ||\lambda u -\lambda_h u_h ||_{0,K}  + h_K^{-2}|u_h-\PiK u_h|_{2,K} 
\nonumber\\
&\quad +||u_h-\Pi_K^0 u_h ||_{0,K} + ||\lambda_h \PiK u_h ||_{0,K} \Big) ||\hat{v}||_{0,K}\Bigg\}\nonumber\\
&\leq C \sum\limits_{K=T^+,T^-}\Bigg\{\Big( h_K^{-1/2}|e|_{2,K} + h_K^{3/2}||\lambda u -\lambda_h u_h ||_{0,K}  + h_K^{-1/2}|u_h-\PiK u_h|_{2,K} 
\nonumber\\
&\quad +h_K^{3/2}||u_h-\Pi_K^0 u_h ||_{0,K} + h_K^{3/2}||\lambda_h \PiK u_h ||_{0,K} \Big) || [\![(\nabla^2\PiK u_h )\boldsymbol{n}_K^{l}]\!]||_{0,l}\Bigg\}.
\end{align*}

Finally, Lemma~\ref{consandstability}, the definition of $S_K$ and Lemma~\ref{R_keffic} conclude the proof.
\end{proof}

The efficiency of $\eta_K$ is provided next.
\begin{theorem}\label{local_eff}
	There exists a positive constant $C$ independent of the parameter $h$ such that
	\begin{align*}
		\eta_K^2\leq C\sum\limits_{K'\in \omega_{K}} \Big\{ |u-u_h|_{2,K'}^2+||u-\Pi_{K'}^{0}u_h||_{0,K'}^2+|u-\Pi_{K'}^{\Delta}u_h|_{2,K'}^2+h^4||\lambda u-\lambda_h u_h ||_{0,K'}^2\Big\}  
	\end{align*}
\end{theorem}
\begin{proof}
	The result is obtained from Lemmas~\ref{R_keffic}-\ref{jump_eff}.
\end{proof}

Finally, the term $h_K^2||\lambda u-\lambda_h u_h ||_{0,K'}$ is asymptotically negligible for the global estimator $\eta$, this result is provided below.
\begin{corollary}\label{global_eff}
	There exists a positive constant $C$ independent of the parameter $h$ such that 
	\begin{align}
		\eta^2\leq C \Big\{|u-u_h|_{2,\O}^2+|| u-\Pi_K^0 u_h||_{0,\Omega}^2+|u-\PiK u_h|_{2,h}^2\Big\}\nonumber. 
	\end{align}
\end{corollary}
\begin{proof}
The proof follows with the same arguments as those applied in  \cite[Corollary~5.14]{MR2019}.
\end{proof}
\begin{remark}
    We refer the reader to \cite{C1VEM_Polyhedral} for more details about the construction of the C$^1$ virtual element discretisation in 3D. Note that Lemma~\ref{eq-residual} includes a boundary term (c.f. the jump estimator $\mathcal{J}$ in \eqref{global_estimators}), which in this case is defined over the faces of each polyhedron. Since this term does not require the explicit evaluation of $u_h$ on the faces—where it belongs to a virtual space—the previously given proofs remain valid in the three-dimensional setting. We also refer to \cite{cangiani2017posteriori} for the construction of the bubble function on faces.
\end{remark}
\begin{comment}
    Using the Theorem~\ref{local_eff} we deduce
	\begin{align}
		\eta^2&\leq C \Big\{|u-u_h|_{2,\Omega}^2+||u-\Pi_K^0 u_h ||_{0,\Omega}^2+|u-\PiK  u_h|_{2,h}^2 + h^4 ||\lambda u-\lambda_h u_h ||_{0,\Omega}^2\Big\}. \nonumber
	\end{align}
	Moreover, assuming that  $||u_h ||_{0,\O}=1$ we have
	\begin{align}
		h^2||\lambda u-\lambda_h ||_{0,\O}^2&\leq 2\lambda^2||u-u_h ||_{0,\O}^2 +2|\lambda - \lambda_h|^2\leq C |u-u_h|_{2,K}^2+2|\lambda - \lambda_h|^2\nonumber
	\end{align}
	On the other hand, from Theorem~\ref{teo:mainConv} (cf. \eqref{double_conv_2}) we obtain
	\begin{align}
		|\lambda -\lambda_h|^2\leq (|\lambda|+|\lambda_h|)|\lambda - \lambda_h|\leq C \Big\{|u-u_h|_{2,K}^2+||u-\Pi_h u_h ||_{0,h}^2+|u-\Pi_h u_h|_{2,h}^2 \Big\},
	\end{align}
	as consequence results
	\begin{align}
		\eta^2\leq C \Big\{|u-u_h|_{2,\O}^2+||u-\Pi u_h ||_{0,h}^2+|u-\Pi_h u_h|_{2,h}^2\Big\}\nonumber. 
	\end{align}
\end{comment}    

\section{Numerical examples}\label{sec:numerical-examples} 
In this section, we present some numerical results that illustrate the properties of the estimator introduced in Section~\ref{SEC:EST_A_POST}, showing the optimal behaviour of the associated adaptive algorithm under different convex polygonal meshes. Then, we introduce the classical L-shape domain to illustrate the capability of capturing singularities in non-convex domains.
\begin{figure}[h!]
    \centering
    \subfigure[Voronoi. \label{fig:polygonal}]{\includegraphics[width=0.23\textwidth,trim={2.7cm 1cm 2.4cm 0.7cm},clip]{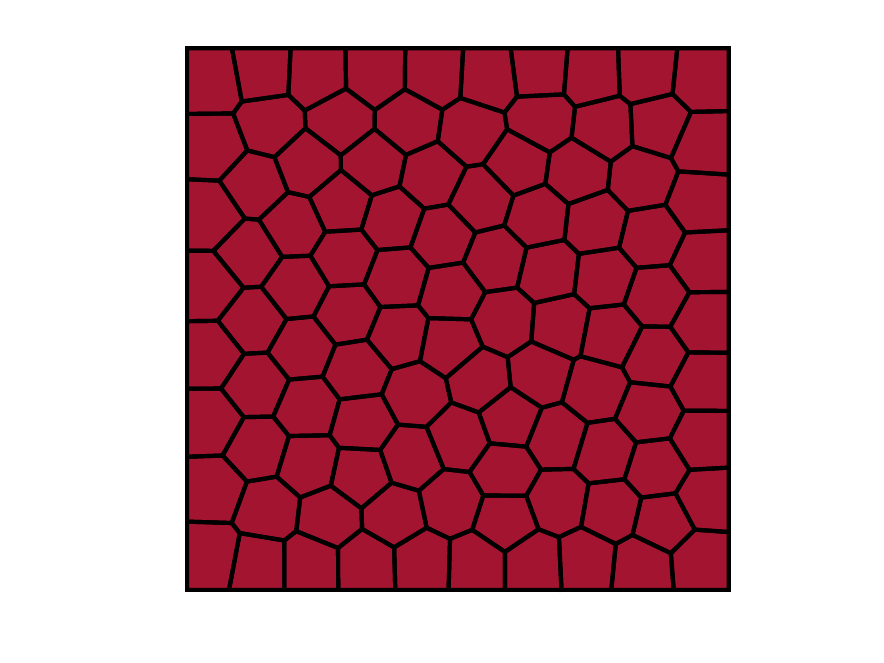}}  
    \subfigure[Square. \label{fig:square}]{\includegraphics[width=0.23\textwidth,trim={2.7cm 1cm 2.4cm 0.7cm},clip]{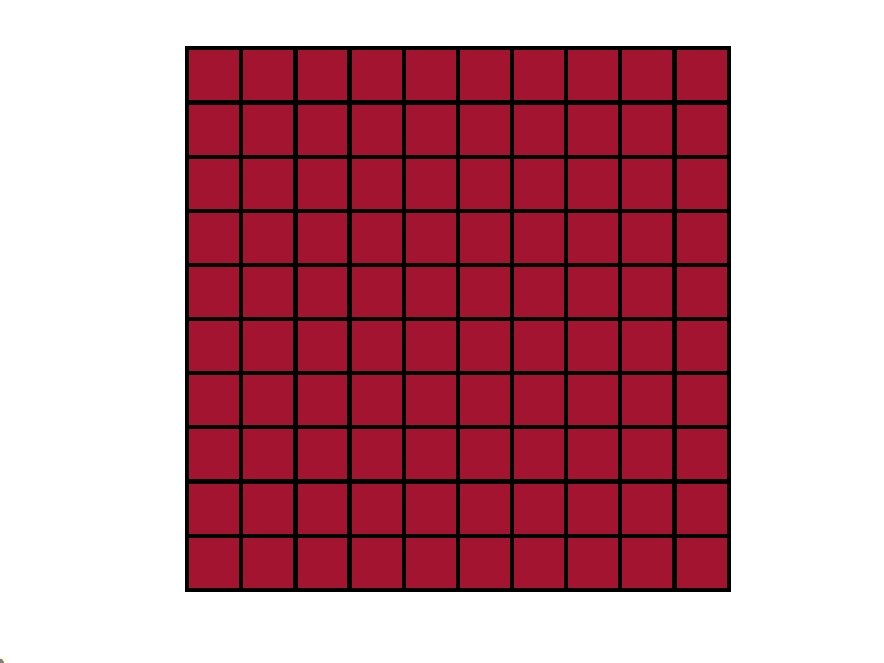}} 
    \subfigure[Crossed. \label{fig:crossed}]{\includegraphics[width=0.23\textwidth,trim={2.7cm 1cm 2.65cm 0.7cm},clip]{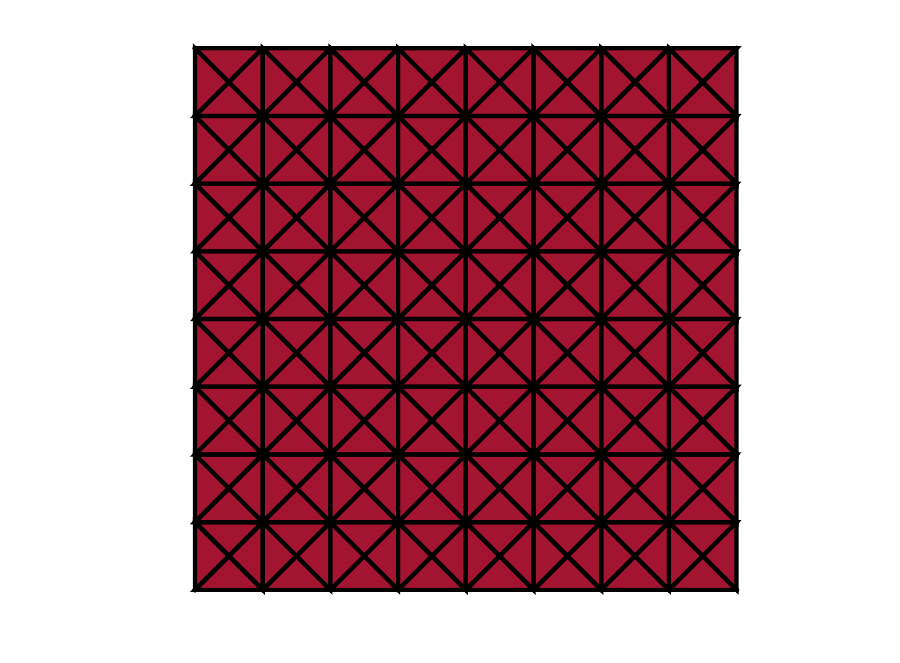}}
    \subfigure[L. \label{fig:L}]{\includegraphics[width=0.23\textwidth,trim={2.7cm 1cm 2.4cm 0.7cm},clip]{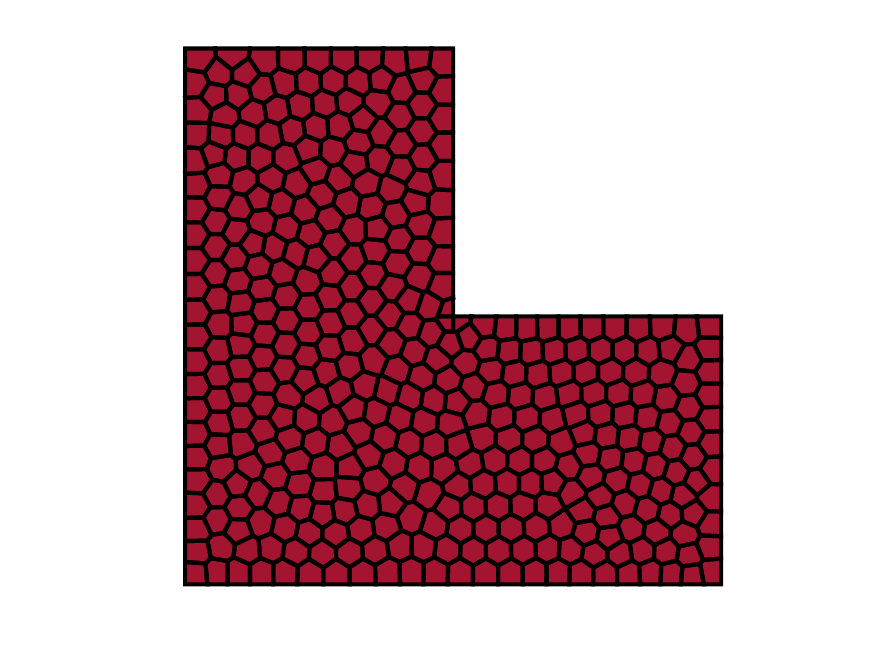}} \\
    \subfigure[Aircraft wing. \label{fig:wingPlane}]{\includegraphics[width=0.23\textwidth,trim={2.7cm 1cm 2.4cm 0.7cm},clip]{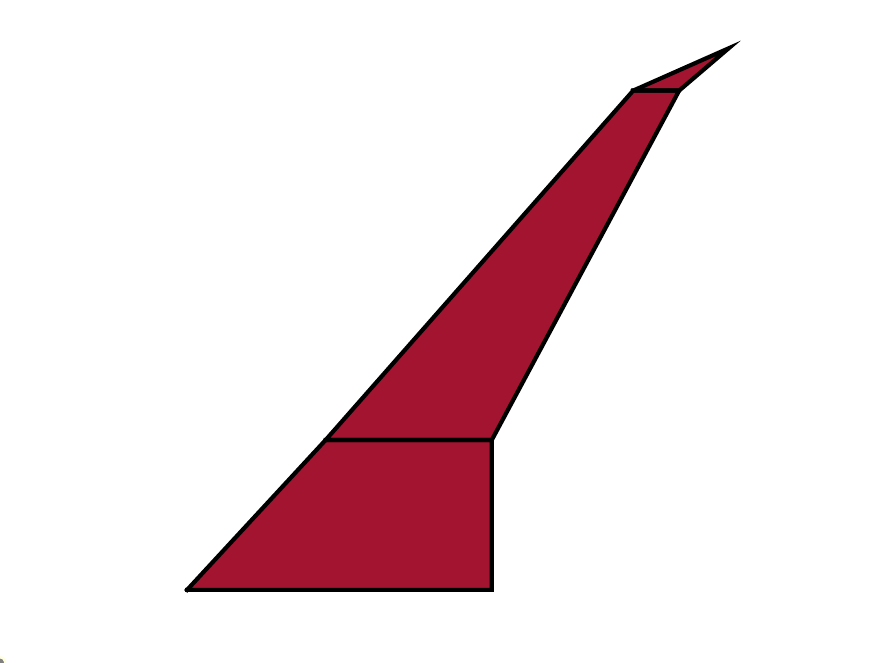}}
    \subfigure[Cube. \label{fig:cube}]{\includegraphics[width=0.23\textwidth,trim={6.cm -2.cm 9.15cm 2.cm},clip]{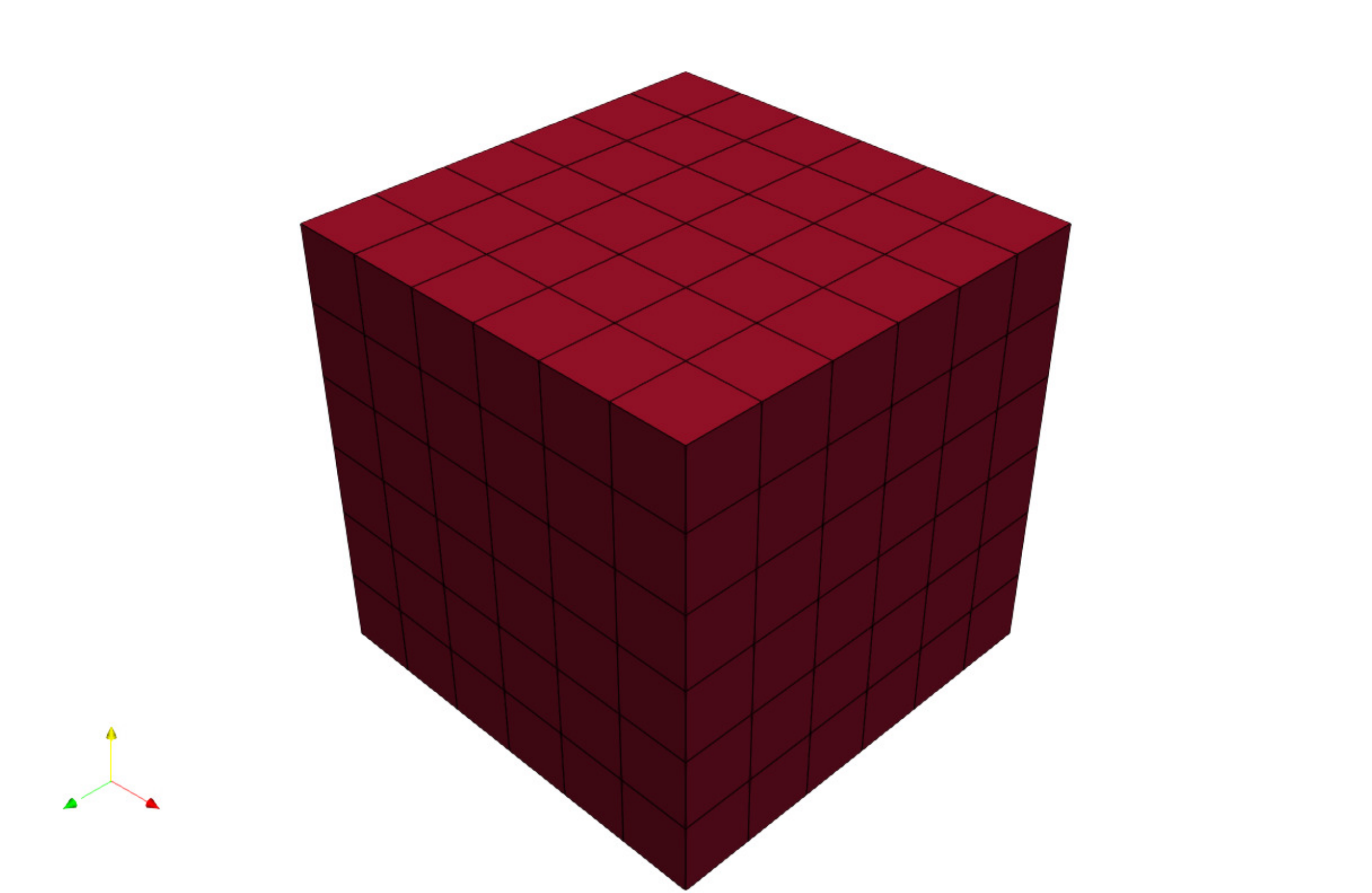}}
    \subfigure[Fichera. \label{fig:Fichera}]{\includegraphics[width=0.23\textwidth,trim={6.cm -2.cm 9.15cm 2.cm},clip]{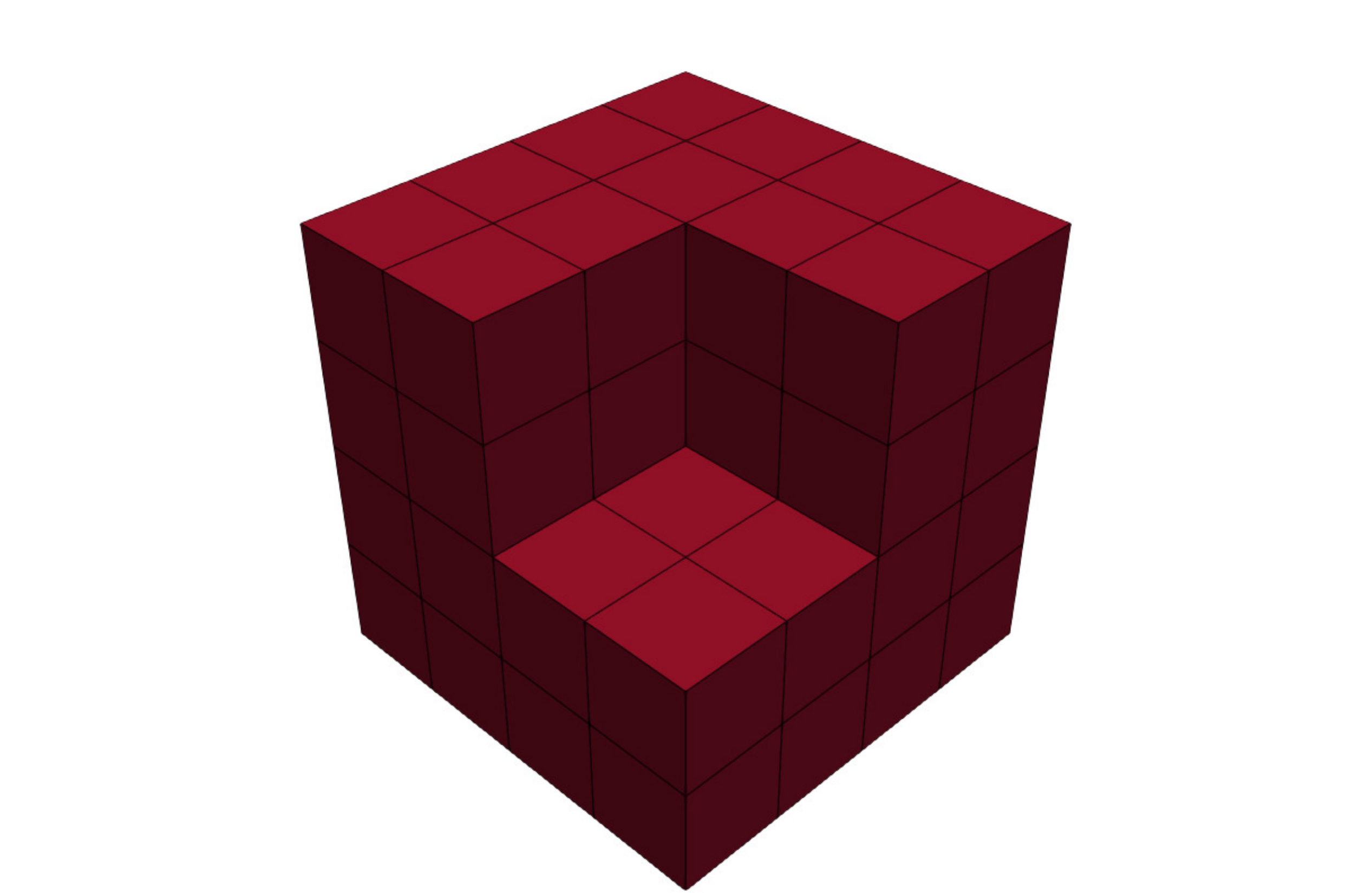}}
    \caption{Domain discretisations used for the numerical examples.}\label{meshes_2d}
\end{figure}

The numerical implementation is done with the library \texttt{VEM++} \cite{dassi2025vem++} where the VE space presented in Section~\ref{SEC:DISCRETE} is available. To solve the generalised eigenvalues problem arising from such discretizations, \texttt{VEM++} contains a wrapper of the C++ library SLEPC \cite{slepc} with the problem type option \texttt{EPS--GHEP}. Regarding the adaptive algorithm, we follow the usual strategy
$$\textnormal{SOLVE} \rightarrow \textnormal{ESTIMATE} \rightarrow \textnormal{MARK} \rightarrow \textnormal{REFINE}$$
where the SOLVE and ESTIMATE steps are performed inside \texttt{VEM++} using SLEPC as a solver for the linear system, in this way, we exploit the high-speed computation capabilities of C++. Next, for the 2D REFINE step, we use the Matlab-based method from \cite{yu2021implementationpolygonalmeshrefinement}, which connects edge mid-points to the polygon barycenter. On the other hand, the 3D adaptive refinement routine operates on cubical meshes with support for hanging nodes, making use of the \texttt{p4est} library \cite{p4est} via the \texttt{GridapP4est} module of the Julia package \texttt{Gridap} \cite{gridap}. We recall that the refinement procedure is independent of the \texttt{VEM++} capabilities. Therefore, the implementation presented in this paper can be extended to more general (non-convex) polygonal shapes that satisfy \textbf{A1}, \textbf{A2} by selecting an appropriate refinement routine. Finally, for the MARK procedure, we follow a D\"orfler/Bulk marking strategy, marking the subset of mesh elements $\mathcal{K} \subseteq \mathcal{T}^h$ with the largest estimated errors such that for $\delta = \frac{1}{2}\in [0,1]$, we have
$$\delta \sum_{K\in \Omega_h} \eta_K^2 \leq \sum_{K\in \mathcal{K}} \eta_K^2.$$ 

The experimental order of convergence $r(*)$ against the total number of degrees of freedom $\textnormal{\#DoF}$ applied to either the error $|\lambda_{i,h} - \lambda_i|$ or the estimator $\eta$, and the effectivity index $\textnormal{eff}$ of the refinement $1\leq j$ for $d=2,3$ are computed as
$r(*)_{j+1} = -d \log\left(*_{j+1}/*_{j}\right)/\log\left(\textnormal{\#DoFs}_{j+1}/\textnormal{\#DoFs}_{j}\right)$, $\quad \textnormal{eff}_{j} = \eta^{2}_j/|\lambda_{i,h} - \lambda_i|_j$.%Note that the estimator $\eta$ to be computed is restricted to the eigenvalue $\lambda$ to be obtained.
\begin{figure}[h!]
    \centering
    \includegraphics[width=0.55\textwidth,trim={1.75cm 0.2cm 2.2cm 0.8cm},clip]{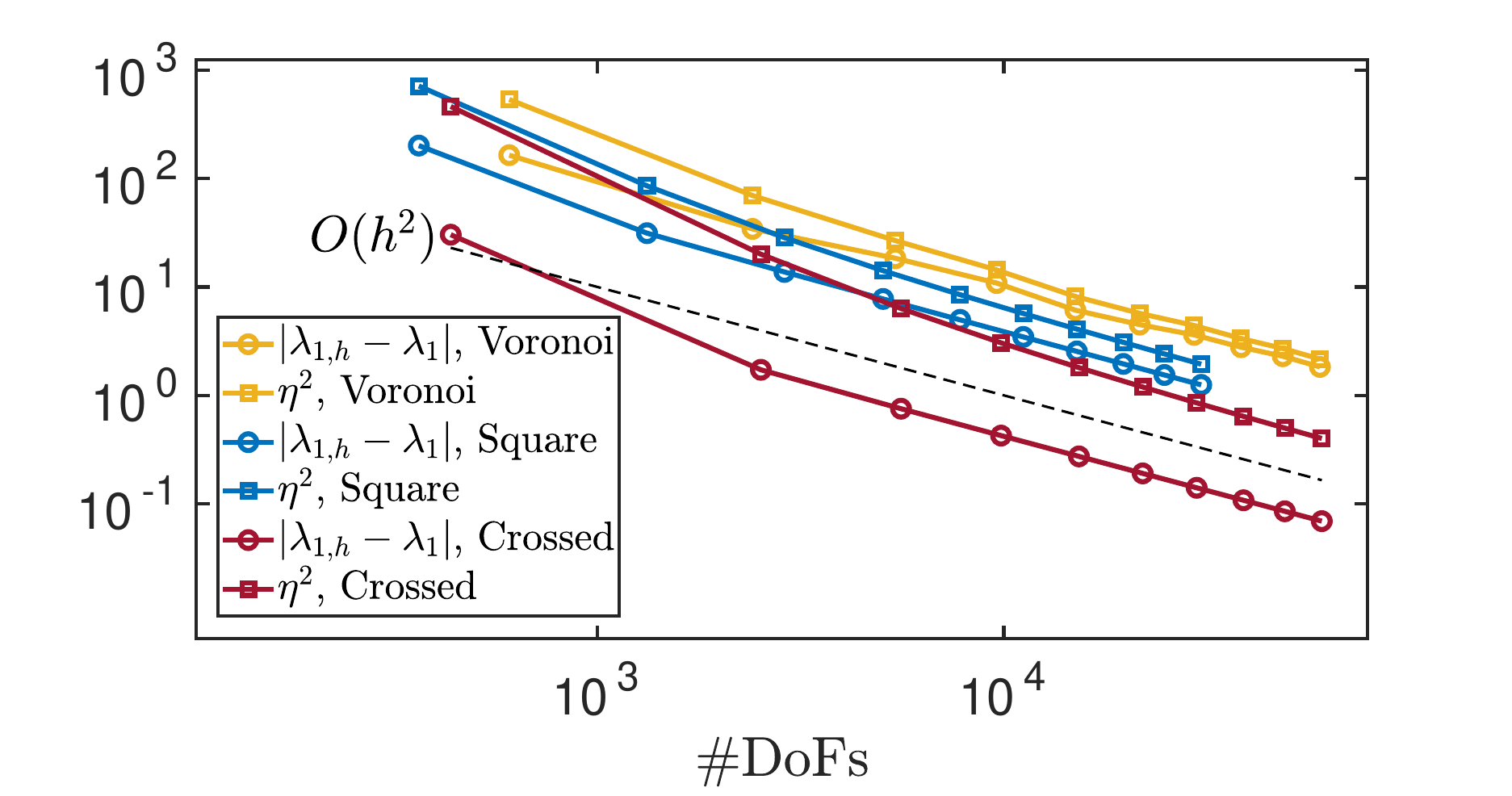}  
    \includegraphics[width=0.3025\textwidth,trim={1.2cm 0.2cm 0.5cm 1.2cm},clip]{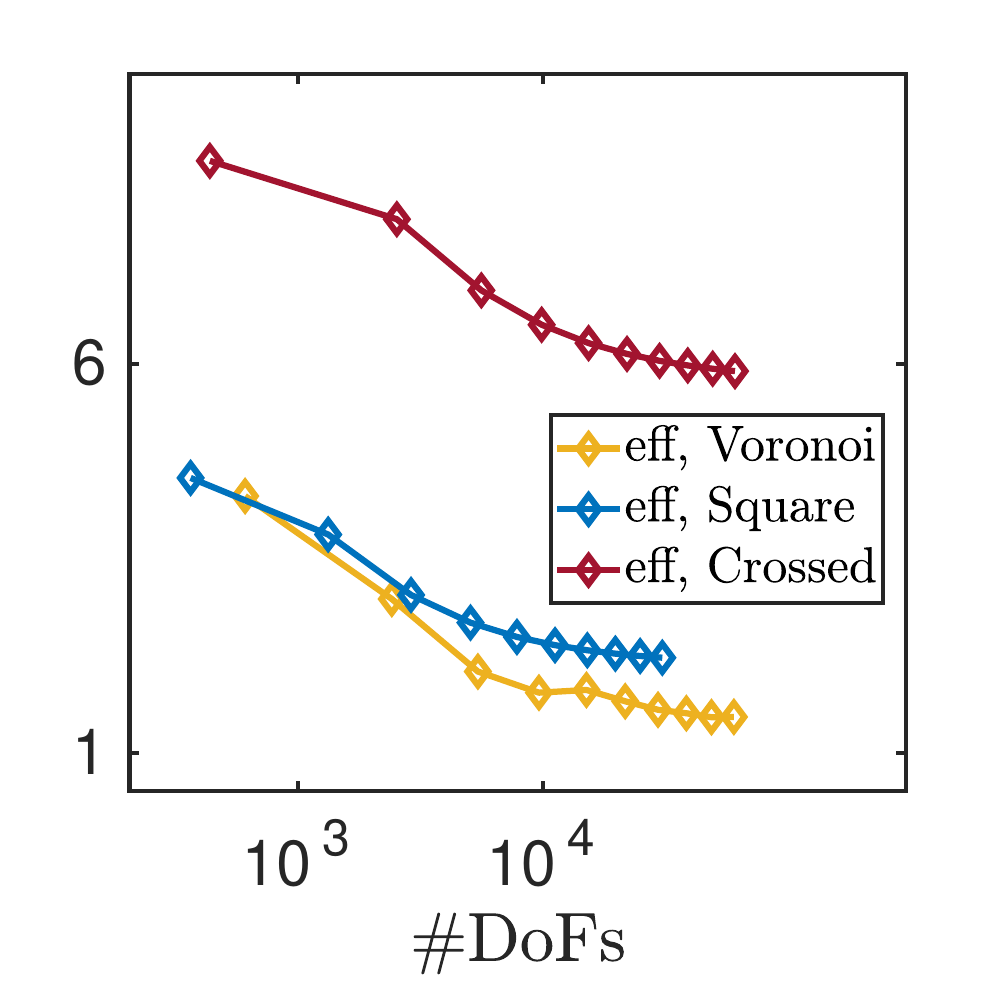}  
    \caption{Example 1. Behaviour of the error $|\lambda_{1,h} - \lambda_1|$, global error estimator $\eta$ (left), and effectivity index $\textnormal{eff}$ (right) for the first eigenvalue $\lambda_1$ across various meshes under uniform refinement.  }\label{convergence_clamped}
\end{figure}
\subsection{Example 1: Behaviour under uniform refinement: 2D case.} 
We consider the vibration problem \eqref{MPr} with boundary clamped boundary conditions as in \eqref{CP} for the unit square $\Omega = (0,1)^2$ under a variety of discretisations (see Figure~\ref{fig:polygonal}-\ref{fig:crossed}). It is well known that the lowest eigenvalue of this problem is given by $\lambda_1 \approx 1294.93397959171$ (see e.g. \cite{Bjorstad1999, Feng2024}). In addition, the associated eigenfunction $u_1$ is a smooth function, which will not affect the convergence rate. 

The curves for the error $|\lambda_{1,h} - \lambda_1|$, and global error estimator $\eta^2$ under uniform refinement are shown in Figure~\ref{convergence_clamped}. Note that the error is bounded above by the estimator, confirming the reliability of the method (see Theorem~\ref{th:reliability}). In addition, we observe the expected convergence rate of $O(h^2)$ given in \eqref{double_conv_1}. Finally, the effectivity index $\textnormal{eff}$ remains bounded for each mesh tested accordingly to Corollary~\ref{global_eff}.
\begin{figure}[h!]
    \centering
    \includegraphics[width=.85\textwidth,trim={1.2cm 0.1cm 12.cm 3.7cm},clip]{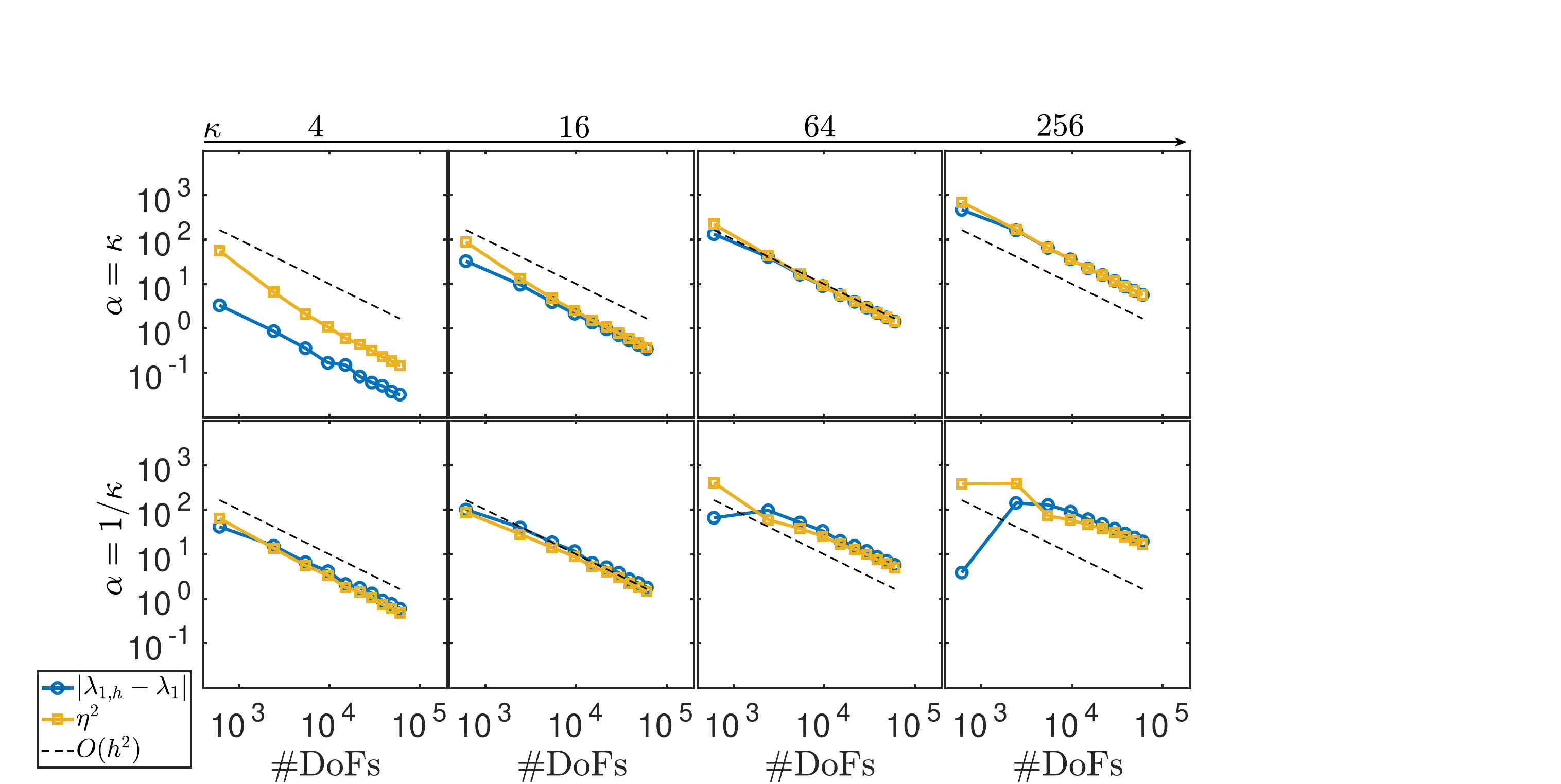}
    \caption{Example 2. Curves of error $|\lambda_{1,h}-\lambda_1|$, and global error estimator $\eta$ for the first eigenvalue $\lambda_1$ in the Voronoi mesh under uniform refinement with a variation of the stabilisation parameter $\alpha$.}\label{convergence_stab}
\end{figure}
\subsection{Example 2: The role of the stabilisation.}\label{ex:2}
For this numerical experiment, we focus on the vibration problem \eqref{MPr} where $\Omega =(0,1)^2$, and we consider simply supported boundary conditions \eqref{CP}. Notably, this boundary condition does not alter the discrete formulation presented in Section~\ref{SEC:DISCRETE}. The main difference is that the linear system incorporates the degrees of freedom corresponding to $\partial_{\boldsymbol{n}} u$. A key advantage of this example is that the first eigenvalue is known explicitly, given by $\lambda_1 = 4\pi^4 \approx 389.6364$ (see \cite{Hu2015}). So we can compute the exact error $|\lambda_{1,h}-\lambda_1|$.

Here, we analyse the influence of the stabilisation operator by introducing two coefficients $\alpha_\Delta, \alpha_0$ in the stabilisation term (see \cite{MRV2018}) as follows:
$$\alpha_{\Delta} s_K^{\Delta}(u_h-\PiK u_h,u_h-\PiK u_h),\quad   \alpha_0 s_{K}^0(u_h-\PiK u_h,u_h-\PiK u_h).$$

These coefficients appear both in the discrete weak formulation (see Problem~\ref{DiscVP}) and the global stabilisation estimator $S$. For simplicity, we fix the Voronoi mesh shown in Figure~\ref{fig:polygonal}, and we set $\alpha_\Delta, \alpha_0=\alpha\in \{\kappa, 1/\kappa\}$, where $\kappa=\{4, 16, 64, 256\}$. Similar results are obtained with other meshes.

In Figure~\ref{convergence_stab}-\ref{eff_stab}, we report the curves for the error $|\lambda_{1,h} - \lambda_1|$ and global error estimator $\eta^2$ (respectively) for the lowest eigenvalue $\lambda_1$ computed by the method for different values of $\alpha$ under uniform refinement. We observe that the SLEPC solver successfully computes an eigenvalue for each uniform refinement and every proposed value of $\alpha$. However, the results indicate that for the values $\alpha = 1/64$ and $\alpha=1/256$, the expected convergence rate of $O(h^2)$ deteriorates due to the influence of this parameter. This observation allows us to extend the results in \cite{MRV2018} to the minimal risk interval $\alpha \in [1/64,256]$, noting that more refinements are required to improve the accuracy of the eigenvalue computed for high values of $\alpha$. %Moreover, it is well known that stabilisation is less involved when the mesh is uniformly distributed. The experiment presented here suggests that one should be mindful of the quality of the initial mesh (uniformity) to minimise potential risks.
\begin{figure}[h!]
    \centering
    \includegraphics[width=.85\textwidth,trim={1.2cm 0.2cm 12.cm 3.6cm},clip]{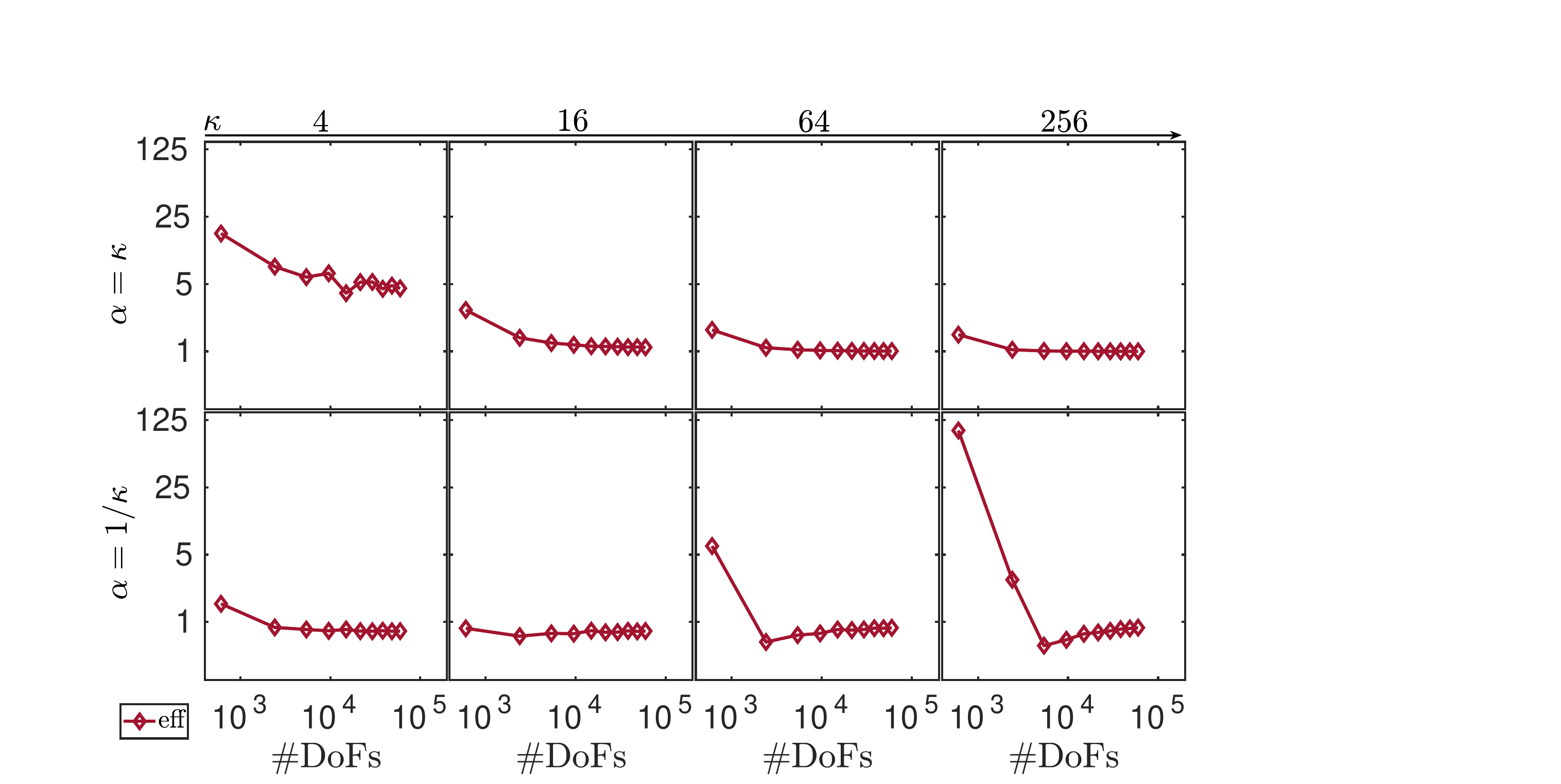}
    \caption{Example 2. Curves of effectivity index $\textnormal{eff}$ for the first eigenvalue $\lambda_1$ in the Voronoi mesh under uniform refinement with a variation of the stabilisation parameter $\alpha$.}\label{eff_stab}
\end{figure}
\subsection{Example 3: Adaptivity in 2D.}
In this test, we study the vibration problem \eqref{MPr} defined in the classical L-shape domain given by $\Omega=(0,1)^2 \setminus (1/2,1)^2$ with clamped boundary conditions \eqref{CP}. The domain is discretised with a Voronoi-type mesh (see Figure~\ref{fig:L}). It is well known that the first four eigenvalues are given by $\lambda_1 \approx  6704.2982$, $\lambda_2 \approx  11055.5189$, $\lambda_3 \approx  14907.0816$, and $\lambda_4 \approx 26157.9673$ (see e.g. \cite{MoRo2009, MRV2018}). 
\begin{figure}[h!]
    \centering
    \subfigure[$\lambda_1$]{
            \centering
            \includegraphics[width=0.225\textwidth,trim={1.9cm 0.cm 2.1cm 1.cm},clip]{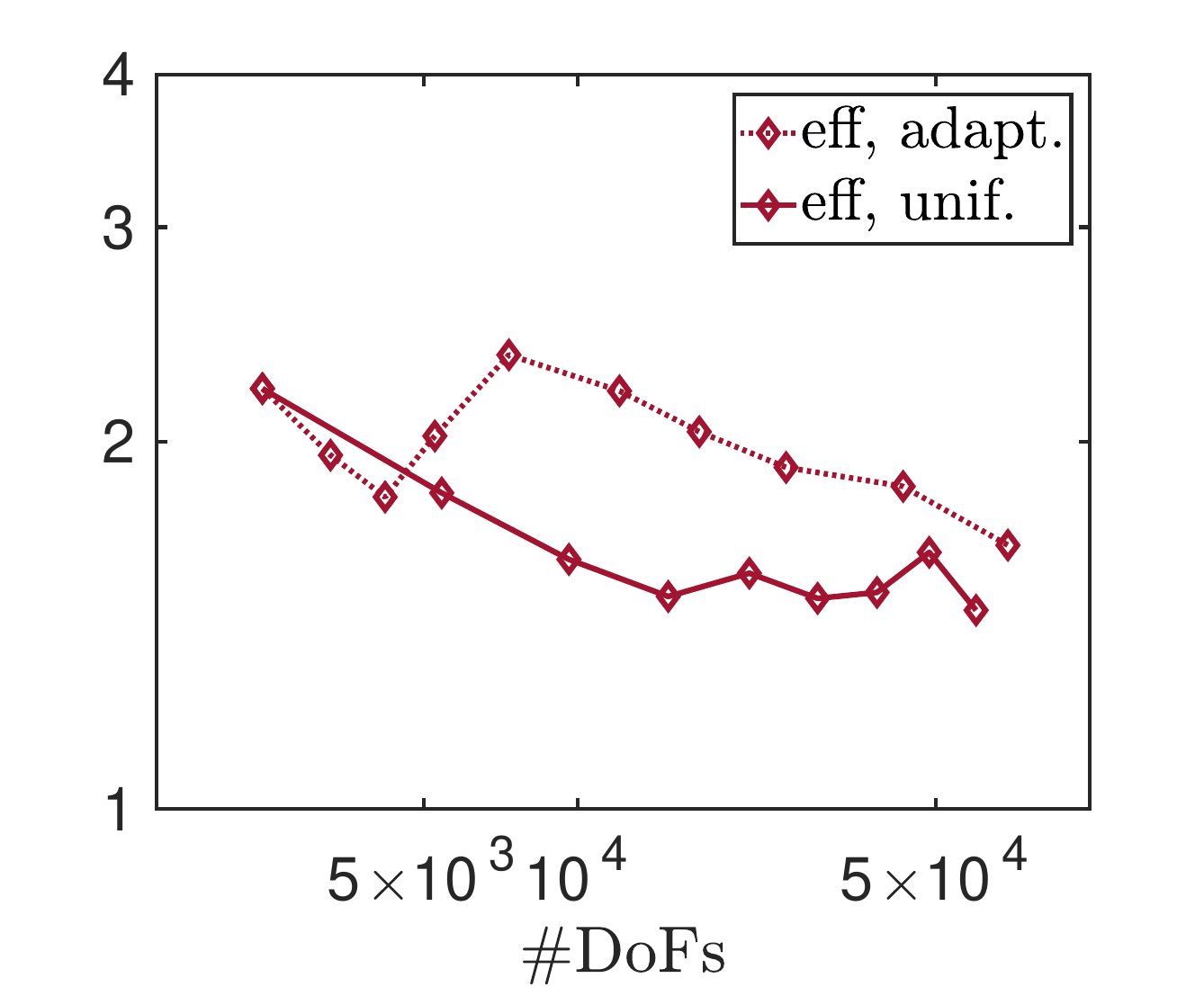}
    }
    \subfigure[$\lambda_2$]{
            \centering
            \includegraphics[width=0.225\textwidth,trim={1.9cm 0.cm 2.1cm 1.cm},clip]{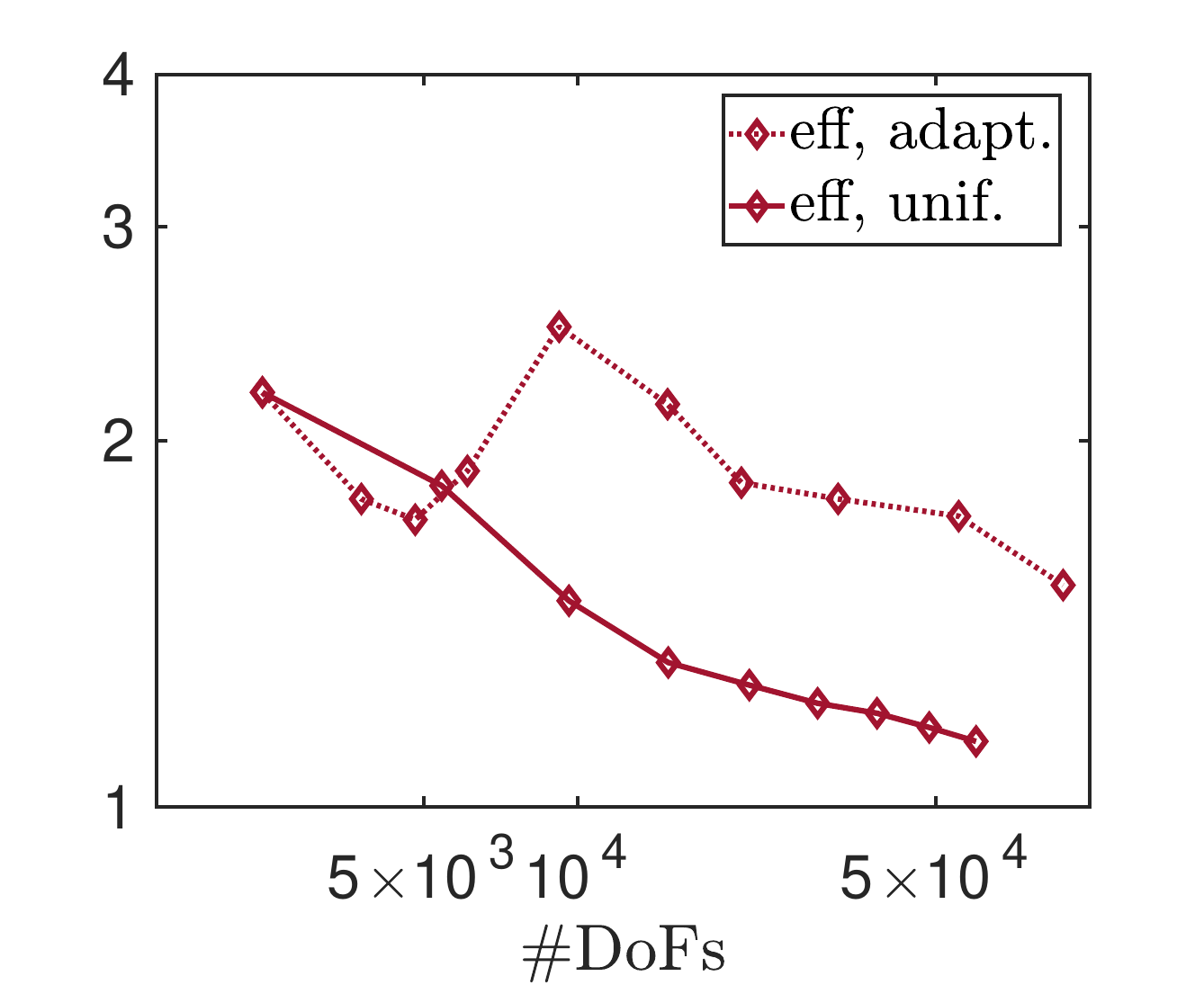}
    }
    \subfigure[$\lambda_3$]{
            \centering
            \includegraphics[width=0.225\textwidth,trim={1.9cm 0.cm 2.1cm 1.cm},clip]{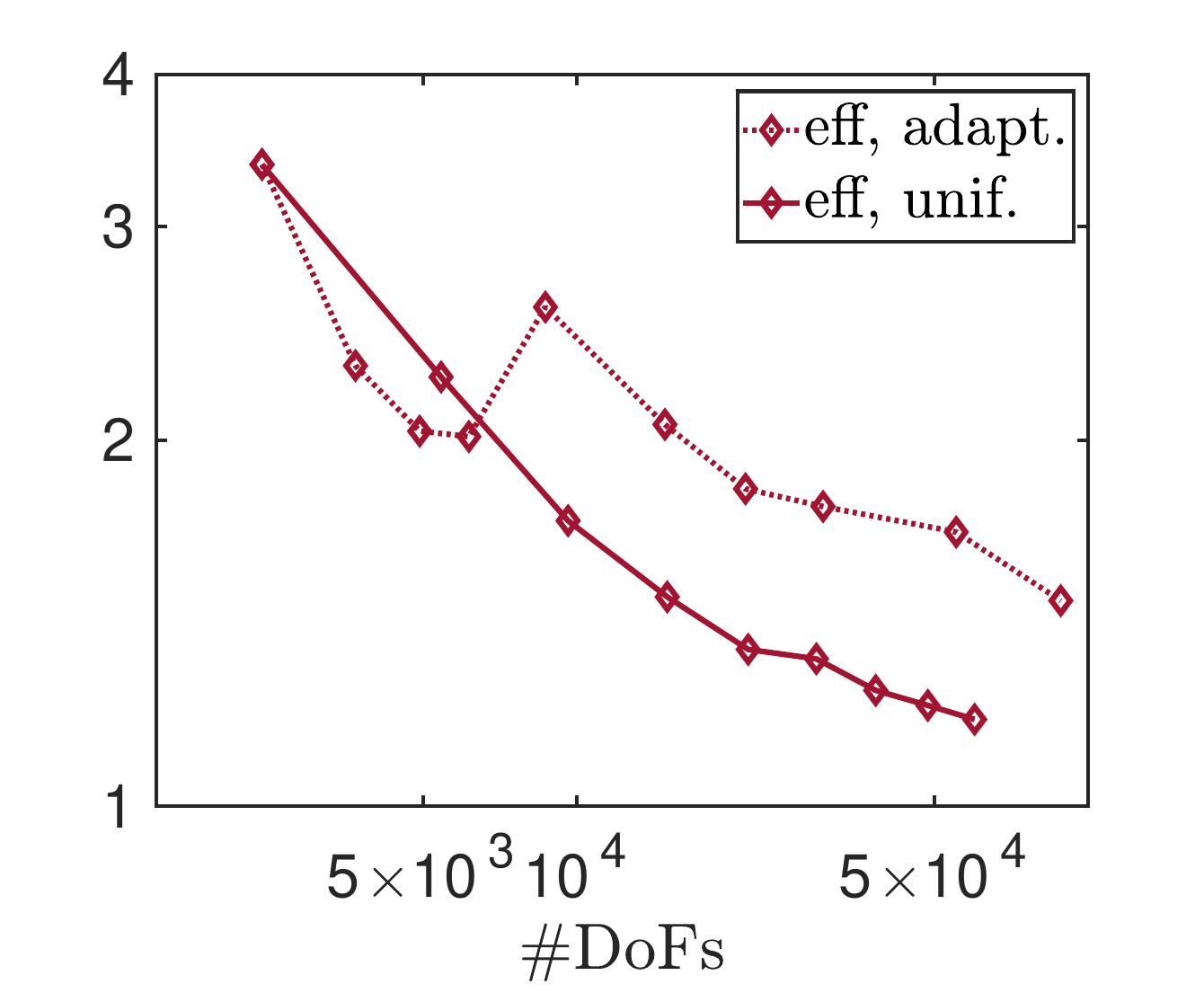}
    }
    \subfigure[$\lambda_4$]{
            \centering
            \includegraphics[width=0.225\textwidth,trim={1.9cm 0.cm 2.1cm 1.cm},clip]{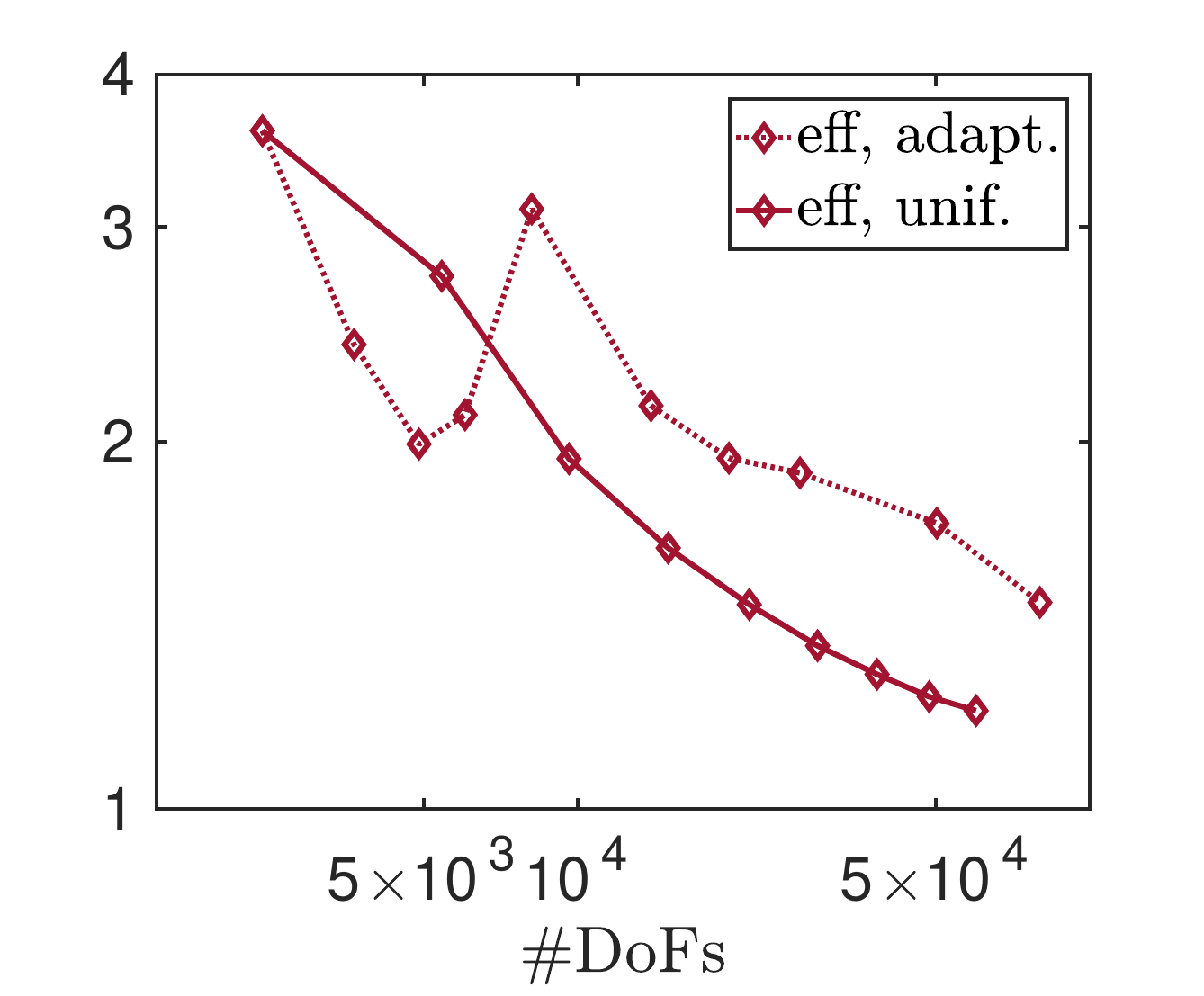}
    }
    \caption{Example 3. Curves of effectivity index $\textnormal{eff}$ for the first four eigenvalues $\lambda_i$ on the L-shaped domain under adaptive and uniform refinement.}
    \label{eff_L}
\end{figure}

Figure~\ref{convergence_L} reports the convergence history of the method for the error $|\lambda_{i,h}-\lambda_i|$ ($i\in \{1,2,3,4\}$), and the estimator $\eta$ under adaptive and uniform refinement. We observe that the adaptive refinement outperforms the uniform refinement in the presence of a singularity in the solution, expected by the re-entry corner of the L-shaped domain. Moreover, the effectivity index $\textnormal{eff}$ remains bounded between 1 and 4 (see Figure~\ref{eff_L}) confirming the efficiency and reliability of the method. Finally, %note that for each eigenvalue $\lambda_i$, the mesh adaptation changes to capture the accurate eigenfunction $u_i$ as shown in Figure~\ref{solution_L}.
Figure~\ref{solution_L} shows snapshots of the eigenfunctions $u_i$ after 9 adaptive refinement steps driven by $\eta$.
\begin{figure}[h!]
    \centering
    \subfigure[$\lambda_1$.]{\includegraphics[width=0.49\textwidth,trim={8cm 3cm 3.3cm 3.6cm},clip]{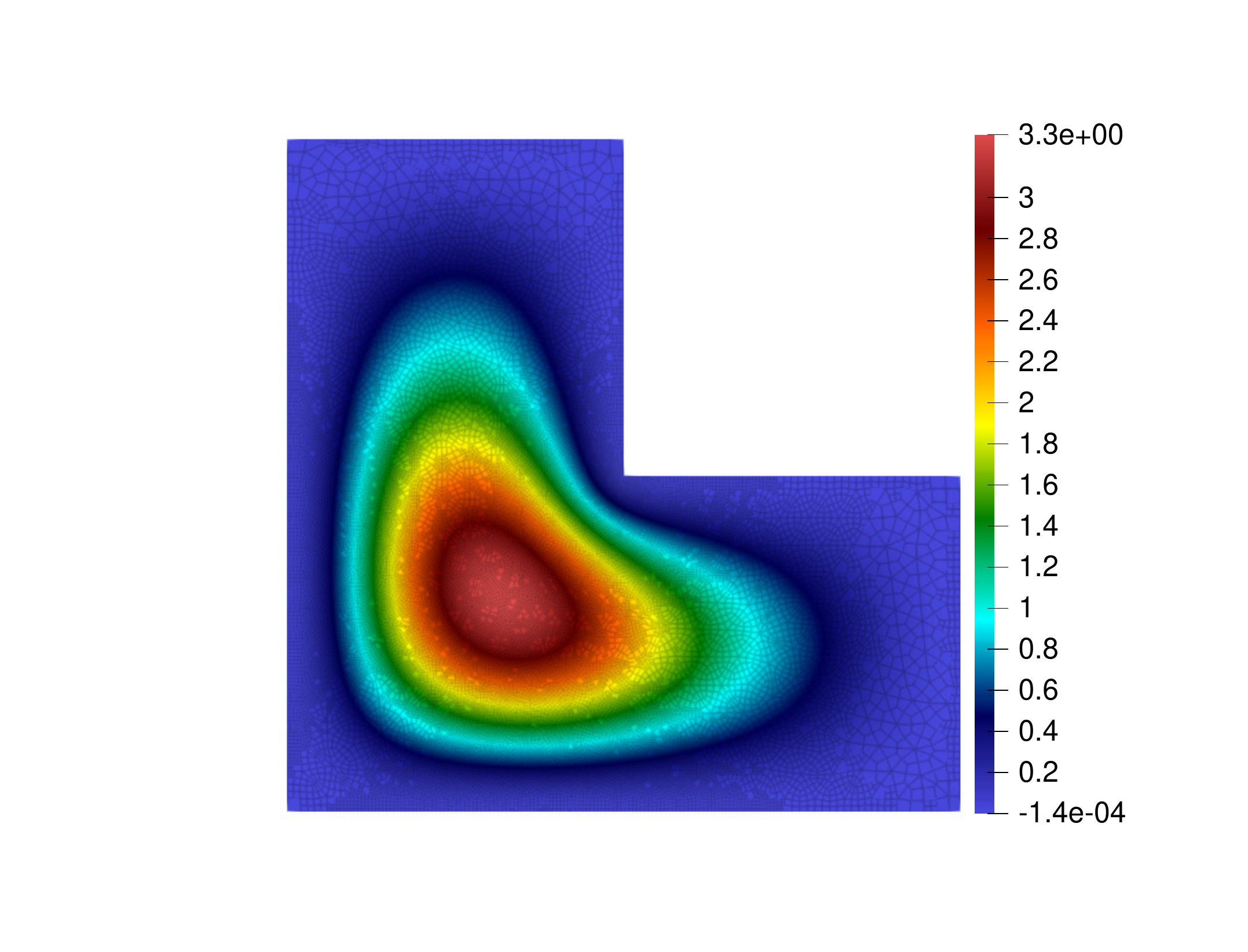}}
    \subfigure[$\lambda_2$.]{\includegraphics[width=0.49\textwidth,trim={8cm 3cm 3.3cm 3.6cm},clip]{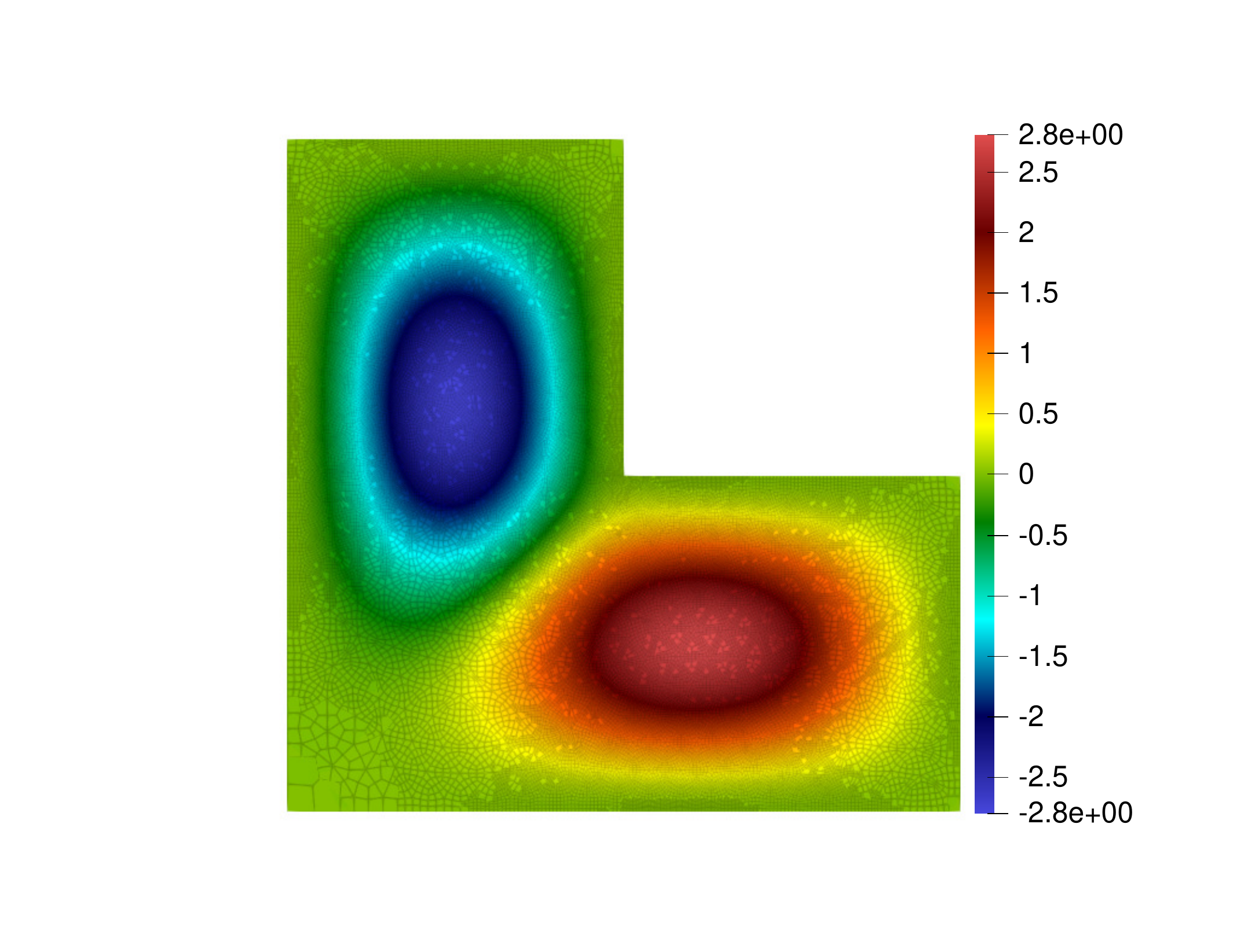}} 
    \subfigure[$\lambda_3$.]{\includegraphics[width=0.49\textwidth,trim={8cm 3cm 3.3cm 3.6cm},clip]{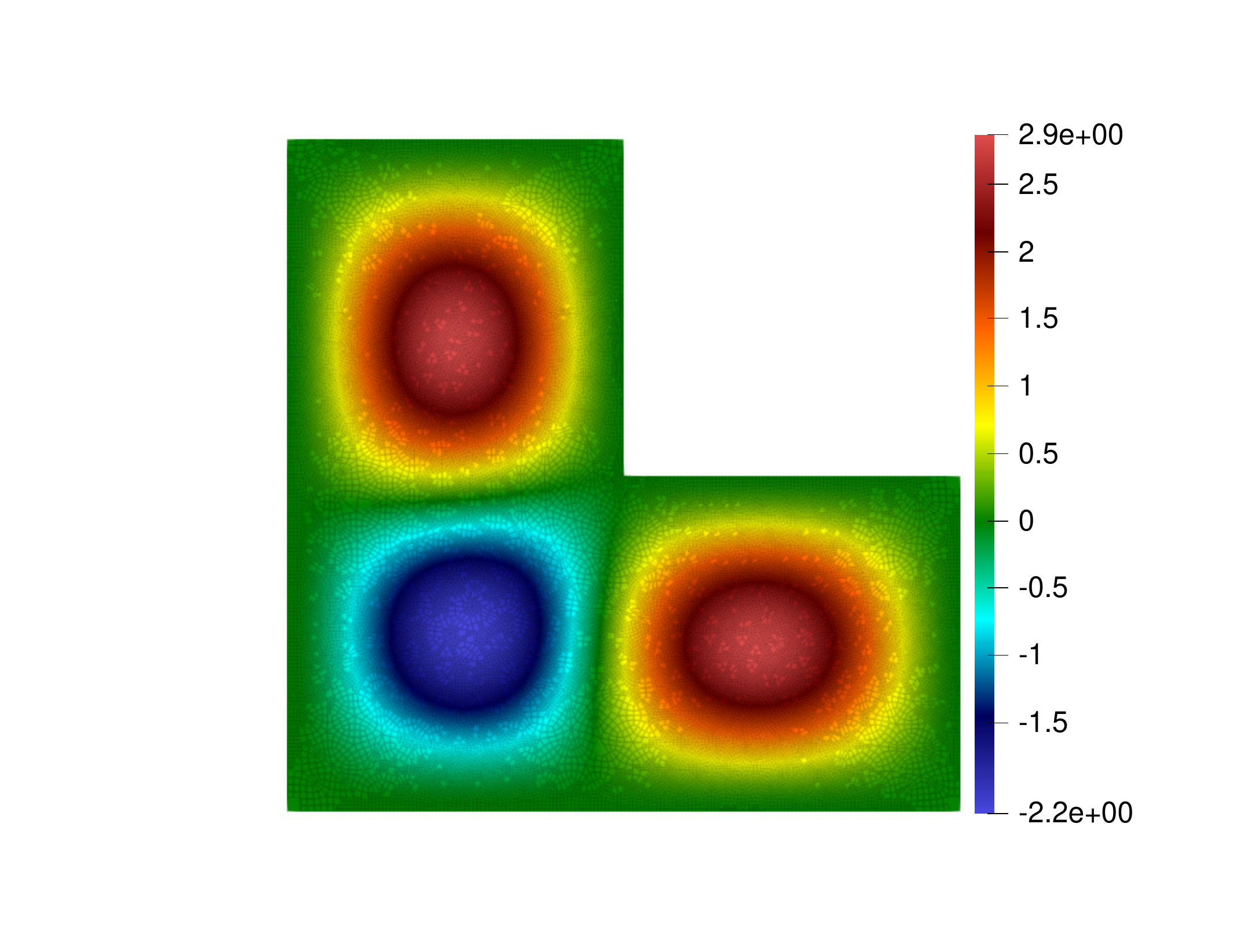}}
    \subfigure[$\lambda_4$.]{\includegraphics[width=0.49\textwidth,trim={8cm 3cm 3.3cm 3.6cm},clip]{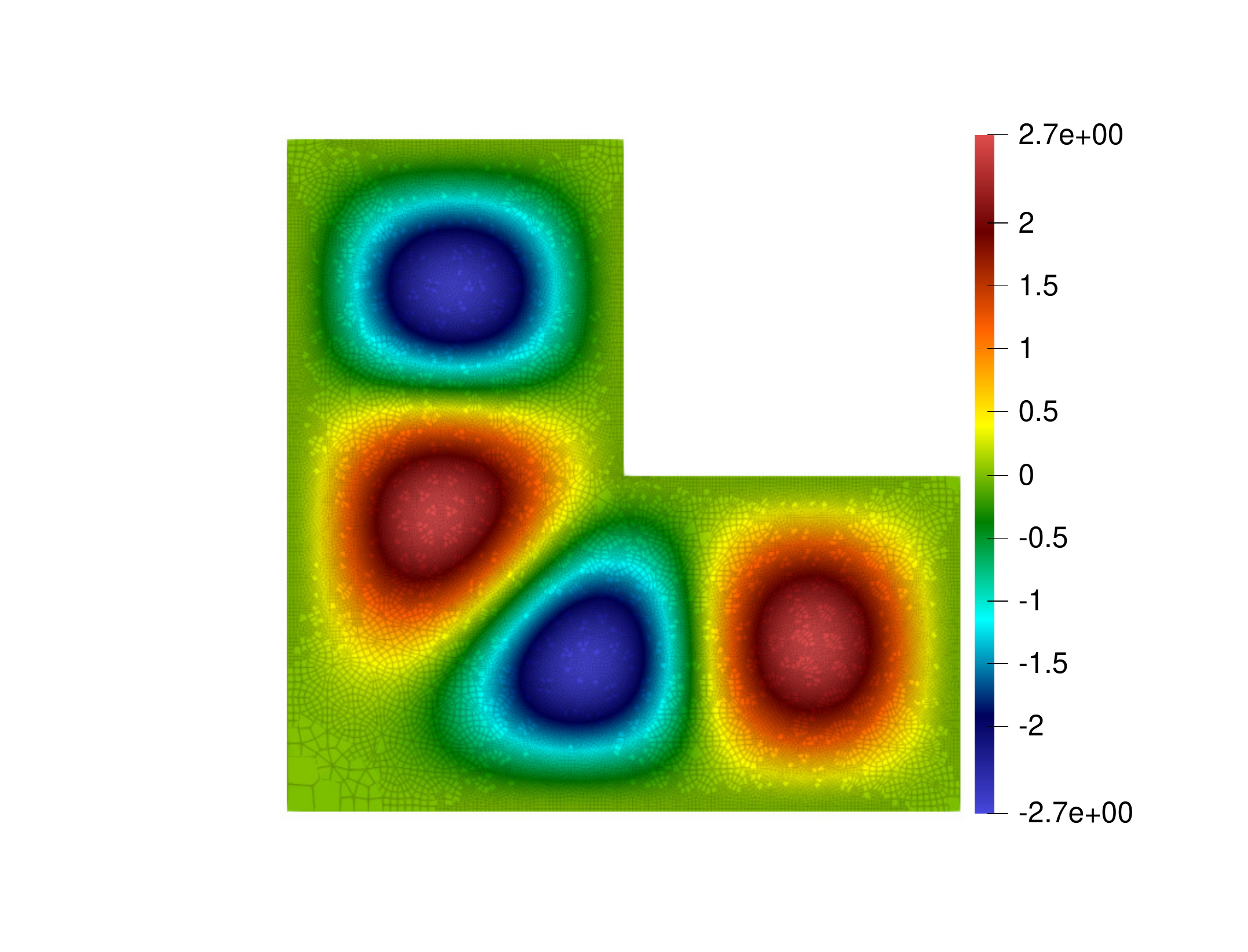}}
    \caption{Example 3. Snapshots of the eigenfunction $u_i$ in the L-shaped mesh for distinct eigenvalues $\lambda_i$ after 9 refinement steps driven by $\eta$.}\label{solution_L}
\end{figure}
\begin{figure}[h!]
    \centering
    \subfigure[$\lambda_1$]{
            \centering
            \includegraphics[width=0.48\textwidth,trim={2.85cm 0.cm 3.5cm 0.75cm},clip]{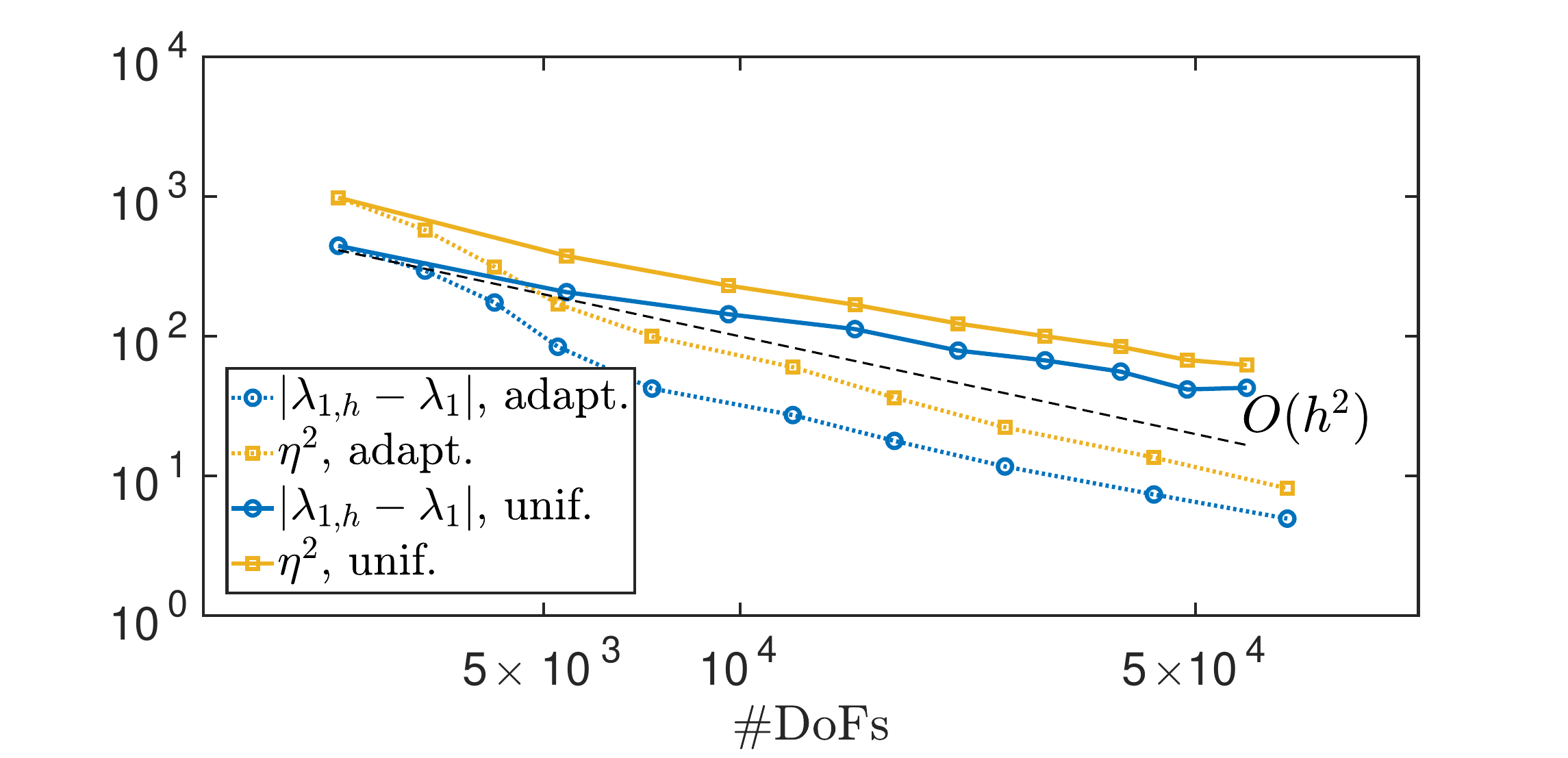}
    }
    \subfigure[$\lambda_2$]{
            \centering
            \includegraphics[width=0.48\textwidth,trim={2.85cm 0.cm 3.5cm 0.75cm},clip]{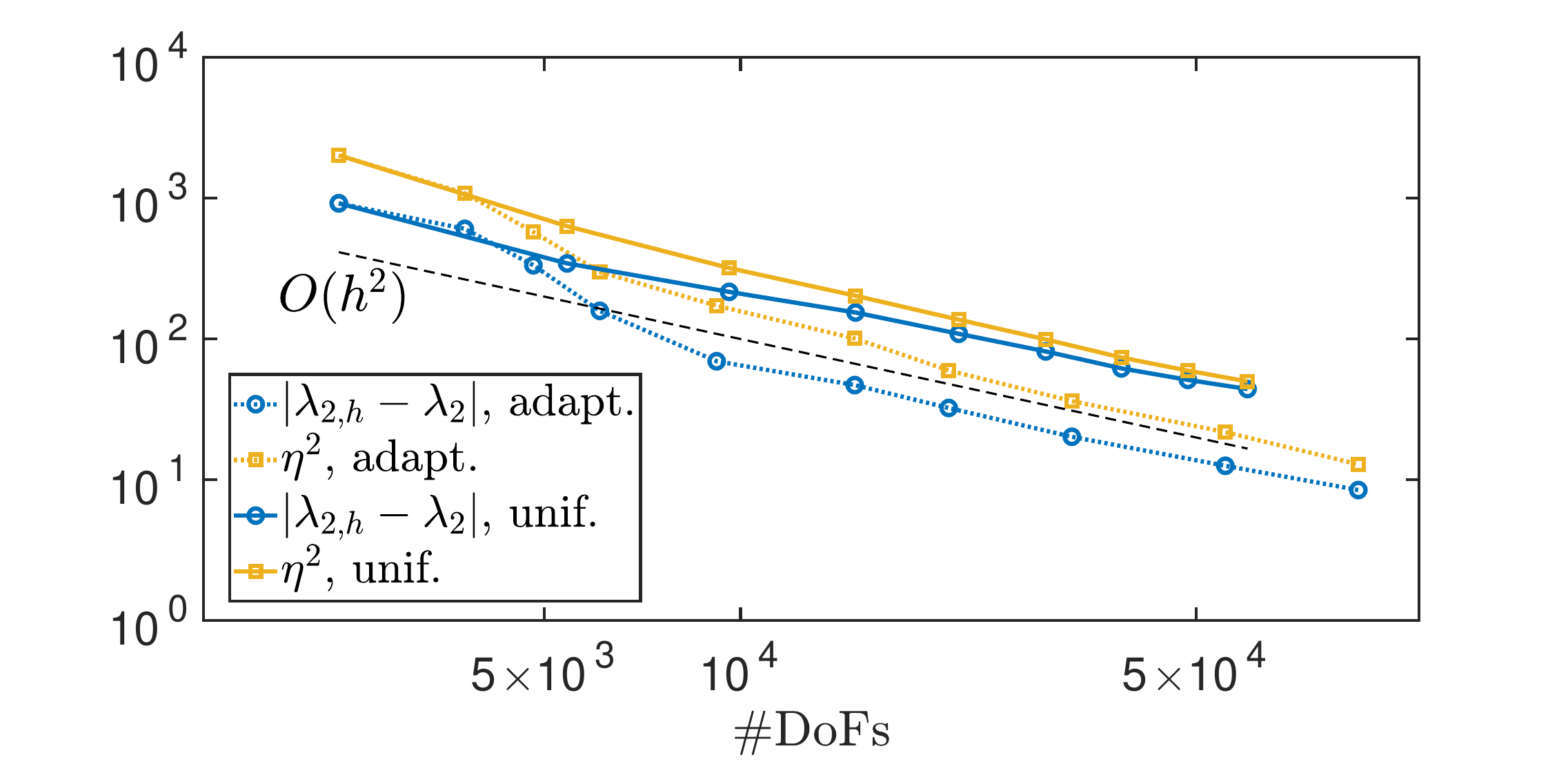}
    }
    \subfigure[$\lambda_3$]{
            \centering
            \includegraphics[width=0.48\textwidth,trim={2.85cm 0.cm 3.5cm 0.75cm},clip]{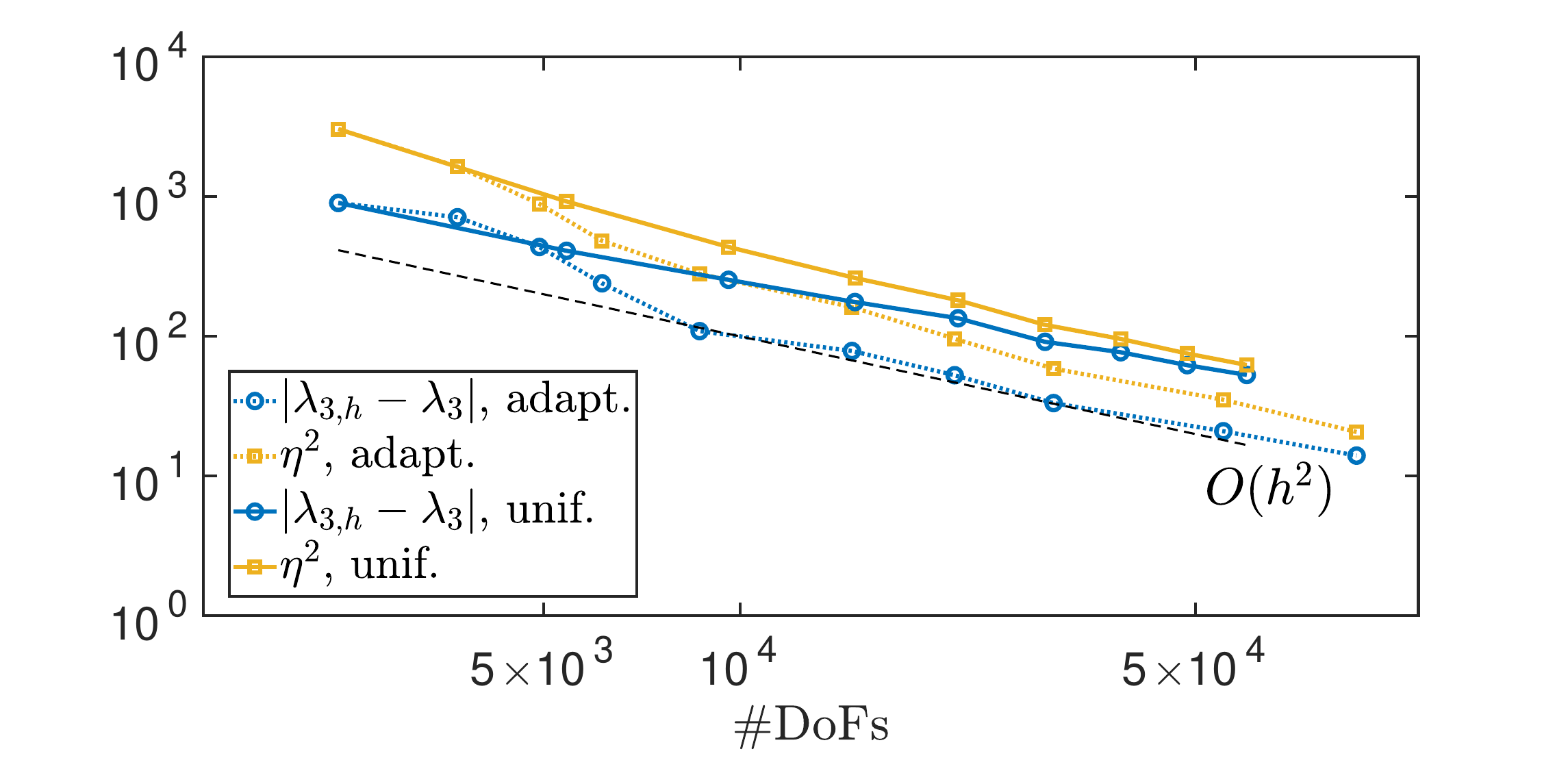}
    }
    \subfigure[$\lambda_4$]{
            \centering
            \includegraphics[width=0.48\textwidth,trim={2.85cm 0.cm 3.5cm 0.75cm},clip]{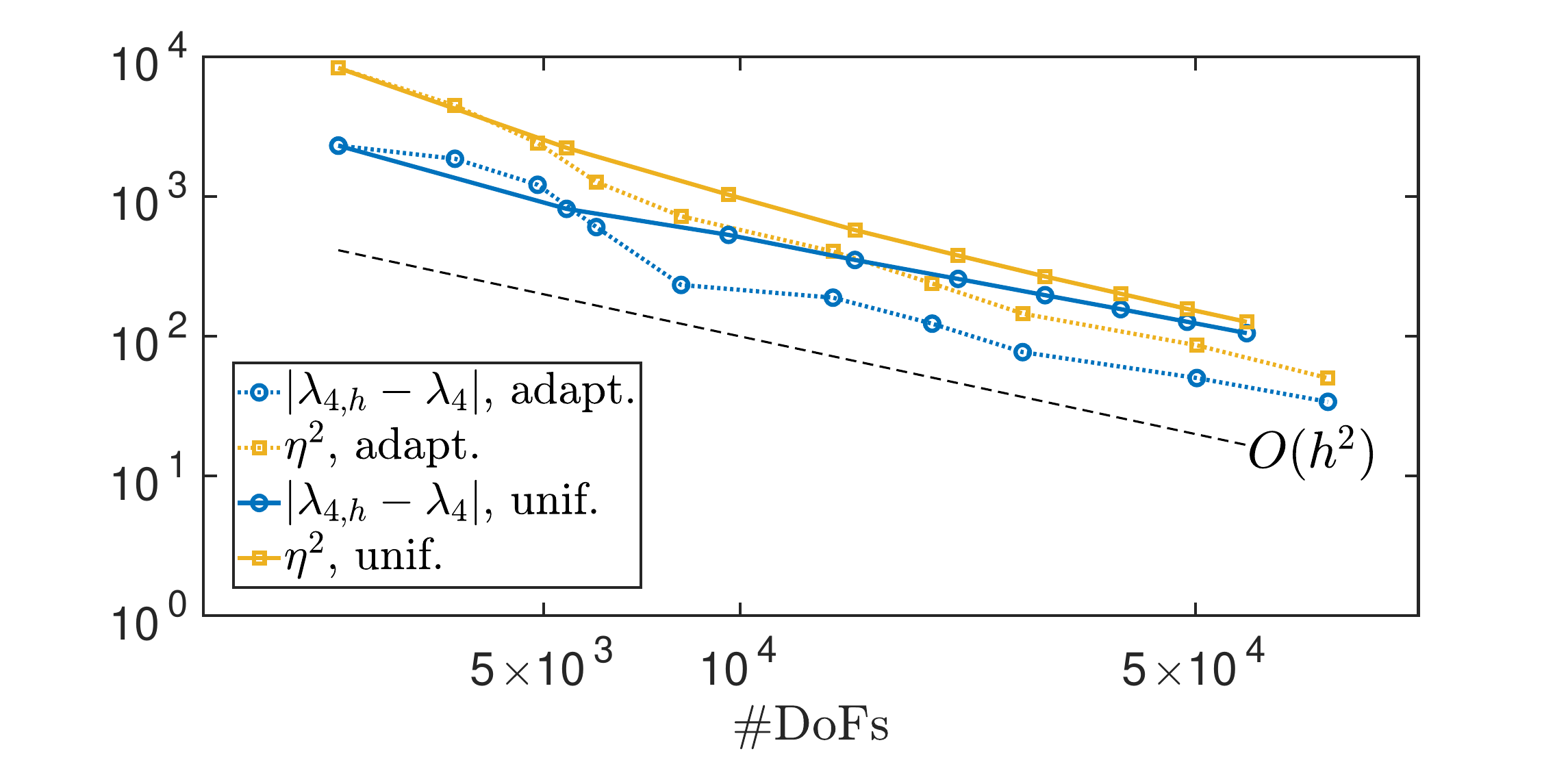}
    }
    \caption{Example 3. Curves of error $|\lambda_{i,h}-\lambda_i|$ and global error estimator $\eta$ for the first four eigenvalues $\lambda_i$ on the L-shaped domain under adaptive and uniform refinement.}
    \label{convergence_L}
\end{figure}

\subsection{Example 4: Vibration on an aircraft wing.} 
As an application, we consider \eqref{MPr} defined in a 2D simplification of an aircraft wing embedded in the unit square whose geometry is shown in Figure~\ref{fig:wingPlane}, the clamped boundary conditions \eqref{CP} are considered only on the part of the wing that is attached to the aircraft fuselage ($y=0$).
\begin{figure}[h!]
    \centering
    \includegraphics[width=.85\textwidth,trim={3.25cm 0.cm 4.cm 1.cm},clip]{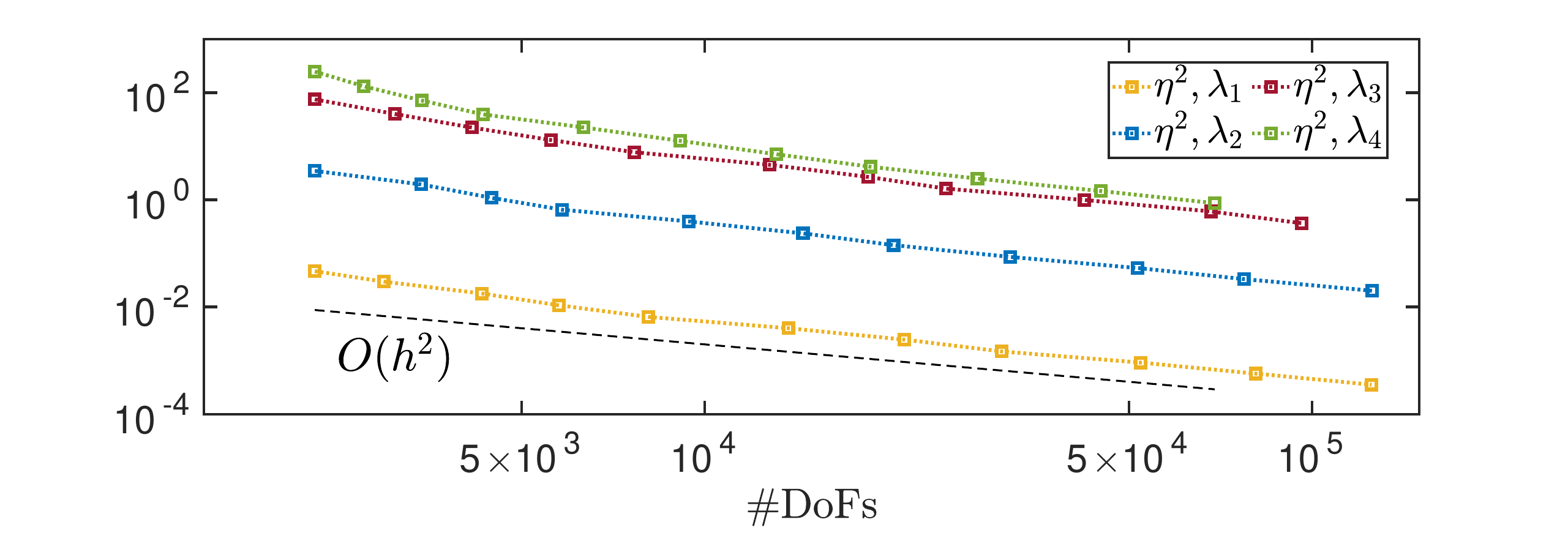}
    \caption{Example 4. Behaviour of the global error estimator $\eta$ for the first four eigenvalues $\lambda_i$ on the aircraft wing.}
    \label{convergence_wing}
\end{figure}

The initial mesh considered on the adaptive refinement routine is constructed by applying four iterations of polymesher's uniform refinement to the mesh shown in Figure~\ref{fig:wingPlane}. To the best of the author's knowledge, no previous approximations of the first four eigenvalues have been reported for this geometry. Our scheme computes these eigenvalues after nine adaptive refinements: $\lambda_1 \approx 25.47102714143987$, $\lambda_2 \approx 453.1962170527747$, $\lambda_3 \approx 2478.31700898077$, $\lambda_4 \approx 4535.233536949408$.
\begin{figure}[h!]
    \centering
    \subfigure[$\lambda_1$.]{\includegraphics[width=0.49\textwidth,trim={11.5cm 3cm 6.45cm 3.6cm},clip]{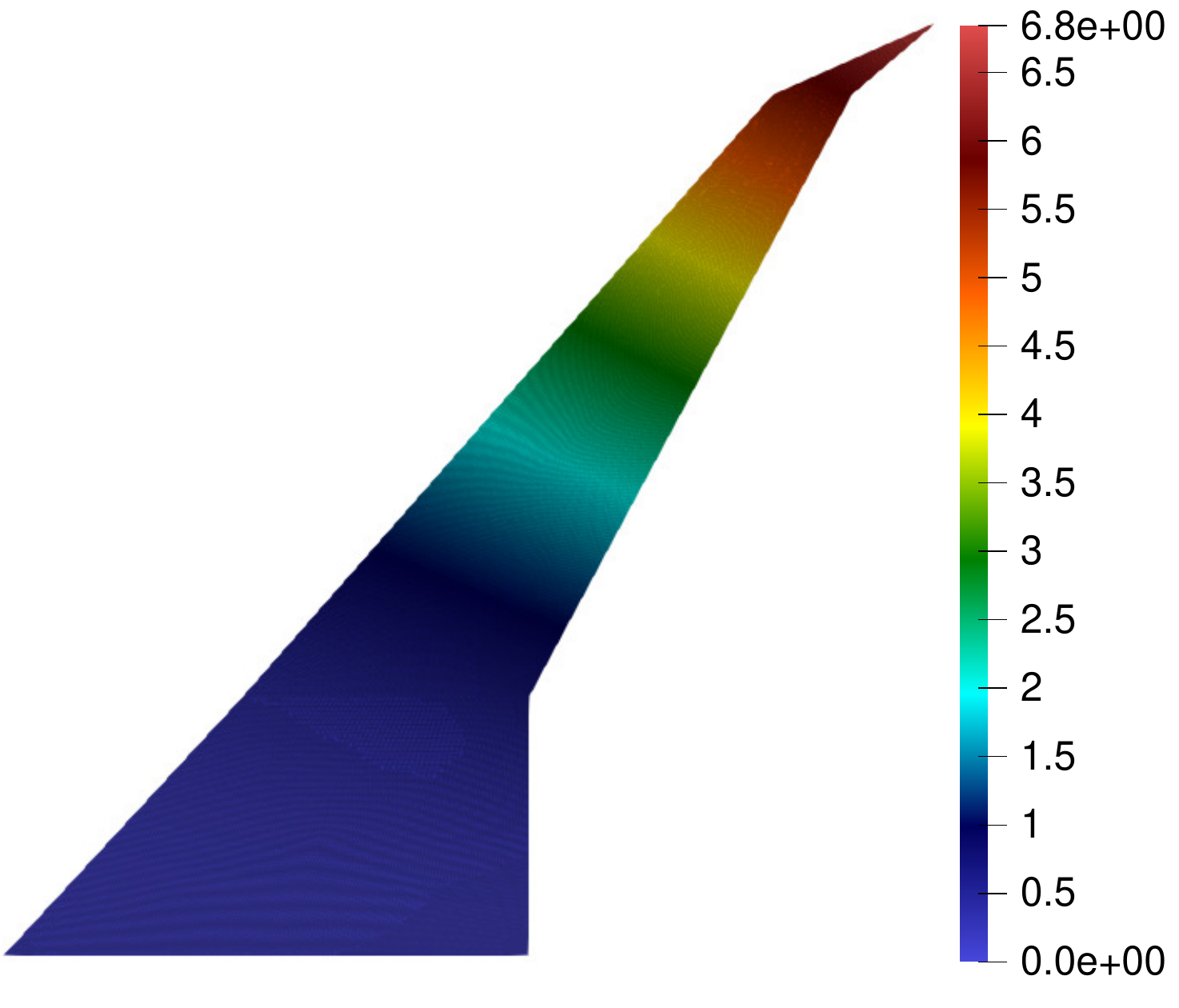}}
    \subfigure[$\lambda_2$.]{\includegraphics[width=0.49\textwidth,trim={11.5cm 3cm 6.45cm 3.6cm},clip]{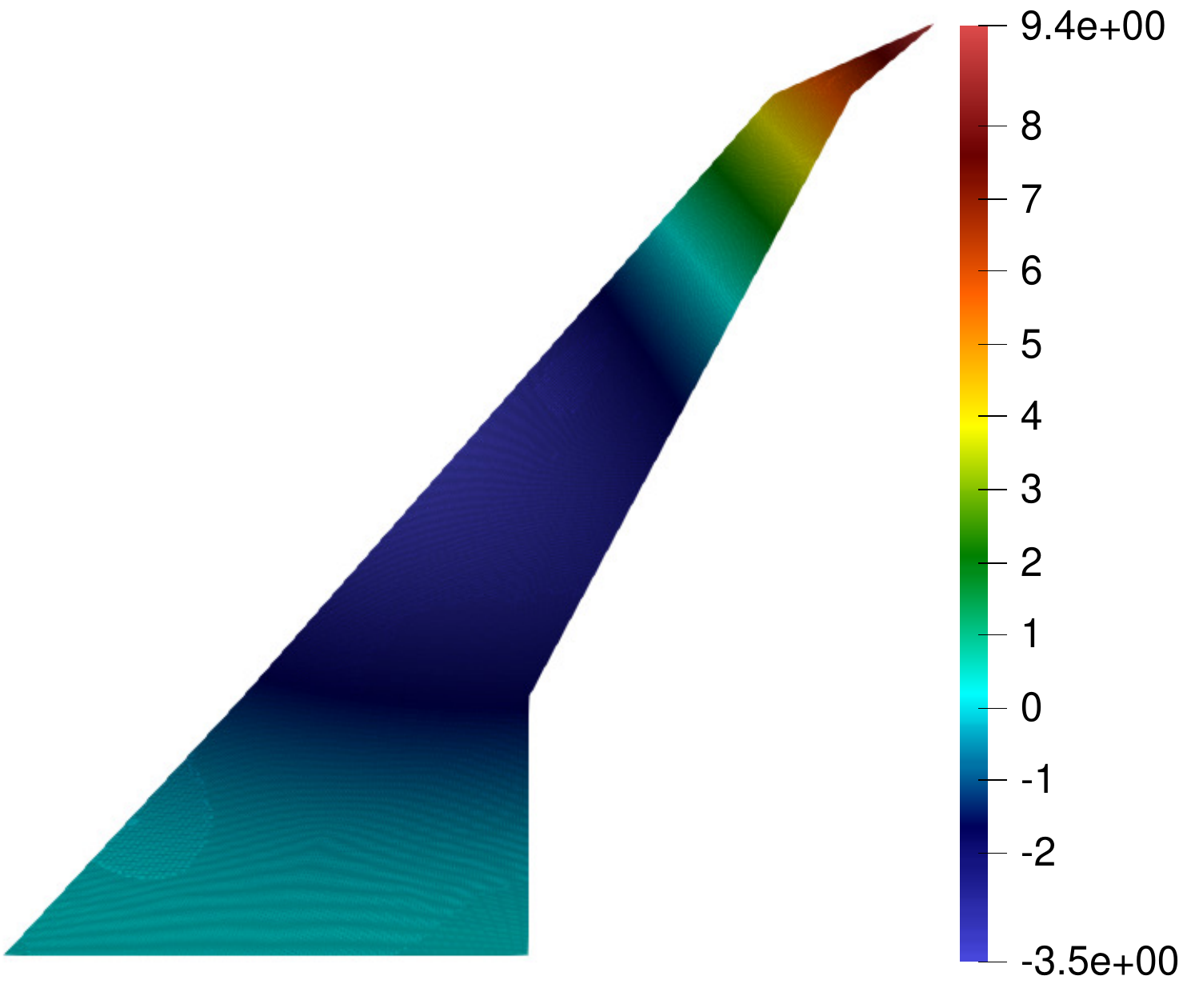}} 
    \subfigure[$\lambda_3$.]{\includegraphics[width=0.49\textwidth,trim={11.5cm 3cm 6.45cm 3.6cm},clip]{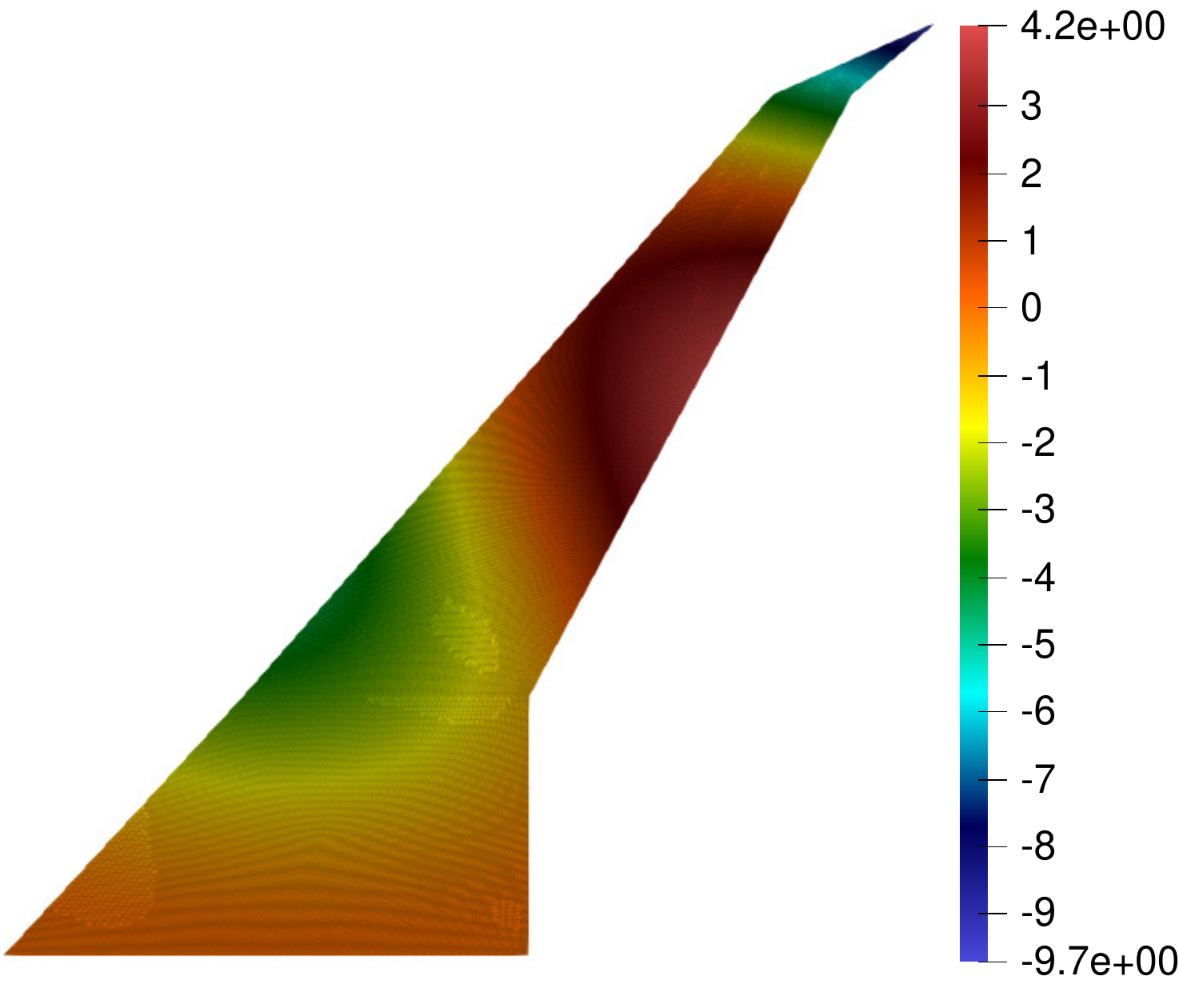}}
    \subfigure[$\lambda_4$.]{\includegraphics[width=0.49\textwidth,trim={11.5cm 3cm 6.45cm 3.6cm},clip]{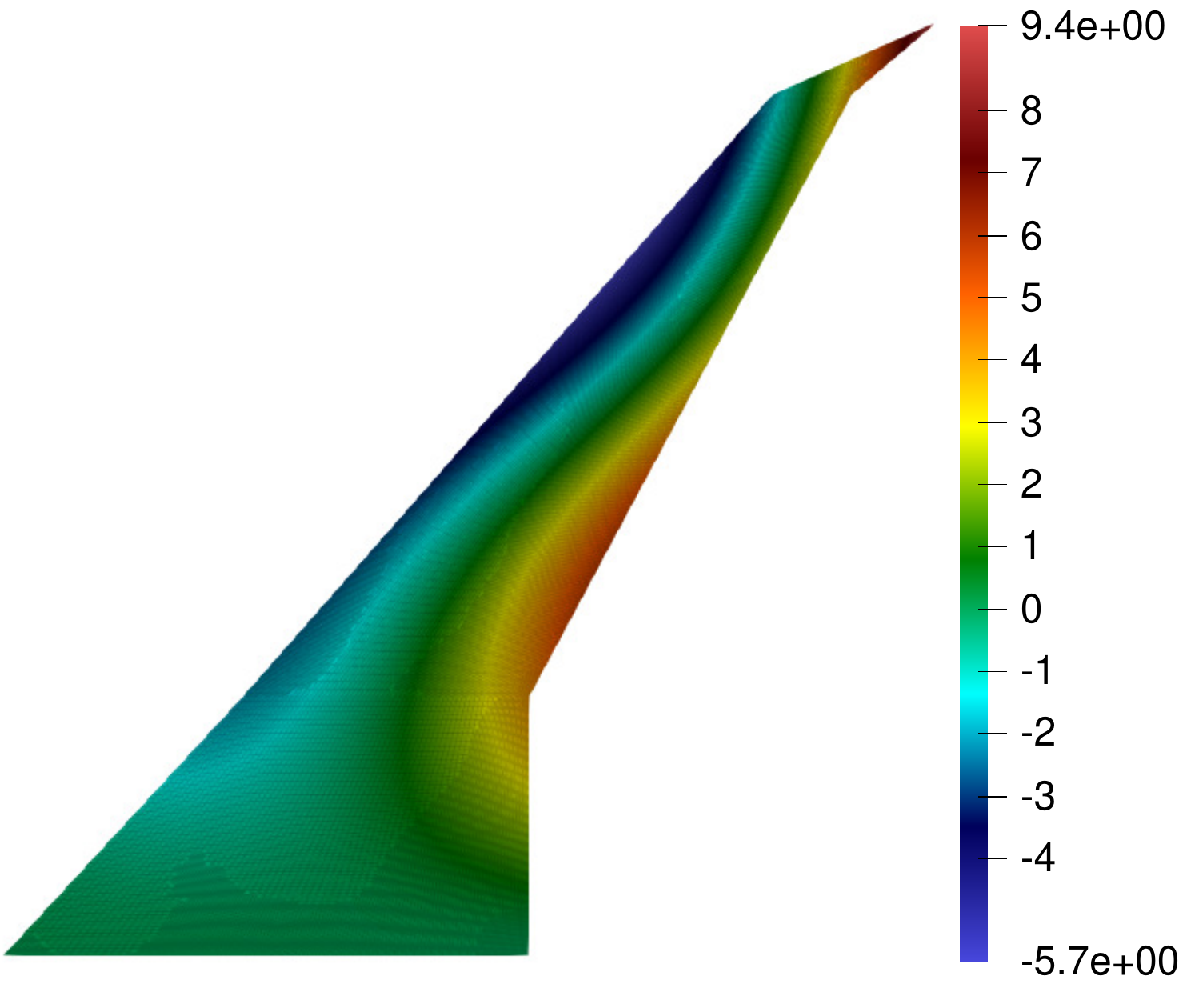}}
    \caption{Example 4. Snapshots of the eigenfunction $u_i$ in the aircraft wing mesh for distinct eigenvalues $\lambda_i$ after 9 refinement steps driven by $\eta$.}\label{solution_wing}
\end{figure}

The computable global error estimator $\eta^2$ is reported in Figure~\ref{convergence_wing}, showing the optimal behaviour of $O(h^2)$ given by Corollary~\ref{global_eff}. Since exact values of $\lambda_i$, $i\in \{1,2,3,4\}$ are not provided, we can not compute the error $|\lambda_{i,h}-\lambda_i|$. However, previous examples have already confirmed that the estimator gives an upper bound of this error. The associated eigenfunctions $u_i$ are shown in Figure~\ref{solution_wing} after 9 refinement steps driven by $\eta$.

\subsection{Example 5: Behaviour under uniform refinement: 3D case.}\label{ex:5}
This numerical experiment extends Subsection~\ref{ex:2} to the three-dimensional case. The vibration problem \eqref{MPr} with simply supported boundary conditions \eqref{SSP} is now defined on the unit cube domain $\Omega = (0,1)^3$ given in Figure~\ref{fig:cube}. This boundary condition allows us to perform the test knowing exact values of $\lambda_1 = 9\pi^4\approx 876.6818$ (see \cite{Hu2015}). Furthermore, a sensitivity analysis on the influence of stabilisation terms is available in \cite{DV_camwa2022}. Based on those findings, we choose $\alpha_\Delta = 1$ and $\alpha_0 = 10^{-4}$.

Due to high computational effort, we were not able to capture the $O(h^2)$ trend. Then we decided to split the estimator contributions. This analysis is reported in  Figure~\ref{convergence_3d}, where we observe that the global volume estimator $\Xi$ with convergence rate of $O(h^4)$ dominates over the global stabilisation estimator $S$ with lower convergence rate of $O(h^2)$. However, the trend shows that $S$ becomes dominant as the uniform refinement progresses while bounding the error $|\lambda_{1,h}-\lambda_1|$. This confirms the reliability of the scheme in the 3D case with the optimal convergence rate. 
\begin{figure}[h!]
    \centering
    \includegraphics[width=.85\textwidth,trim={3.5cm 0.cm 4.cm 0.7cm},clip]{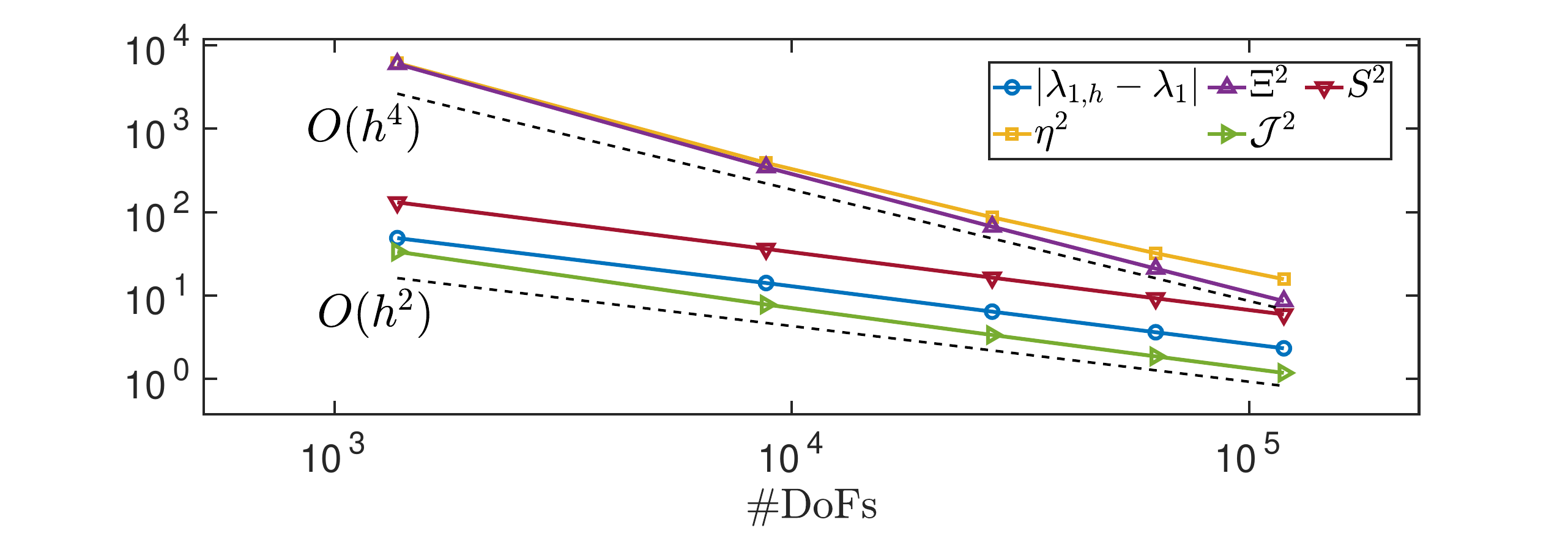}  
    \caption{Example 5. Curves of error $|\lambda_{1,h}-\lambda_1|$, global error $\eta$,  volume $\Xi$, jump $\mathcal{J}$, and stabilisation $S$ estimators for the first eigenvalue $\lambda_1$ in the unit cube under uniform refinement.}\label{convergence_3d}
\end{figure}

\subsection{Example 6: Adaptivity in 3D.}
To illustrate the applicability of the method in 3D, we consider the vibration problem \eqref{MPr} with clamped boundary conditions \eqref{CP} in the Fichera cube domain $\Omega = (0,1)^3 \setminus (1/2,1)^3$ (see Figure~\ref{fig:Fichera}). We obtain an approximation for the first eigenvalue $\lambda_1 \approx 6657.574172648315$ after 13 adaptive refinement steps. The closest approximation of this value is given in \cite{DV_camwa2022}. Similarly to Subsection~\ref{ex:5}, we select $\alpha_\Delta = 1$ and $\alpha_0 = 10^{-4}$.

The curves for the error $|\lambda_{1,h}-\lambda_1|$ and global error estimator $\eta$ are shown in Figure~\ref{convergence_3d_adapt}. Note that the adaptive refinement outperforms the uniform refinement, recovering the optimal convergence rate of $O(h^2)$. Moreover, the adaptive refinement captures the singularity of the solution close to the re-entrant corner as shown in Figure~\ref{snapshots3d}.

\begin{figure}[h!]
    \centering
    \includegraphics[width=0.49\textwidth,trim={10.cm 2.2cm 6.25cm 2.cm},clip]{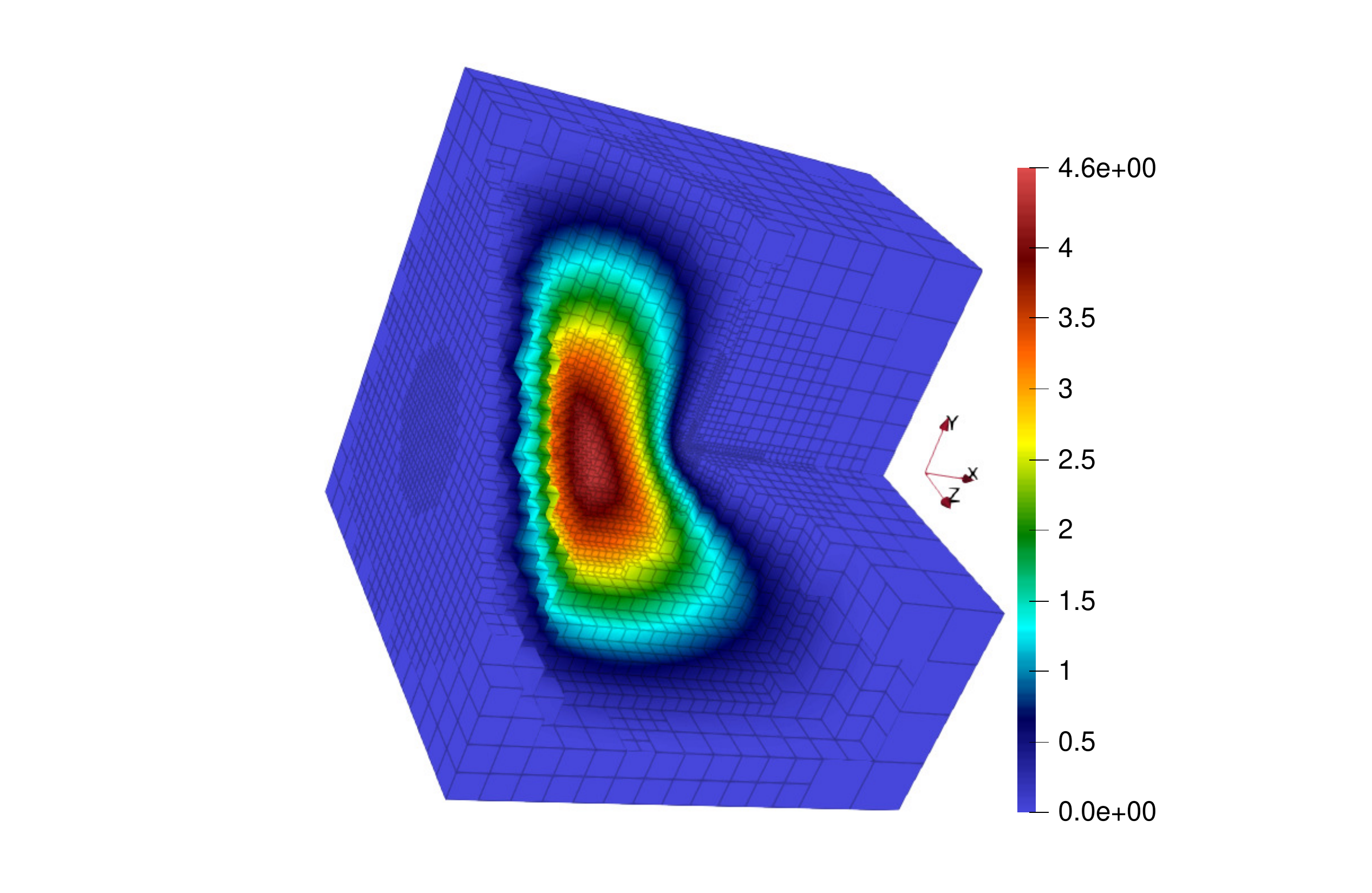}  
    \includegraphics[width=0.49\textwidth,trim={10.cm 2.2cm 6.25cm 2.cm},clip]{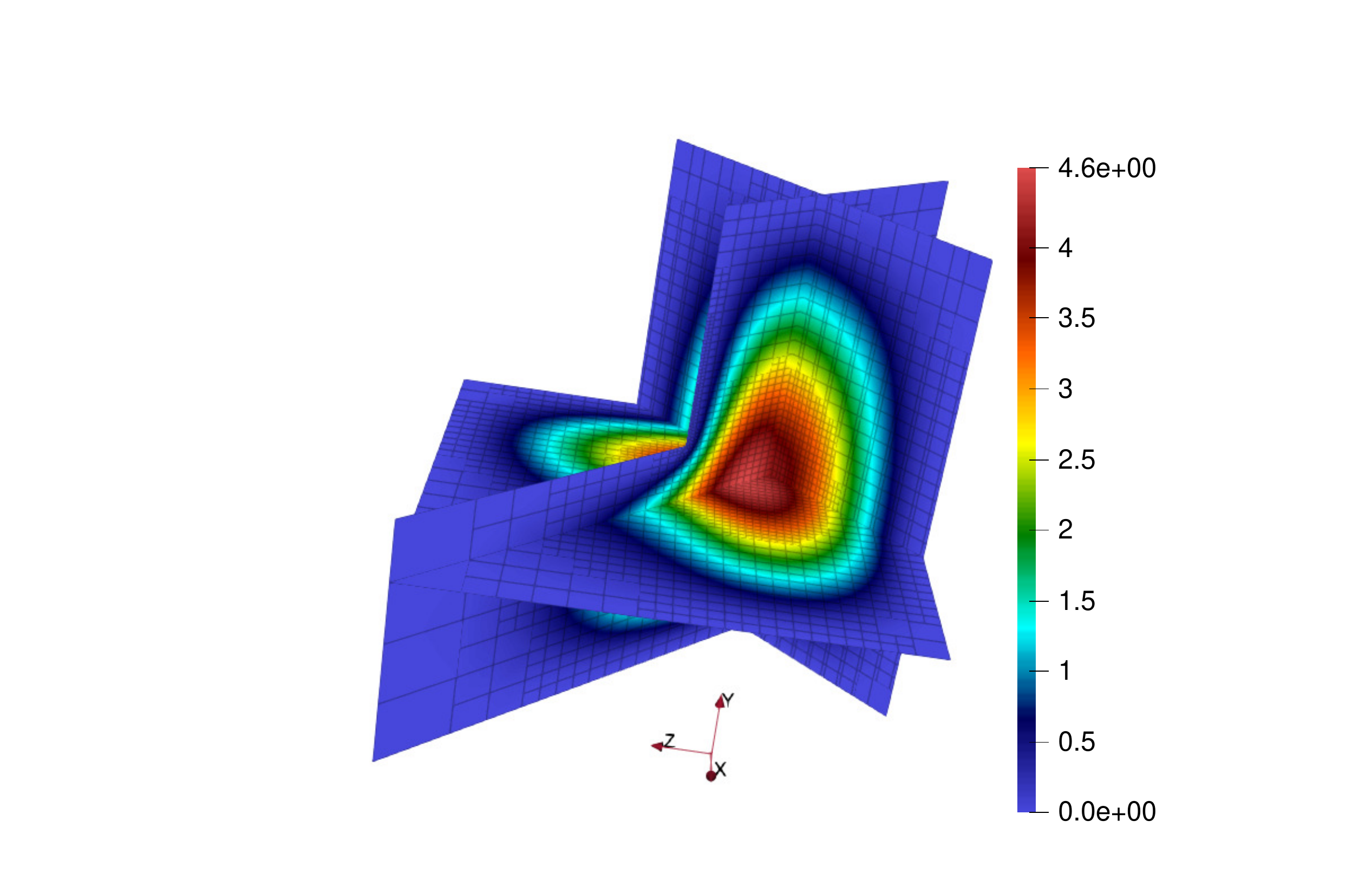}
    \caption{Example 6. Snapshots of the approximated eigenfunction $u_{1,h}$ in the Fichera cube after 14 adaptive refinement.}\label{snapshots3d}
\end{figure}
\begin{figure}[h!]
    \centering
    \includegraphics[width=0.45\textwidth,trim={0.cm 0.cm 2.2cm 1.5cm},clip]{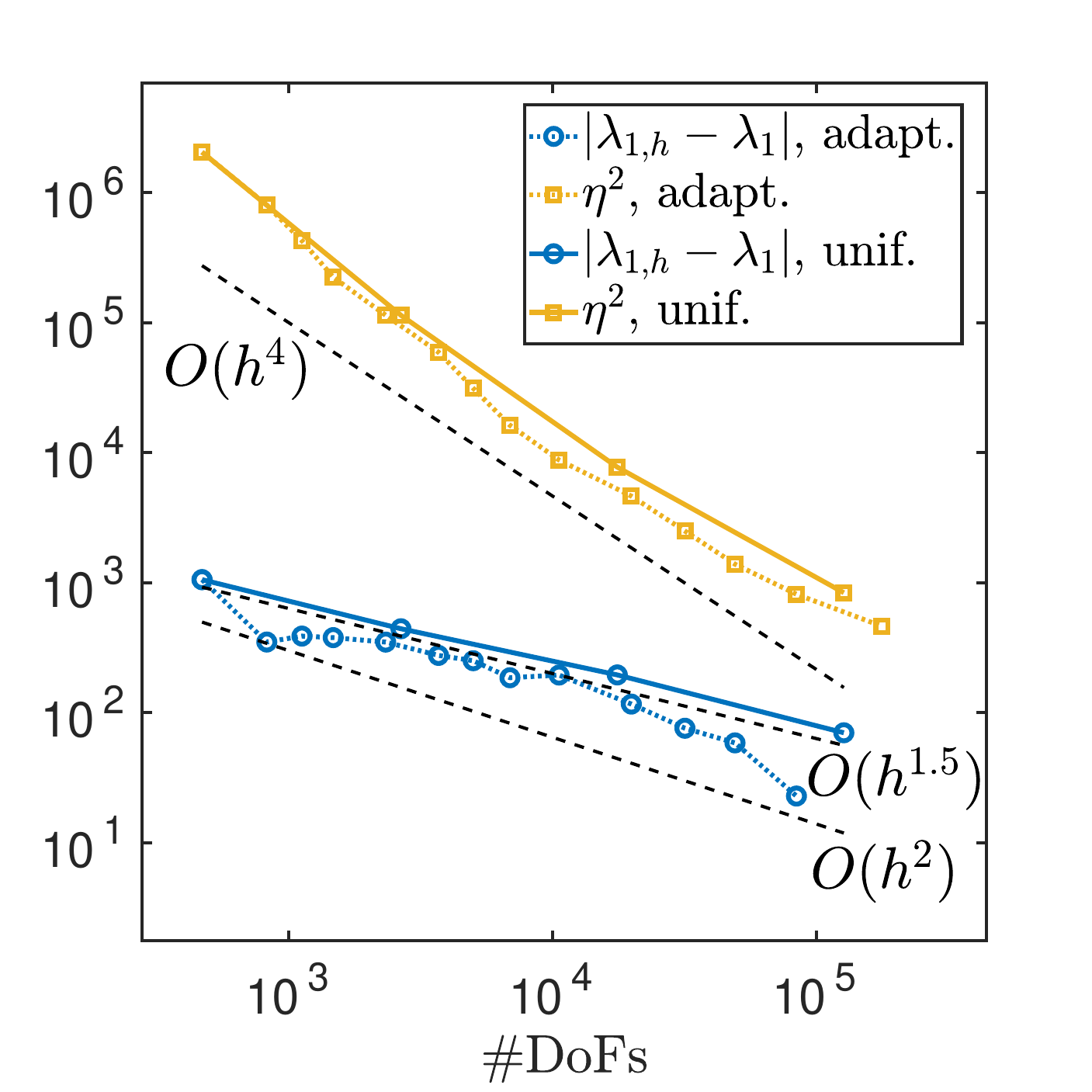}  
    \includegraphics[width=0.45\textwidth,trim={0.cm 0.cm 2.2cm 1.5cm},clip]{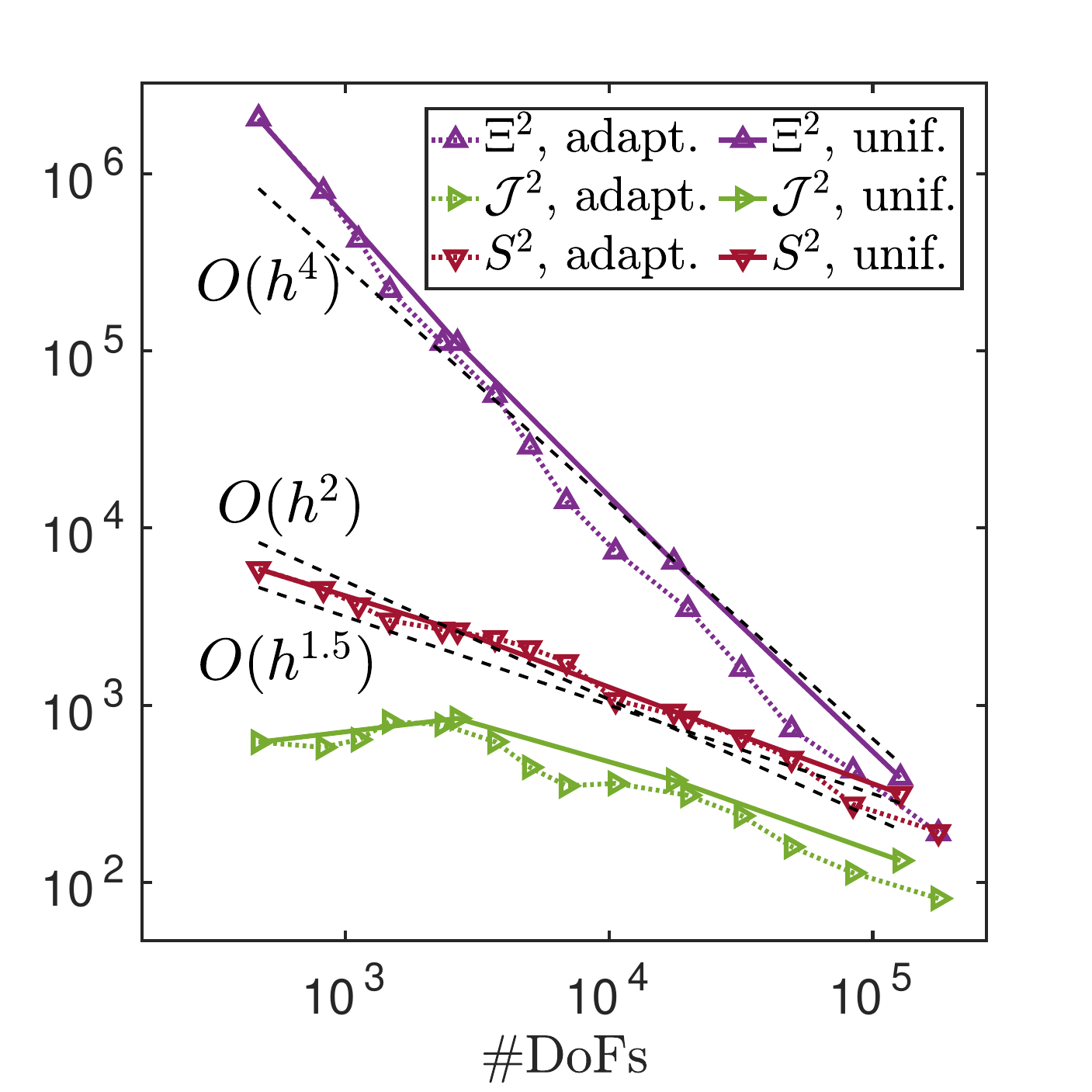} 
    \caption{Example 6. Curves of error $|\lambda_{1,h}-\lambda_1|$, global error $\eta$ (left),  volume $\Xi$, jump $\mathcal{J}$, and stabilisation $S$ (right) estimators for the first eigenvalue $\lambda_1$ in the Fichera cube under uniform and adaptive refinement.}\label{convergence_3d_adapt}
\end{figure}

%===============================================================================

%===============================================================================
%			Acknowledgements			
\subsection*{Acknowledgements} 	
Producto derivado del proyecto INV-CIAS-4167 financiado por la Universidad Militar Nueva Granada - Vigencia 2025. 
The first author was partially supported by the European Research Council project NEMESIS (Grant No. 101115663).
The second author has been partially supported by the Australian Research Council through the \textsc{Future Fellowship} grant FT220100496.
For the the third author this is a product derived from the project INV-CIAS-4167 funded by the  Universidad Militar Nueva Granada - Validity 2025.
We kindly thank Jesus Vellojin for the support provided in the PETSC solver implementation.

%\bibliographystyle{siam}
%\bibliography{bibliography}
% %---------------------------------------------------------------
%   AQUÍ GENERO LAS REFERENCIAS usando los archivos .bib actualizado con mathscinet
% %---------------------------------------------------------------
%\section*{References}
\bibliographystyle{abbrv} %siam, abbrv, ieeetr, unsrt, acm
\bibliography{references}
%%%%%%%\bibliographystyle{plain}
%-----------------------------------------------------------------------
\end{document}